\documentclass[11pt]{amsart}
\usepackage{amssymb,amsmath,amsthm,amsfonts,mathrsfs}
\usepackage[hmargin=3cm,vmargin=3.5cm]{geometry}
\usepackage[dvipsnames,table,xcdraw]{xcolor}
\usepackage[colorlinks = true,
            linkcolor = Fuchsia,
            urlcolor  = ForestGreen,
            citecolor = WildStrawberry,
            anchorcolor = blue]{hyperref}
\usepackage{graphicx,psfrag}
\usepackage{amscd}  
\usepackage[shellescape]{gmp}
\usepackage{stmaryrd}  
\usepackage[all,2cell]{xy} \UseAllTwocells \SilentMatrices
\usepackage{tikz-cd}
\usepackage[dvips]{epsfig}
\usepackage{MnSymbol}
\usepackage{comment}
\usepackage{enumitem}

\usetikzlibrary{arrows,automata}

\theoremstyle{definition}
\newtheorem{theorem}{Theorem}[section]

\newtheorem{remark}[theorem]{Remark}

\newtheorem{prop}[theorem]{Proposition}

\newtheorem{corollary}[theorem]{Corollary}
\newtheorem{problem}[theorem]{Problem}

\newtheorem{example}[theorem]{Example}

\let\oldtocsection=\tocsection
\let\oldtocsubsection=\tocsubsection
\renewcommand{\tocsection}[2]{\hspace{0em}\oldtocsection{#1}{#2}}
\renewcommand{\tocsubsection}[2]{\hspace{1em}\oldtocsubsection{#1}{#2}}

\newcommand*\circled[1]{\tikz[baseline=(char.base)]{
            \node[shape=circle,draw,inner sep=2pt] (char) {#1};}}  

\usepackage{chngcntr}
\counterwithin{figure}{subsection}

\makeatletter
\@namedef{subjclassname@2020}{%
  \textup{2020} Mathematics Subject Classification}

\makeatother

\title[Topological theories and automata]{Topological theories and automata}

\author{Mee Seong Im and Mikhail Khovanov}
 \address{Department of Mathematics, United States Naval Academy, Annapolis, MD 21402, USA}
 \email{\href{mailto:meeseongim@gmail.com}{meeseongim@gmail.com}}
 \address{Department of Mathematics, Columbia University, New York, NY 10027, USA}
 \email{\href{mailto:khovanov@math.columbia.edu}{khovanov@math.columbia.edu}}
 
\date{February 27, 2022}

\providecommand{\keywords}[1]{\textbf{\textit{Key words and phrases.}} #1}

\begin{document}

\def\gen{\mathsf{generators}}
\def\init{\mathsf{in}}
\def\t{\mathsf{t}}
\def\out{\mathsf{out}}
\def\I{\mathsf I}
\def\R{\mathbb R}
\def\Q{\mathbb Q}
\def\Z{\mathbb Z}
\def\mc{\mathcal{c}}
\def\N{\mathbb N} 
\def\C{\mathbb C}
\def\S{\mathbb S}
\def\SS{\mathbb S} 
\def\CP{\mathbb P}
\def\Ob{\mathsf{Ob}}
\def\op{\mathsf{op}}
\def\rat{\mathsf{rat}}
\def\rec{\mathsf{rec}}
\def\coev{\mathsf{coev}}
\def\ev{\mathsf{ev}}
\def\id{\mathsf{id}}
\def\Notation{\textsf{Notation}}
\def\circleft{\raisebox{-.18ex}{\scalebox{1}[2.25]{\rotatebox[origin=c]{180}{$\curvearrowright$}}}}
\renewcommand\SS{\ensuremath{\mathbb{S}}}
\newcommand{\kllS}{\kk\llangle  S \rrangle} 
\newcommand{\kllSS}[1]{\kk\llangle  #1 \rrangle}
\newcommand{\klS}{\kk\langle S\rangle}  
\newcommand{\aver}{\mathsf{av}}  
\newcommand{\ophana}{\overline{\phantom{a}}}
\newcommand{\Bool}{\mathbb{B}}
\newcommand{\dmod}{\mathsf{-mod}}
\newcommand{\pfmod}{\mathsf{-pfmod}}
\newcommand{\primitive}{\mathsf{irr}}
\newcommand{\Bmod}{\Bool\mathsf{-mod}}  
\newcommand{\Bmodo}[1]{\Bool_{#1}\mathsf{-mod}}  
\newcommand{\Bfmod}{\Bool\mathsf{-fmod}} 
\newcommand{\Bfpmod}{\Bool\mathsf{-fpmod}} 
\newcommand{\Bfsmod}{\Bool\mathsf{-}\underline{\mathsf{fmod}}}  
\newcommand{\undvar}{\underline{\varepsilon}} 
\newcommand{\undotimes}{\underline{\otimes}}
\newcommand{\sigmaacirc}{\Sigma^{\ast}_{\circ}} 
\newcommand{\cl}{\mathsf{cl}}
\newcommand{\PP}{\mathcal{P}} 
\newcommand{\wedgezero}{\{ \vee ,0\} } 

\let\oldemptyset\emptyset
\let\emptyset\varnothing

\newcommand{\undempty}{\underline{\emptyset}}
\def\basis{\mathsf{basis}}
\def\irr{\mathsf{irr}} 
\def\spanning{\mathsf{spanning}}
\def\elmt{\mathsf{elmt}}

\def\l{\lbrace}
\def\r{\rbrace}
\def\o{\otimes}
\def\lra{\longrightarrow}
\def\Ext{\mathsf{Ext}}
\def\mf{\mathfrak} 
\def\mcC{\mathcal{C}}
\def\Fr{\mathsf{Fr}}

\def\ovb{\overline{b}}
\def\tr{{\sf tr}} 
\def\det{{\sf det }} 
\def\tral{\tr_{\alpha}}
\def\one{\mathbf{1}}   

\def\lra{\longrightarrow}
\def\hooklra{\raisebox{.2ex}{$\subset$}\!\!\!\raisebox{-0.21ex}{$\longrightarrow$}}
\def\kk{\mathbf{k}}  
\def\gdim{\mathsf{gdim}}  
\def\rk{\mathsf{rk}}
\def\undep{\underline{\epsilon}}
\def\mathM{\mathbf{M}}  

\def\CCC{\mathcal{C}} 

\def\complement{\mathsf{comp}}
\def\Rec{\mathsf{Rec}} 

\def\Cob{\mathsf{Cob}} 
\def\Kar{\mathsf{Kar}}   

\def\dmod{\mathsf{-mod}}   
\def\pmod{\mathsf{-pmod}}    

\newcommand{\alphai}{\alpha_I}  
\newcommand{\alphac}{\alpha_{\circ}}  
\newcommand{\alphap}{(\alphai,\alphac)} 

\newcommand{\brak}[1]{\ensuremath{\left\langle #1\right\rangle}}
\newcommand{\oplusop}[1]{{\mathop{\oplus}\limits_{#1}}}
\newcommand{\ang}[1]{\langle #1 \rangle } 
\newcommand{\ppartial}[1]{\frac{\partial}{\partial #1}} 

\newcommand{\mcA}{{\mathcal A}}
\newcommand{\cZ}{{\mathcal Z}}
\newcommand{\sq}{$\square$}
\newcommand{\bi}{\bar \imath}
\newcommand{\bj}{\bar \jmath}

\newcommand{\undn}{\mathbf{n}}
\newcommand{\undm}{\mathbf{m}}
\newcommand{\cob}{\mathsf{cob}} 
\newcommand{\comp}{\mathsf{comp}} 

\newcommand{\Aut}{\mathsf{Aut}}
\newcommand{\Hom}{\mathsf{Hom}}
\newcommand{\Ind}{\mbox{Ind}}
\newcommand{\Id}{\textsf{Id}}
\newcommand{\End}{\mathsf{End}}
\newcommand{\iHom}{\underline{\mathsf{Hom}}}

\newcommand{\drawing}[1]{
\begin{center}{\psfig{figure=fig/#1}}\end{center}}

\def\endomCempt{\End_{\mcC}(\emptyset_{n-1})}

\begin{abstract} The paper explains the connection between  topological theories for one-manifolds with defects and values in the Boolean semiring and automata and their  generalizations.  Finite state automata are closely related to regular languages. To each pair of a regular language and a circular regular language we associate a  topological theory for one-dimensional manifolds with zero-dimensional defects labelled by letters of the language. This theory takes values in the Boolean semiring. Universal construction of topological theories gives rise in this case to a monoidal category of Boolean semilinear combinations of  one-dimensional cobordisms with defects modulo skein relations. 
The latter category can be interpreted as a semilinear rigid monoidal closure of standard structures associated to a regular language, including minimal deterministic and nondeterministic finite state automata for the language and the syntactic monoid. The circular language plays the role of a regularizer, allowing to define the rigid closure of these structures. When the state space of a single point for a regular language describes a distributive lattice, there is a unique associated circular language such that the resulting theory is a Boolean  TQFT. 

\end{abstract}

\maketitle
\tableofcontents

%
%

\section{Introduction}

\subsection{Universal construction}
In the \emph{universal construction} approach to topological theories~\cite{BHMV,Kh1,RW1,Kh2}, one
is given  an \emph{evaluation} $\alpha$ of closed $n$-dimensional objects $M$, such as $n$-manifolds, taking values in a commutative ground ring $R$ (often a field), 
\[
\alpha \ : \ \mathrm{closed} \ n\mathrm{-dimensional\ objects}\  \lra \ R.
\] 
 The $n$-dimensional objects  may be manifolds, manifolds with boundary or decorations (defects), embedded  manifolds or foams, etc. 
 
Map $\alpha$ is multiplicative on the disjoint union of objects, takes the empty $n$-dimensional manifold $\emptyset_n$ to $1$, and depends only on the isomorphism (diffeomorphism or homeomorphism) class of a manifold:
\begin{eqnarray}
     \alpha(M_1\sqcup M_2) & = &  \alpha(M_1)\alpha(M_2), \label{eq_1}\\
     \alpha(\emptyset_n) & = & 1, \label{eq_0_1}\\
     \alpha(M_1) & = & \alpha(M_2)\label{eq_2}\  \textrm{if}\ M_1\cong M_2. 
\end{eqnarray}
Condition \eqref{eq_0_1} is imposed to avoid degenerate maps, for instance taking all $M$ to $0$. 

From $\alpha$ one can define  state spaces $A(N)$ for $(n-1)$-dimensional objects $N$, by starting with a free $R$-module $\Fr(N)$ with a basis $\{[M]\}_{\partial M \cong N}$ given by formal symbols $[M]$ of all $n$-dimensional objects $M$ which have $N$ as boundary (with a fixed diffeomorphism $\partial M \cong N$).

On $\Fr(N)$ introduce a  bilinear pairing $(\hspace{1mm},\hspace{1mm})_N$ given on basis elements $[M_1],[M_2]$ with $\partial M_1 \cong N \cong \partial M_2$ by 
 coupling $M_1,M_2$ along the  boundary and evaluating  the resulting  closed object $M_1 \cup_N M_2$ via $\alpha$: 
 \[ ([M_1], [M_2])_N \ := \alpha(M_1 \cup_N M_2). 
 \] 
 Now define the state space $A(N)$ as the quotient of $\Fr(N)$ by the kernel of this bilinear form,
 \[
 A(N) \ := \ \Fr(N)/\ker((\hspace{1mm},\hspace{1mm})_N). 
 \]
 If keeping track of orientations, one would instead use $(-M_1) \cup_N M_2$, 
 introduce an involution $\overline{\phantom{a}}$ on $R$ and require 
 \begin{equation}
 \alpha(-M)=\overline{\alpha(M)}. \label{eq_3}
 \end{equation}  
 
 Examples of universal construction include 
 \begin{itemize}
 \item Original paper by Blanchet, Habegger, Masbaum and Vogel~\cite{BHMV}, who introduced that notion and considered the universal construction for the Witten--Reshetikhin--Turaev $SU(2)$ 3-manifold invariant.  Universal construction  was mentioned at about the same time by Kapranov~\cite[Section 2.5]{Ka}. 
 \item 
 The universal pairing theory of Freedman, Kitaev, Nayak, Slingerland, Walker and Wang~\cite{FKNSWW}, also see~\cite{CFW,W}, which constitutes a generic case of the universal construction. In the Freedman et al. universal pairing the ring $R$ is generated by closed $n$-dimensional manifolds, subject to the multiplicativity and isomorphism property above, orientation compatibility and no other relations. In this sense, the ring is the largest possible that one can build out of the given collection of objects. 
 \item Many combinatorially constructed link homology theories are based on state spaces or homology of planar graphs. One way to build these state spaces is via foam evaluation, see~\cite{Kh1,MV,EST,RW1,RW2,KK,AK}. 
 \item Some other examples of the  universal construction  in dimensions $1$ and $2$ were recently considered in~\cite{Kh2,KS3,Kh3,KKO,KQR,KL,IK,IZ}. Meir~\cite{Me} placed the universal construction in the framework of interpolation of monoidal categories. 
 \item 
 The present paper came out of extending~\cite{Kh3} to one-dimensional theories with values in the Boolean semifield $\Bool$ (rather than in a field $\kk$, as in~\cite{Kh3}).  
 \end{itemize} 
  Vector spaces or $R$-modules $A(N)$  that one  assigns to $(n-1)$-dimensional  objects in universal constructions usually fail  the Atiyah tensor product axiom $A(N_1\sqcup N_2)\cong A(N_1)\otimes A(N_2)$, see~\cite{A}. Instead, there is a homomorphism (injective when $R$ is a field)
\[
A(N_1)\otimes A(N_2) \lra   
A(N_1\sqcup N_2), 
\]
giving a sort of a lax  tensor structure on $A$.  

 \vspace{0.1in} 
 
 
 \subsection
 {Universal construction and rigid symmetric categories}
 \label{subsec_rigid} 
 
 Following~\cite{Me,KKO}, it is clear that the universal construction can be done in any \emph{rigid symmetric monoidal category}. A symmetric monoidal category $\CCC$ is rigid~\cite{Mug} if for any object $X$ there is the ``dual" object $X^{\ast}$ together with morphisms 
 \begin{equation} \ev_X \ : \ X^{\ast}\otimes X\lra \one, \ \ 
 \coev_X \ : \ \one \lra X \otimes X^{\ast}
 \end{equation}  
 that induce adjunction isomorphisms 
 \[ \Hom_{\CCC}(X^{\ast} \otimes Y,Z)\cong \Hom_{\CCC} (Y, X \otimes Z) 
 \] 
 for $Y,Z\in \Ob(\CCC)$. That is, compositions
\[ 
\xymatrix@-0.5pc{
 X \ar[rr]^-{\coev_X\otimes \id} & &  
 (X \otimes X^{\ast}) \otimes X \cong X \otimes (X^{\ast}\otimes X) \ar[rr]^-{\id \otimes \ev_X} & & X,  \\ 
 X^{\ast} \ar[rr]^-{\id \otimes \coev_{X}} & &  X^{\ast} \otimes (X\otimes X^{\ast}) \cong (X^{\ast} \otimes X)\otimes X^{\ast} \ar[rr]^-{\ev_{X}\otimes \id } & &  X^{\ast} 
}
\] 
 are identities. These morphisms can be visualized as a ``cap" and a ``cup", see Figure~\ref{cupcap} and relations on them are the isotopy relations for arcs in the plane.  
 
 \vspace{0.1in} 
 
\input{cupcap}

 
  Permutation morphism $X\otimes X^{\ast}\lra X^{\ast}\otimes X$, 
 which is part of the symmetric structure of $\CCC$ and denoted by a crossing of suitably oriented strands labelled $X$ and $X^{\ast}$, allows to define ``cap" $\ev'_X:X\otimes X^*\lra \one$  and ``cup" $\coev'_X:\one \lra X^*\otimes X$ for oppositely oriented arcs, see Figure~\ref{permutation-op-cupcap}.

\input{permutation-op-cupcap} 
 
 Cups and caps can then be used to define the dual $f^{\ast}:Y^{\ast}\lra X^{\ast}$ of a morphism $f:X\lra Y$, see Figure~\ref{closing-dual} on the right. 
 
\input{closing-dual}
 
 Informally, a rigid structure on $\CCC$ allows to ``bend" the morphisms any way we want. From this viewpoint, it would have been more natural to christen these \emph{flexible} categories rather than \emph{rigid} categories, since in these categories the morphisms are fully bendable. 
 
 In particular, a 
 rigid structure on $\CCC$ allows to extend any morphism $f:X\lra Y$ to many endomorphisms of the identity object $\one$ of $\CCC$, by closing up $f$ with any morphism $g:X^{\ast}\lra Y^{\ast}$ via a cup for $X$ and a cap for $Y^{\ast}$, see Figure~\ref{closing-dual} left, where this pairing between $f$ and $g$ is denoted 
 \begin{equation}
 \cl(f,g)\ := \ \ev'_Y\circ (f\otimes g) \circ \coev_X.
 \end{equation} 
 
 \vspace{0.05in} 
 
 \emph{Remark:} A rigid symmetric monoidal category is also called a \emph{compact closed category}  and a \emph{symmetric autonomous category}~\cite{KL80,Se}.

 \vspace{0.1in} 
 
 Let $\CCC$ be a rigid symmetric category. For the following construction it's best to assume that $\CCC$ is  ``set-theoretical" or ``discrete", so that homs between objects don't come with an extra structure, such as that of an abelian group, a module, a topological space, etc.  Consider the commutative monoid $\End_{\CCC}(\one)$ of endomorphisms of the identity object and choose a multiplicative homomorphism $\alpha:\End_{\CCC}(\one)\lra R$ 
 of this monoid into a commutative ring $R$ or even a commutative semiring (see Section~\ref{subsection:semiring-valued-univ-constr} for the definition of a semiring).  
 
 Form the $R$-linear closure $\CCC_R$ of $\CCC$ by extending morphisms to be finite linear combinations of morphisms in $\CCC$ with coefficients in $R$. Category $\CCC_R$ has the same objects of $\CCC$ and $\Hom_{\CCC_R}(X,Y)$ is $R^{\Hom_{\CCC}(X,Y)}$, the free $R$-module (or free $R$-semimodule) with basis $\Hom_{\CCC}(X,Y)$. Composition in $\CCC_R$ is extended from that in $\CCC$ by $R$-bilinearity.
 Category $\CCC_R$ is a rigid symmetric monoidal category. 
 
 Define the quotient category $\CCC_{\alpha}$ of $\CCC_R$ as follows. It has the same objects as $\CCC$ and $\CCC_R$. Introduce an equivalence relation $\sim_{\alpha}$ on  morphisms $f_1,f_2:X\lra Y$ in $\CCC_R$: 
 \begin{equation}\label{eq_sim_alpha} f_1 \sim_{\alpha} f_2 \ \ \mathrm{if} \ \  
 \alpha(\cl(f_1,g))=\alpha(\cl(f_2,g)) \ \ \forall g \in \Hom_{\CCC_R}(X^{\ast}, Y^{\ast}).
 \end{equation}
 In other words, for any $g$ as above, closures of $(f_1,g)$ and $(f_2,g)$ to endomorphisms of $\one$ evaluate to the same element of $R$ under $\alpha$,  see Figure~\ref{closed-evaluation}.
 Notice that it's enough to check the condition above for $g\in \Hom_{\CCC}(X^{\ast},Y^{\ast})$ rather than for $g$ in the larger category $\CCC_R$. 
 
 \vspace{0.1in} 

\input{closed-evaluation}

 Define the category $\CCC_{\alpha}$ as the quotient of $\CCC_R$ by the above equivalence relation. We may call $\CCC_{\alpha}$ the \emph{gligible quotient}, following the terminology in~\cite{KKO,Kh3}. 
 
 For each object $X\in \Ob(\CCC_{\alpha})$ define its state space 
 \begin{equation}\label{eq_A_space}  A_{\alpha}(X):=\Hom_{\CCC_{\alpha}}(\one,X)
 \end{equation} 
 as homs from the identity object to $X$. One can informally think of $A_{\alpha}(X)$ as diagrams (or their $R$-linear combinations) with top boundary $X$ and empty bottom boundary. There is a nondegenerate pairing 
 \begin{equation} \label{eq_A_pairing}   A_{\alpha}(X)\otimes_R A_{\alpha}(X^{\ast})\lra R,\ \    (f,g) \mapsto \alpha(\ev'_X\circ (f\otimes g) ),  
 \end{equation}
 depicted in Figure~\ref{emptybottom-ev}. It is nondegenerate in the sense that if $(f_1,g)=(f_2,g)$ for all $g\in A_{\alpha}(X^{\ast})$ then $f_1=f_2\in A_{\alpha}(X)$, and likewise for the other side of the pairing. 
 
\input{emptybottom-ev}
 
 \vspace{0.05in} 
 
 The quotient relation \eqref{eq_sim_alpha} above is similar to the quotient of a rigid symmetric category by negligible morphisms. If $R$ is a ring, we can instead define the ideal $I_{\alpha}$ which consists of morphisms $f$ such that closing with any $g$ evaluates to $0$ under $\alpha$. Such morphisms are called \emph{negligible}.  Then we quotient preadditive category $\CCC_R$ by that ideal. If $R$ is only a (commutative) semiring, one has to use the above equivalence relation instead. Categories $\CCC_R$ and $\CCC_{\alpha}$ are then \emph{presemiadditive} rather than preadditive.  
 
 \vspace{0.05in} 
 
 With the category $\CCC_{\alpha}$ at hand, one can form related categories, such the additive closure of $\CCC_{\alpha}$, by allowing finite direct sum of objects, the Karoubi closure, by further adding objects for idempotents and so on, as explained in~\cite{KS3} and further considered in~\cite{KQR,KKO,Kh3}.

\begin{remark}
In symmetric monoidal categories, rigidity is simplified compared to just monoidal categories. Universal construction can also be done in rigid monoidal categories that are not necessarily symmetric, see~\cite{KL} for examples. 
\end{remark} 
 
 Categories of (decorated) cobordisms provide examples of rigid symmetric monoidal categories. 
 Assume that we are given a category $\CCC$ of  $n$-dimensional cobordisms between $(n-1)$-dimensional objects. These objects may be manifolds, embedded or decorated manifolds, CW-complexes, and so on. This category is monoidal and often symmetric monoidal, with the tensor product denoted by $\sqcup$. Usually the tensor product is given directly by the disjoint union of $(n-1)$-dimensional objects, respectively $n$-dimensional cobordisms. The empty object is denoted $\emptyset_{n-1}$. The empty morphism (the identity endomorphism of $\emptyset_{n-1}$) is denoted $\emptyset_n$.
 In this setup two cobordisms that are homeomorphic (or diffeomorphic) relative to (rel) the boundary define the same morphism in the category since boundary is fixed during isotopy. 
 
 
 \subsection{Semiring-valued universal construction}
 \label{subsection:semiring-valued-univ-constr}
 In this paper we discuss topological theories when $\alpha$ takes values in a commutative semiring rather than a ring. A commutative semiring $R$ contains distinguished elements $0,1$ and two binary operations $(+,\cdot)$. It's an abelian monoid under each of these operations, with the unit element $0$, respectively $1$, and distributivity holds: $a(b+c)=ab+ac$. The axioms are weaker than those for a commutative ring, due to lack of subtraction: $a-b$ is not defined, in general. Any (commutative) ring is a (commutative) semiring, but not vice versa. An important example is the Boolean semiring $\Bool = \{0,1|1+1=1\}$ consisting of two elements. 
 
 We may equip $R$ with an involution $\overline{\phantom{a}}\ $ if our manifolds or other $n$-dimensional objects are oriented and require $\alpha(-M)=\overline{\alpha(M)}$, but in the present paper we omit this condition. Consequently, our bilinear form $(\hspace{1mm},\hspace{1mm})_N$ (or semibilinear form, if one prefers) is not symmetric, in general. 
   
 Recall that in our categories $\CCC$ of $n$-dimensional cobordisms the identity object $\one$ of $\CCC$ is the empty $(n-1)$-manifold denoted $\emptyset_{n-1}$. The identity endomorphism $\id_{\one}$ of $\one$ is the empty $n$-manifold, denoted $\emptyset_n$. We record this correspondence: 
 \begin{equation}
     \one  \ = \  \emptyset_{n-1}, \ \ \  \ 
     \id_{\one} \ = \ \emptyset_{n}.
 \end{equation}
 Endomorphisms $\End_{\CCC}(\one)=\endomCempt$ of $\one=\emptyset_{n-1}$ are called \emph{closed $n$-objects}
 (or closed cobordisms or closed $n$-manifolds). 
 Assume that we are given a function 
 \begin{equation}
     \alpha \  : \ \endomCempt \lra R 
 \end{equation}
 into a commutative semiring $R$ subject to conditions \eqref{eq_1}-\eqref{eq_2} above, and condition \eqref{eq_3} in the oriented case. For convenience, we list these conditions together below. 
\begin{eqnarray}
     \alpha(M_1\sqcup M_2) & = &  \alpha(M_1)\alpha(M_2), \label{eq1_1}\\
     \alpha(\emptyset_n) & = & 1, \label{eq1_2} \\
     \alpha(M_1) & = & \alpha(M_2)\label{eq1_3}\  \textrm{if}\ M_1\cong M_2, \\
     \alpha(-M) & = & \overline{\alpha(M)}\ \ (\textrm{oriented \ case\ with \ involution}). \label{eq1_4}  
\end{eqnarray}
Note that \eqref{eq1_3} is usually built into our definition of category $\CCC$, since diffeomorphic closed  $n$-manifolds $M_1, M_2$ define the same endomorphism of $\emptyset_{n-1}$. 
 
A semimodule $M$ over a commutative semiring $R$ is an abelian semigroup $(M,+)$ with $0$ as the unit element and a map $R\times M \lra M$ satisfying 
\begin{itemize}
\item $1m=m$ and $0m=0$ for $m\in M$, 
\item $a0=0$ for $a\in R$, 
\item $a(bm)= (ab)m$ for $a,b\in R$ and $m\in M$,
\item $a(m+n)=am+an$ and $(a+b)m=am+bm$ for $a,b\in R$ and $m,n\in M$,
\end{itemize}
see~\cite{G}, for instance. 
For each set $S$, there's the free semimodule $R^S$ of maps from $S$ to $R$. Denote by $R^m$ the standard free semimodule on $m$ generators. A semimodule $M$ is called finitely-generated if there exists a surjective semimodule homomorphism $R^m\lra M$ for some $m$. 

For an object $N$ of $\CCC$ denote by $\Fr(N)$ the free semimodule $R^{\Hom_{\CCC}(\emptyset_{n-1},N)}$ on the set of morphisms from the empty $(n-1)$-manifold $\emptyset_{n-1}$ to $N$. We think of elements of $\Hom_{\CCC}(\emptyset_{n-1},N)$ as ``$n$-manifolds with boundary $N$". Denote by $[M]$ the generator of $\Fr(N)$ associated to a manifold $M$. An element of $\Fr(N)$ has a presentation as a finite sum
\[ \sum_M  a_M [M] , \ \ a_M\in R,
\] 
over finitely many $M\in \Hom_{\CCC}(\emptyset_{n-1},N) $. This presentation is unique if we require that $a_M$'s are not zero (and $0\in \Fr(N)$ corresponds to the empty sum). 
 
Define a (semi)bilinear pairing 
\[  (\hspace{1mm},\hspace{1mm})_N \ : \ \Fr(N) \times \Fr(N) \lra R 
\]
via its values on basis elements
\[ ([M_1],[M_2])_N = \alpha(M_1^{\ast}M_2) \in R
\] 
and extending via semilinearity: 
\[ (\sum_i a_i[M_i],\sum_j b_j [M'_j] )_N = \sum_{i,j} a_i b_j ([M_i],[M'_j])_N = \sum_{i,j} a_i b_j \ 
\alpha(M_i^{\ast}M_j).  \] 
This (semi)bilinear form is symmetric. 

Define the state space $\alpha(N)$ as the quotient of $\Fr(N)$ by the kernel of the form. Working over the semiring and in the absence of subtraction we need to do that carefully. Define an equivalence relation on $\Fr(N)$ by
\[ x\sim y \ \ \mathrm{if} \ \ (x,z)=(y,z) \ \forall z\in \Fr(N), 
\] 
and define the state space as the quotient of $\Fr(N)$ by this equivalence relation, 
\[  \alpha(N) \ := \ \Fr(N)/\sim. 
\] 
The state space is naturally a semimodule over $R$. The semimodule assigned to the empty object is naturally isomorphic to $R$: 
\[ \alpha(\emptyset_{n-1}) \cong R [\emptyset_n]\cong R,
\]
with the generator being the symbol of the empty $n$-manifold; that is the identity morphism on the identity object. 

This construction extends to a functor from $\CCC$ to $R\dmod$, the category of $R$-semimodules. A cobordism $M$ with $\partial M = N_1\sqcup (-N_0)$ induces a natural map of state spaces
\begin{equation}\label{alpha_cob} \alpha(M) \ : \ \alpha(N_0)\lra \alpha(N_1) 
\end{equation}  
defined on symbols by 
\[  \alpha(M)([M_0]) = [M M_0], \ \ \mathrm{for} \  \partial(M_0)=N_0. 
\] 
These maps are compatible with the composition of cobordisms: if composition $MM'$ is defined then 
\[  \alpha(MM') = \alpha(M)\circ \alpha(M'),
\] 
extending $\alpha$ to a functor \[ \alpha \ : \ \CCC \lra R\dmod
\] 
that associates $R$-semimodule $\alpha(N)$ to object $N$ of $\CCC$ and semimodule homomorphism $\alpha(M)$ in  \eqref{alpha_cob} to a cobordism $M$. One usually wants to study corresponding functors for evaluations $\alpha$ with state spaces $\alpha(N)$ sufficiently small, for instance, requiring that $\alpha(N)$ is a finitely-generated $R$-semimodule for any $(n-1)$-manifold $N$. 


\subsection{Summary} 

In the present paper we concentrate on the case $n=1$ and evaluations $\alpha$ taking values in the Boolean semiring $\Bool$. Our one-manifolds are oriented and come with $0$-dimensional defects (points on one-manifolds) labelled by elements of a finite set $\Sigma$ (set of letters). We allow one-manifolds with boundary. 
Correspondingly, one-dimensional cobordisms may have components (intervals) with one of both endpoints strictly inside the cobordism (\emph{floating} boundary points). We relate resulting one-dimensional $\Bool$-valued topological theories with $0$-dimensional defects to regular languages and finite state automata. 

\vspace{0.05in} 

In Section~\ref{sec_regular} we review the classical notions of a regular language, deterministic and non-deterministic finite automata, and also consider the notion of \emph{circular automata}, for cyclicly invariant languages (\emph{circular languages}). Such languages naturally come out from our interpretation of languages as $\Bool$-valued evaluations of one-dimensional manifolds with defects.  Evaluations of intervals with defects correspond to languages, while those of circles with defects to \emph{circular languages}. Evaluation $\alpha$ of arbitrary decorated one-manifolds requires a pair of languages: language $L_I$ for evaluation of decorated intervals (the \emph{interval language} of $\alpha$) and a \emph{circular} language  $L_{\circ}$ for evaluation of decorated circles (the \emph{circle language} of $\alpha$),
\begin{equation}  \label{eq_ev_alpha} 
    \alpha \ = \ (\alpha_I,\alpha_{\circ}) \ = \ (L_I,L_{\circ}). 
\end{equation}
Here $L_I\subset \Sigma^{\ast}$ denotes a language, that is, a subset of the set of all words 
\begin{equation}
\Sigma^{\ast} \ := \ \sqcup_{n\ge 0} \{ a_1\cdots a_n \, | \, a_i\in \Sigma \} 
\end{equation}
in the alphabet $\Sigma$. Languages are in a bijection with Boolean functions $\alpha_I:\Sigma^{\ast}\lra \Bool$, where $L_I=\alpha_I^{-1}(1)$. Likewise, $L_{\circ}=\alpha_{\circ}^{-1}(1)$ denotes a circular language, for a map $\alpha_{\circ}: \Sigma^{\ast}_{\circ}\lra \Bool$. 
If a circular language contains a word $\omega$, it then contains all its cyclic rotations $\omega_2\omega_1$ (for $\omega=\omega_1\omega_2$), and the set $\Sigma^{\ast}_{\circ}=\Sigma^{\ast}/\sim $ consists of equivalence classes under this relation. 

\vspace{0.05in} 

Section~\ref{sec_oned} describes $\Bool$-valued one-dimensional topological theories with defects. Such a theory corresponds to evaluation $\alpha$ as in \eqref{eq_ev_alpha}. To $\alpha$ we assign a rigid symmetric monoidal category $\CCC_{\alpha}$ where hom spaces between objects are $\Bool$-semimodules. Objects are oriented 0-manifolds, encoded by their sign sequences $\varepsilon=(\varepsilon_1,\ldots, \varepsilon_k),$ $\varepsilon_i\in \{+,-\}$. Hom spaces are $\Bool$-semilinear combinations of $\Sigma$-decorated oriented 1-cobordisms, modulo relations coming from the universal construction for $\alpha$. 

We say that a theory $\alpha$ is \emph{rational} or \emph{finite} if hom spaces in $\CCC_{\alpha}$ are finitely-generated (equivalently, finite) $\Bool$-semimodules. We observe in Proposition~\ref{prop_when_regular}  that $\alpha$ is rational if and only if both languages $L_I$ and $L_{\circ}$ are regular, that is, described by finite state automata. 

\vspace{0.05in}

Section~\ref{subsec_semiduality} contains a review of $\Bool$-semimodules. $\Bool$-semimodule structure is equivalent to that of a semilattice with the $0$ element. Finite $\Bool$-semimodules correspond to finite semilattices, and the latter can be enhanced to finite  lattices, in a unique way. Finite projective $\Bool$-semimodule correspond to distributive 
 lattices. We observe in Proposition~\ref{prop_proj_rigid} that finite projective $\Bool$-semimodules and semimodule maps constitute a rigid symmetric monoidal category. In the diagrammatical interpretation, the identity morphism of any object can bend arbitrarily, with isotopy relations in Figure~\ref{cupcap} satisfied.  

\vspace{0.05in} 

Each evaluation $\alpha$ corresponds to a pair of languages: evaluation $\alpha_I$ of words on an interval encodes a language $L_I$, and  evaluation $\alpha_{\circ}$ of words on a circle encodes a circular language $L_{\circ}$. 
State spaces $A(+),A(-)$ of a single point carry information only about the interval language $L_I$. The state space $A(+-)$ keeps track of both languages (interval and circular languages) of $\alpha$ and their interactions. 

\vspace{0.05in} 

Section~\ref{sec_sspaces} considers state spaces $A(+), A(-), A(+-)$ of 0-manifolds which are a single point (with $+$ or $-$  orientation) or a pair of points for these topological theories. We explain how to extract the minimal \emph{deterministic} finite automaton (DFA) for a language from the state space $A(-)$  of a single point in the corresponding theory and the minimal DFA for the opposite language from $A(+)$. Taking the minimal free cover $\Bool^J\lra A(-)$ of the  $\Bool$-semimodule $A(-)$ together with a lifting of the action of the free word monoid $\Sigma^{\ast}$ from $A(-)$ to $\Bool^J$ describes all minimal \emph{non-deterministic} finite automata (NFA) for the language. States of a minimal non-deterministic automaton are enumerated by elements of $J$ or, equivalently, by the irreducible (indecomposable) elements of the $\Bool$-semimodule $A(-)$. 

Semimodules $A(+)$ and $A(-)$ are dual. More generally, $A(\varepsilon)$ and $A(\varepsilon^{\ast})$ are dual semimodules, where $\varepsilon^{\ast}=(\varepsilon_k^{\ast},\ldots, \varepsilon_1^{\ast})$ is the dual sequence ($+^{\ast}=-, -^{\ast}=+$). 

\vspace{0.05in} 

If $A(-)$ is a projective semimodule, there is a unique circular language $L_{\circ}$ build from $L_I$ with the property that the inclusion $A(+)
\otimes A(-)\subset A(+-)$ is an isomorphism of semimodules. Circular language $L_{\circ}$ is defined via the identity decomposition for $L_I$, possible exactly when $A(-)$ is projective. 

For such a pair $(L_I,L_{\circ})$ inclusions $A(\varepsilon)\otimes A(\varepsilon')\subset A(\varepsilon\varepsilon')$ are isomorphisms for all sign sequences $\varepsilon,\varepsilon'$, and the category $\CCC_{\alpha}$ gives rise to a $\Bool$-valued TQFT in the sense of Atiyah, see Section~\ref{subsec_decomp}. Such a structure exists if and only if $A(-)$ is a projective semimodule and $L_{\circ}$ is then determined by $L_I$,  so that $\CCC_{\alpha}$ is a TQFT. A semimodule is projective if and only if the lattice constructed from $A(-)$ is a \emph{distributive} lattice. A finite distributive lattice is isomorphic to the lattice of open sets of a finite topological space. 

\vspace{0.05in} 

Going back to the general case or regular languages $L_I,L_{\circ}$ and associated evaluation $\alpha$, 
the circular language $L_{\circ}$ can be thought of as a way to regularize an interval language $L_I$ and build a rigid symmetric monoidal category $\CCC_{\alpha}$ from the pair $\alpha=(\alpha_I,\alpha_{\circ})$.

\vspace{0.05in} 

In Section~\ref{sec_semimod_automata} we 
introduce a diagrammatic rigid monoidal category of decorated 1-cobordisms associated to a pairing between a Boolean semimodule $M$ and its dual. Cobordisms are decorated only at the endpoints that are strictly inside the cobordism (floating endpoints). Depending on orientation, one labels these endpoints by generators of $M$ or $M^{\ast}$, and such decorated interval is evaluated via the pairing. There are no dots on intervals or circles and evaluation of the circle is an additional Boolean-valued parameter. The resulting category can be thought of as a rigid monoidal closure of the pairing between $M$ and its dual. In the same section we combine this construction with  evaluations of the earlier type, when dots are allowed, corresponding to languages and circular languages.  A special case of this construction, considered in Section~\ref{subsec_measuring}, leads to the notion of a ``distance" between two regular languages that measures how much bigger the state space $A(-)$ is in the monoidal category that keeps track of both languages versus a category for a single language.

\vspace{0.05in} 

Section~\ref{section:examples} contains a number of examples to illustrate various constructions and results of the paper. 

\vspace{0.1in} 

{\bf Acknowledgments.} M.S.I. was partially supported by AMS-Simons Travel Grant and U.S. Naval Academy. M.K. was partially supported by NSF grant DMS-1807425. 
The authors are grateful to Kirill Bogdanov for interesting discussions and to the program Braids in Representation Theory and Algebraic Combinatorics at ICERM at Brown University for conducive working environment. 

\vspace{0.1in}

\section{Regular languages, finite state automata and circular automata}
\label{sec_regular} 

This section contains a brief review of regular languages and finite state automata, including deterministic finite automata (DFA) and nondeterministic finite automata (NFA). We also suggest how to define \emph{circular automata} that accept regular \emph{circular languages}. 


\subsection{Regular languages}
Denote by $\Sigma$ a finite set (the set of \emph{letters}) and by $\Sigma^{\ast}=\emptyset\sqcup \Sigma \sqcup \Sigma^2\sqcup \ldots$ the set of finite sequences of elements of $\Sigma$, that is, the set of \emph{words} in $\Sigma$. A subset $L\subset \Sigma^{\ast}$  is called a \emph{language}. 

A language $L$ correspond to a map $\alpha: \Sigma^{\ast}\lra \Bool$, with $L=\alpha^{-1}(1)$. Here 
\[
\Bool \ = \  \{0,1|1+1=1\}
\] is the Boolean semiring. We can formally write 
\[ L\ = \ \sum_{\omega\in L} \omega
\] 
and think of this sum as formal noncommutative power series, with $\Sigma$ the set of variables, and with coefficients in the Boolean semiring $\Bool$. Terms with coefficient $0$, that is, terms not in $L$, are omitted. 

A language $L$ is called \emph{finite} if $|L|<\infty$. A language is called \emph{regular} if it can be inductively obtained as follows:
\begin{itemize}
    \item A finite language is regular. 
    \item Union $L\cup L'$ and ordered product  $L L'$ of regular languages is regular.
    \item The star closure $L^{\ast}=\emptyset \sqcup L \sqcup L^2 \sqcup \ldots$ of a regular language is regular. 
\end{itemize}

A language $L$ is regular if and only if the opposite language $L^{\op}$ is regular (the latter is given by reading all words in $L$ in the opposite direction). 

Introduce an equivalence $\sim_L$ on $\Sigma^{\ast}$ by $\omega\sim_L \omega'$ if and only if for any $x,y\in \Sigma^{\ast}$ words $x\omega y$ and $x\omega'y$ are either both in $L$ or both not in $L$. The set of equivalence classes $E_L = \Sigma^{\ast}/\sim_L$ is naturally a monoid under composition $\omega_1 \ast \omega_2 = \omega_1\omega_2$. The identity of $E_L$ is given by the equivalence class of the empty sequence $\emptyset$. Monoid $E_L$ is called the \emph{syntactic monoid} of $L$.

The following proposition is well-known~\cite[Proposition 10.1]{E1}.  

\begin{prop} Monoid $E_L$ is finite if and only if $L$ is a regular language. \end{prop} 

This map from regular languages to (isomorphism classes of) finite monoids is easily seen to be surjective, if one allows finite sets of letters $\Sigma$ of arbitrary large size. 

\begin{remark}
The relation between finite monoids and regular languages (and finite automata, see later) to some degree answers the question why an  undergraduate course on modern algebra often has a whole semester dedicated to finite groups while the notion of a monoid may not even appear in a year-long course on the topic. On one hand, finite groups bear close relation to number theory, starting with classification of finite cyclic groups and Sylow theorems, while finite monoids are not structurally as beautiful. Their number (up to isomorphism) grows quickly with size, and no similar relation to number theory is in sight. Second, finite monoids are studied implicitly in an introductory course on formal languages in any computer science department. Each finite automaton gives rise to the language $L$ that it accepts (see below) and, in turn, to the finite monoid $E_L$. In this sense, finite monoids are implicitly encoded in finite state automata and (implicitly) appear in courses in computer science departments. 
\end{remark}


\subsection{Finite state automata}
\label{subset_fsa} 
A \emph{deterministic} finite automaton (DFA) $F$ consists of a finite set of states $Q$, the subset $Q_{\t}\subset Q$ of \emph{terminal} or \emph{accepting} states, the initial state $q_{\init}\in Q$ 
and a map $\delta:\Sigma\times Q\lra Q$. The latter map describes which state  to go to from a state $s$ of an automaton upon reading letter $a\in \Sigma$, for all $s$ and $a$. Given a word $\omega\in \Sigma^\ast$, start with the initial state $q_{\init}$ and travel through the states of $F$: if we're in a state $s$ and the next letter of $\omega$ is $a$, we go to state $\delta(a,s)$. If at the end of the word we're in one of the accepting states (from the subset $Q_{\init}$), word $\omega$ is in the language $L_F$ associated to automaton $F$. Otherwise, $\omega\notin L_F$. 
The following result holds, see~\cite{E1,HU,U}.

\begin{prop}[Kleene's theorem] A language $L_F$ associated to a deterministic finite automaton $F$ is regular. Any regular language $L$ can be obtained from some DFA.
\end{prop} 

 Note that there are infinitely many pairwise nonisomorphic DFAs giving rise to any given regular language $L$. We obtain surjective many-to-one map 
 
 \vspace{0.1in} 
 
 \begin{center}
 (Isomorphism classes of) DFA with alphabet $\Sigma \quad \lra \quad$ regular languages $L\subset \Sigma^{\ast}$.   
\end{center}

 \vspace{0.1in} 

Given a language $L$, consider an equivalence relation $\sim_{\ell L}$ on $\Sigma^{\ast}$ with $\omega\sim_{\ell L}\omega'$ if and only if $\forall x\in \Sigma^{\ast}$ words $\omega x$ and $\omega'x$ are either both in $L$ or both not in $L$. Denote the set of equivalence classes by $X_L$.  
The syntactic monoid $E_L$ acts faithfully on the set $X_L$ by right multiplication (concatenation) by equivalence classes in $E_L$, giving us a right action
\[ X_L \times E_L \lra X_L . 
\] 

\begin{remark} \label{rem_actions} In modern mathematics we tend to use left actions. In the theory of finite state automata the convention is to have $\Sigma^{\ast}$ and $L$ act on the right on the set of states of an automaton. We will keep this difference in mind throughout the paper.
\end{remark}

Likewise, define equivalence relation $\sim_{r L}$ by writing $x$ to the left of $\omega$ and $\omega'$, so that $\sim_{r L} = \sim_{\ell L^{\op}}$. Denote the set of equivalence classes for $\sim_{r L}$ by ${}_L X$. Monoid $E_L$ acts on ${}_L X$ on the left by multiplication (concatenation):
\[ E_L \times {}_L X \lra {}_L X. 
\] 

\begin{prop}
 Set $X_L$ is finite if and only if the language $L$ is regular. 
\end{prop}

Set $X_L$ gives rise to a canonical automaton $F(X_L)$ accepting a regular language $L$. States of $F(X_L)$ are elements of $X_L$, an element $a$ of $\Sigma$ acts on equivalence classes in $X_L$ by right multiplication by $a$. Automaton $F(X_L)$ is universal among automata accepting $L$, as follows. Assume that an automaton $F$ accepts $L$ and any state of $F$ is reachable from the initial state. Then there's a unique surjective map $F\lra F(X_L)$, that is, a surjective map from the states of $F$ to the states of $F(X_L)$ that respects the structure of automata: initial state goes to the initial state, the map commutes with the action of $\Sigma$, and accepting states go to accepting states. 

\begin{prop} Automaton $F(X_L)$ is the unique, up to an isomorphism, deterministic automaton with the smallest number of states accepting $L$.  
\end{prop} 

Likewise, the set ${}_LX$ gives rise to a canonical (and minimal) automaton accepting the opposite language $L^{\op}$.

\vspace{0.1in}

A nondeterministic finite automaton    $(Q,\delta,Q_{\init},Q_{\t})$ over a finite set $\Sigma$ consists of  
\begin{enumerate}
\item a finite set $Q$ of states,
\item a transition function $\delta:Q\times \Sigma\lra \PP(\Sigma)$, where $\PP(\Sigma)$ is the set of subsets of $\Sigma$,
\item A nonempty subset of \emph{initial} states $Q_{\init}\subset Q$ and a subset of \emph{terminal} states $Q_{\t}\subset Q$. 
\end{enumerate}
A nondeterministic automaton accepts a following language $L$. Word $\omega=a_1\cdots a_n$ is in $L$ if and only if there exists a sequence of states $(q_0,q_1,\ldots, q_n)$ such that $q_0\in Q_{\init},q_n\in Q_{\t}$ and $q_{i+1}\in \delta(q_i,a_i)$ for $0\le i\le n-1$. The following is another basic result on automata~\cite{E1,HU,BR2}. 

\begin{prop} 
A language is accepted by a nondeterministic automaton if and only if it is regular. 
\end{prop} 

In Section~\ref{subsec_plus_minus} we will interpret nondeterministic automata for a language $L$ via free semimodule covers of the state space $A(-)$ for the topological theory which has $L$ as the interval language. 

\vspace{0.1in} 

\begin{example}  \label{ex_L_b}
Consider the language $L=(a+b)^{\ast}b(a+b)$ that consists of words with the second to last letter $b$. The minimal DFA for $L$, shown in Figure~\ref{new-file-003}, has four states, with $q_0$ the initial state. States $q_2,q_3$ are accepting, shown with thick border. 

\input{new-file-003}

For this language, the set ${}_L X$ has four elements, in bijection with the states  of the minimal automaton. These elements can represented by words $\emptyset,b,b^2,ba$. These paths go to the states $q_0,q_1,q_2,q_3$, respectively, from the initial state $q_0$.

\vspace{0.05in} 

A nondeterministic finite automaton with the minimal number of states for $L$ is given in Figure~\ref{new-file-002}.

\input{new-file-002}

The dual language $L^{\op}=(a+b)b(a+b)^{\ast}$ consists of words whose second letter is $b$. Its minimal DFA is shown in  Figure~\ref{new-file001}, with states $q_0',q_1',q_2',q_3'$, where  $q_0'$ is the initial state. 
 
\input{new-file001}

The syntactic monoid $E_L$ for this language is easy to write down due to the specifics of the language $L$.  Syntactic monoid consists of seven elements, that can be represented by words of length at most two, $E_L=\{\emptyset,a,b,aa,ab,ba,bb\}$. The multiplication is given by concatenation followed by truncating a word to leave at most two rightmost letters. For instance, $(ab)a=aba\sim ba$. 
We will revisit this example later in Example~\ref{ex:2nd-last-b} in  Section~\ref{section:examples}. 
\end{example}

The theory of  finite state automata  can be found in many textbooks on the field, including treatises by Eilenberg~\cite{E1,E2} and Conway~\cite{Co}, see also~\cite{EK,Ch,BR2,HU,KuS,SS,U}. 

Any monoid $E$ gives rise to the bialgebra $\kk[E]$ with $\kk$ a field and comultiplication $\Delta(g)=g\otimes g$ for $g\in E$. Finite state automata with associated syntactic monoids are considered from this viewpoint in  Underwood~\cite[Chapter 2]{U}.


\subsection{Circular languages and automata}
\label{subsec_circular} 
A language $L$ is called \emph{circular} if $\omega_1\omega_2\in L$ whenever $\omega_2\omega_1\in L$, so that together with every word in it $L$ contains all cyclic rotations of that word. Denoting the set of equivalence classes of words in $\Sigma^{\ast}$ by $\sigmaacirc$, we can write $L\subset \sigmaacirc$ for a circular language. 

It appears hard to directly modify the notion of finite state automata to describe those that accept circular languages. The reason is that an automaton gives a local construction of a regular language: we read one letter and move around in the automaton correspondingly. Circularity is a global condition that can be tested only when an entire word is available. Pictorially, and that's closely related to the topic of our paper, one can image a word $\omega$ as being written along an oriented interval, with orientation allowing to distinguish between the word $\omega$ and its opposite $\omega^{\op}$. One moves across the interval while reading the word, and the current position of the reader can be depicted by a mark between two consecutive letters, see Figure~\ref{circular-lang-01}. 

\vspace{0.1in} 

\input{circular-lang-01}

\input{circ-abab-02}


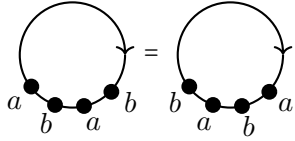
A circular word can be naturally written along an oriented circle, see Figure~\ref{circular-lang-01} and an example in Figure~\ref{circ-abab-02} of equal circular words.

If one mark is placed on a circle between two consecutive letters and moved as we read the word, we might not recognize when we come back to the starting position. A way around that difficulty is to place two marks, first next to each other between two consecutive letters, and then read the word in two opposite directions, moving the marks past the letters as that happens and increasing portion of the circle between the two marks until they come together again somewhere on the circle, see Figure~\ref{circular-lang-02}. 

\input{circular-lang-02}

A deterministic \emph{circular} finite automaton (DCFA) consists of a finite set of letters $\Sigma$, a finite set $Q$ of states, a subset $Q_{\t}$ of terminating (or accepting) states, the initial state $q_{\init}\in Q$,
and two maps 
\[ \delta_{\ell}, \delta_r \ : \  \Sigma \times Q \lra Q
\] 
subject to the following conditions. For each $a\in \Sigma$ we also write the maps $\delta_{\ell}(a,\bullet), \delta_r(a,\bullet): Q \lra Q$ as $\delta_{\ell,a}, \delta_{r,a}$, respectively. 
\begin{enumerate}
    \item\label{circular-first} For all $a,b\in\Sigma$, maps $\delta_{\ell,a},\delta_{r,b}:Q\lra Q$ commute: 
    \[ \delta_{\ell,a}\delta_{r,b} \ = \ \delta_{r,b}\delta_{\ell,a}. 
    \]
    \item\label{circular-second} $\delta_{\ell}(a,q_{\init}) = \delta_r(a,{q_{\init}})$ for all $a\in\Sigma$. 
    \item\label{circular-third} States $\delta_{\ell}(a,q),\delta_r(a,q)$ are either both accepting (both in $Q_{\t}$) or not accepting (both in $Q\setminus Q_{\t}$), for any $q\in Q,a\in \Sigma$. 
\end{enumerate}
This definition can be motivated as follows. The initial state $q_{\init}$ 
is the state of the automaton 
when it starts reading a circular word and lands on a circle in some interval $I_0$ between two consecutive letters in a circular word $\omega$. At that moment the two marks are on the same interval $I_0$. ``Left" transition function $\delta_{\ell}(a,q)$ gives the next state of the automaton from state $q$ when the mark going clockwise along the circle encounters letter $a$ (a dot labelled $a$ in our pictures). ``Right" transition function $\delta_r(a,q)$ describes the next state of the automaton when the mark going counterclockwise along the circle encounters letter $a$. When two marks together traverse the circle and meet at some interval between two consecutive dots (letters), circular word $\omega$ is accepted if and only if the automaton is an accepting state (in subset $Q_{\t}$). 

Commutativity condition~\eqref{circular-first} is needed to make sure that the order of passing a label $a$ on the left and a label $b$ on the right does not matter, see Figure~\ref{circular-lang-03}.

\input{circular-lang-03}

Condition~\eqref{circular-second} says that the state of an arc with a single label $a$ is the same whether the automaton starts with two marks to the left of $a$ or to the right of $a$, see Figure~\ref{circular-lang-04}.

\input{circular-lang-04}

Condition~\eqref{circular-third} is dual to \eqref{circular-second} and says that if only one letter is left unread, then crossing it with either of the two marks does not matter for the accept or reject decision, see Figure~\ref{circular-lang-05}.

\vspace{0.1in} 

\input{circular-lang-05}

We view a circular language $L$ as a subset $L\subset \sigmaacirc$. Denote by $\widetilde{L}=\pi^{-1}(L)\subset \Sigma^{\ast}$ the associated language. Here $\pi: \Sigma^{\ast}\lra \sigmaacirc$ is the quotient map from the set of $\Sigma$-sequences  to circular $\Sigma$-sequences. 

A circular language $L$ is called \emph{c-regular}, or \emph{circular regular} if it's accepted by some finite deterministic circular automaton. 

\begin{prop} \label{prop_c_reg} A circular language $L$ is c-regular if and only if the language $\widetilde{L}$ is regular. 
\end{prop}

\begin{proof} Given a circular automaton for $L$, forgetting the right action $q_r$ gives a deterministic finite automaton for $\widetilde{L}$, showing implication in one direction. In other words, we only move the left mark, keeping the right mark fixed on the circle. 

In the other direction, suppose that $\widetilde{L}$ is regular. Then the syntactic monoid $E_{\widetilde{L}}$ is finite. From that syntactic monoid we can build the circular automaton for $L$ using $E_{\widetilde{L}}$ as the set of states, empty word as the initial state, and left and right commuting actions of $a\in \Sigma$ via left and right concatenation with $a$, viewed as maps from $E_{\widetilde{L}}$ to itself. Invariance of $\widetilde{L}$ under rotations of words insures that the resulting automaton is indeed circular and satisfies the above properties \eqref{circular-first}-\eqref{circular-third}. 
\end{proof} 

The proposition shows that c-regular circular languages are just regular languages which are rotationally invariant: $\omega_1\omega_2\in L\Leftrightarrow \omega_2\omega_1 \in L$.   

\vspace{0.05in} 

\begin{remark} We use the word \emph{circular} to describe a language as above, since the word \emph{cyclic} is already reserved in the literature~\cite{BR1} for languages that are both circular and together with each word $\omega$ contain all its powers and roots $\sqrt[k]{\omega}$ (when any such exist). A possible alternative is to call circular languages \emph{central languages}, following that usage in noncommutative power series~\cite{Re}.
\end{remark}

{\it Cyclic derivatives.}
  Rota--Sagan--Stein's~\cite{RSS} cyclic derivative $\partial/\partial a$, where $a\in \Sigma$, takes a cyclic word $\omega$ and sends it to the sum of ordinary (non-cyclic) words given by removing all occurrences of letter $a$ in $\omega$, for instance,
  \[ \frac{\partial}{\partial a}(babbcaa) \ = \ bbcaab + ababbc + babbca 
  \] 
  (note that removing a letter breaks up the cycle of a circular word), see Figure~\ref{sect02-001}. 
  Rota, Sagan and Stein showed that cyclic derivative of a rational noncommutative power series is rational.  
  Cyclic derivative operation is important in the study of Calabi--Yau algebras and noncommutative geometry~\cite{Ko93,Gin07}. 
  
  \input{sect02-001}
  
  We note that cyclic derivative $\partial/\partial a$ of a circular language $L$ can be defined as follows: derivative $\partial/\partial a(\omega)$ of a circular word is the $\Bool$-semilinear combination of (usual) words given by removing each occurrence of $a$ in $\omega$. For instance,  
  \[ \ppartial{a}(baba) = bab+bab=bab.
  \] 
  Thinking of a language (or a circular language) as a formal sum $L=\sum_{\omega\in L} \omega$ of words (or circular words) in it, define the derivative of a circular language by 
  \begin{equation}
      \ppartial{a}(L) := \sum_{\omega\in L} \ppartial{a}(\omega). 
  \end{equation}
  \begin{prop} Partial derivative $\ppartial{a}(L)$ of a c-regular circular language is c-regular. 
  \end{prop} 
  \begin{proof} By Proposition~\ref{prop_c_reg}, it suffices for show that the ordinary language $\widetilde{L_a}$ associated to the circular language $L_a:=\ppartial{a}(L)$ is regular. One can start with a nondeterministic finite automaton $F$ that recognizes $\widetilde{L}$ and modify it to an automaton for $\widetilde{L_a}$
  by creating two copies of $F$ with a transition from the first to the second copy possible if the letter being read is $a$. Acceptable states need to be suitably changed (removed from the first copy of $F$). We leave details to an interested reader. 
  \end{proof}

  \begin{remark}
  The Hessian of a circular language can be defined by analogy with Ginzburg~\cite{Gin07}. 
  It would be interesting to understand the analogue of the quotient 
  \[ 
  \C\langle x_1,\ldots, x_n\rangle/(\!\!(\partial \Psi/\partial x_i)\!\!)_{i=1,\ldots, n}
  \] 
  of the free algebra by the two-sided ideal of the partial derivatives of the potential $\Psi$ in our case, with $\Psi$ replaced by a cyclic regular language $L$. As a first step, one can look for the analogue of the Poincar\'e lemma~\cite[Proposition~1.5.13]{Gin07}  for languages. 
  \end{remark}

%
%

\section{One-dimensional topological theories with defects} 
\label{sec_oned} 

\subsection{One-dimensional topological theories valued in \texorpdfstring{$\Bool$}{Bool}}
\label{subsec_valued} 

We refer to \cite{Kh3,IZ} for some background on one-dimensional topological theories. 

We fix a finite set $\Sigma$ of labels and consider oriented one-manifolds, possibly with boundary and with zero-dimensional submanifolds (defects) labelled by elements of $\Sigma$. We may write $\Sigma$ as $\{s_1,\ldots, s_m\}$ or as $\{a,b,\ldots\}$.  Defects live strictly inside the manifold and not on its boundary. Any such decorated 1-manifold $M$ is a union of intervals and circles. 

For a given connected component $M_i$ of $M$ that is an interval, its homeomorphism (or diffeomorphism) type is determined by the word $\omega$ one reads while traversing $M_i$ in the direction of its orientation, $\omega=a_1a_2\cdots a_n$ for $a_i\in \Sigma$. Denote an interval labelled by the word $\omega$ by  $I(\omega)$. See Figure~\ref{fig_03-01} left. 

\vspace{0.1in} 

\input{fig_03-01}

\vspace{0.1in} 

For a connected component that is a circle $\SS^1$, traveling along the component produces a word $\omega$, well-defined up to cyclic order. Denote a circle with label $\omega$ by $\SS^1(\omega)$. Due to cyclic invariance,  $\SS^1(\omega_1\omega_2)\cong \SS^1(\omega_2\omega_1)$, where $\omega_1,\omega_2\in \Sigma^{\ast}$ are words in $\Sigma$. Define an equivalence relation $\sim_{\circ}$ on $\Sigma^{\ast}$ via $\omega_1\omega_2\sim_{\circ} \omega_2\omega_1$ for 
 $\omega_1,\omega_2\in \Sigma^{\ast}$ and call it \emph{the circular equivalence}. Denote by 
 \[
 \Sigma^{\ast}_{\circ}\ := \  \Sigma^{\ast}/\sim_{\circ}
 \]
 the set of equivalence classes. 
We call a word $\omega\in \Sigma^{\ast}$ up to circular equivalence a \emph{circular word}. 

These two types of connected components are shown in Figure~\ref{fig_03-01}. 

\vspace{0.1in} 

Consider the category $\CCC$ of oriented one-dimensional $\Sigma$-decorated cobordisms. Its objects are sequences of signs $\varepsilon=(\varepsilon_1,\ldots, \varepsilon_k)$, $\varepsilon_i\in\{+,-\}$, including the empty sequence, denoted $\emptyset$. Morphisms are $\Sigma$-decorated oriented cobordisms. Components of cobordisms are allowed to end in the middle and not at either of the two sign sequences. Two cobordisms represent the same morphism if they are diffeomorphic rel boundary. Figure~\ref{generalized-ex-01} depicts an example of a morphism in this category.  

\input{generalized-ex-01}

This morphism has three ``floating" components (with all boundary points strictly inside the cobordism). Two of these components are intervals, with words $\emptyset$ and $b$ on them, and one component is a circle, with the cyclic word $acb$. Four other components each have one end on the boundary (top or bottom) of the cobordism and one end inside the cobordism, and carry words $\emptyset$, $a$, $ab$, and $\emptyset$ respectively. The four remaining components each have both ends on the boundary of the cobordism. Endpoints on the boundary of the cobordism can be called \emph{outer}, while endpoints strictly inside the cobordism are called \emph{inner}. Composition of morphisms is given by concatenation of cobordisms.

This category is rigid strict symmetric monoidal, with the tensor product of morphisms given by placing them next to each other. The dual of a sequence 
$\underline{\varepsilon}=(\varepsilon_1,\ldots, \varepsilon_k)$ is given by reversing its signs $+\leftrightarrow -$ and reversing their order, 
\[\underline{\varepsilon}^{\ast}\ :=\ (\overline{\varepsilon_k},\ldots, \overline{\varepsilon_1}), \hspace{4mm} \mbox{ where } 
\hspace{4mm} 
\overline{+}=-, \ \overline{-}=+. 
\] 
An example of the evaluation morphism for this rigid structure is shown below. Note that we're using the convention $(N_1\otimes N_2)^{\ast}=N_2^{\ast}\otimes N_1^{\ast}$, with the duality reversing the order of objects in the tensor product. This is made to avoid having lines in the duality morphisms intersect, which would add to the graphical complexity of the pairing between $N$ and $N^{\ast}$, when $N$ is a tensor product of several terms. 
See Figure~\ref{parallel-cups-01}.

\vspace{0.1in}

\input{parallel-cups-01}

\vspace{0.1in} 

The empty 0-manifold (the empty sequence $\emptyset$) is the identity object. 
For more details on $\widetilde{\CCC}$ we refer to~\cite[Section 3.1]{Kh3}, where this category is denoted $\widetilde{\CCC}$ and $S$ rather than $\Sigma$ is used for the set of letters.  

\vspace{0.1in} 

A Boolean evaluation $\alpha$ for the category $\CCC$ 
is a monoid homomorphism $\alpha:\End_{\CCC}(\emptyset)\lra \Bool$. It is determined by its values on homeomorphism classes of connected closed objects. The latter are of two types: decorated intervals and circles. Decorated intervals $I(\omega)$, up to homeomorphism, are classified by $\omega\in \Sigma^{\ast}$. Decorated circles $\SS^1(\omega)$ are classified by circular words $\omega\in \Sigma^{\ast}_{\circ}$. 

We see that $\alpha$ is determined by two maps 
\[  \alphai \ : \ \Sigma^{\ast} \lra \Bool, 
\hspace{4mm} 
\alpha_{\circ} \ : \ \Sigma^{\ast}_{\circ} \lra \Bool. 
\] 

These maps are classified by subsets of $\Sigma^{\ast}$ and $\Sigma^{\ast}_{\circ}$, respectively. We call a subset $L_I = \alphai^{-1}(1)\subset \Sigma^{\ast}$ \emph{the interval language} of $\alpha$ and $L_{\circ}= \alphac^{-1}(1)\subset \Sigma^{\ast}_{\circ}$ \emph{the circle language} of $\alpha$. We call the pair $L=(L_I,L_{\circ})$ the language of $\alpha$.


\subsection{\texorpdfstring{$\Bool$}{Bool}-semimodules and duality} 
\label{subsec_semiduality} 
Given a semimodule $M$ over $\Bool$, any $a\in M$ satisfies the idempotence property $a+a=a$, since $1+1=1$ in $\Bool$. Note that $M$ has the zero element $0$. A $\Bool$-semimodule is the same as an idempotented abelian monoid. 

$\Bool$-semimodules are also known as \emph{semilattices} or \emph{join semilattices} with \emph{the minimal element}, see~\cite{CC,HMS} and references therein. Namely, $M$ has a partial order with $a\le b$ if and only if $a+b=b$. Element $0\in M$ is the minimal element of the semilattice, and the symmetric, associative operation $(a,b)\mapsto a+b$ is taking the ``supremum" $\sup(a,b)$ of $a$ and $b$, usually written as $a\vee b$. Properties of $\vee$ are then 
\begin{itemize}
    \item $a\vee b=b\vee a$ and $(a\vee b)\vee c = a \vee (b\vee c)$ for $a,b,c\in M$, 
    \item $0 \vee a = a$ for $a\in M$. 
\end{itemize}
Such semilattices with $0$ are also called \emph{$\wedgezero$-semilattices}, see~\cite{GW99}. 
In this paper, by a \emph{semilattice} we mean a semilattice with $0$, thus a \emph{$\wedgezero$-semilattice}.  

$\Bool$-semimodule $M$ is naturally a semilattice under $+$ operation, with the minimal element $0$. 
Denote by $M^{\vee}$ or just $M$ the semimodule $M$ viewed as a semilattice.

A nonzero element $a\in M$ is called \emph{irreducible} if it cannot be represented as  a sum of two  elements neither one of which is $a$. Denote by $\primitive(M)$ the set of irreducible elements of $M$.

For a $\Bool$-semimodule $M$ and  
 $S\subset M$ a set of generators of $M$, there's a surjective homomorphism $\Bool^S\lra M$ from a free semimodule to $M$ taking generators of $\Bool^S$ to those of $M$.
A semimodule $M$ over $\Bool$ is called \emph{finitely-generated} if there exists a surjective semimodule homomorphism $\Bool^n\lra M$ for some $n$. Here $\Bool^n$ is the free $\Bool$-semimodule on $n$ generators, which we can denote $v_1, \ldots, v_n$. An arbitrary element of $\Bool^n$ can be written as 
\[  v_I = \sum_{i\in I} v_i, 
\]
for a unique subset $I\subset \{1,\ldots, n\},$ with the addition rule $v_I+v_J = v_{I\cup J}$. Alternatively,  elements of $\Bool^n$ can be written as column ``vectors'' with $n$ coefficients in $\Bool$. 
A $\Bool$-semimodule is finitely-generated if and only if it's a finite set. 

\begin{prop}
\label{prop:semimod-gen}
A finitely-generated $\Bool$-semimodule $M$ is generated by its set of irreducible elements $\primitive(M)$. 
Any set $V$ of generators of $M$ contains $\primitive(M)$.
\end{prop} 

\begin{proof}
 The second statement is clear. For the first statement, we only need to check the impossibility of loops in sum decompositions. Namely, if $a=b+c$ and $b=a+d$ in $M$, then 
 \[  a = b+c = b+b+c = b+ a + d + c = a + a + d = a+ d = b. 
 \] 
This proves the proposition. 
\end{proof}
Proposition~\ref{prop:semimod-gen} tells us that in a finitely-generated $\Bool$-semimodule $M$, the set $\primitive(M)$ is the unique minimal set of generators. 
 When $M$ is finitely-generated (equivalently, finite), for any surjection $\Bool^X\lra M$ there exists a finite subset $J\subset X$ with $|J|=|\primitive(M)|$ such that generators in $J$ are mapped bijectively to the elements of $\primitive(M)$. In particular, a minimal free covering $\Bool^J\lra M$ is unique, with $|J|=|\primitive(M)|$.

 \vspace{0.1in} 
 
 {\it Category of $\Bool$-semimodules.}
 A morphism $f:M\lra N$ of  $\Bool$-semimodules is map of sets such that $f(a+b)=f(a)+f(b)$ for $a,b\in M$ and $f(0)=0$.   
 
 Denote by $\Bmod$ the category of $\Bool$-semimodules and semimodule maps, and by $\Bfmod$ its full subcategory of finitely-generated (equivalently, finite) $\Bool$-semimodules. Hom sets in $\Bmod$ are naturally $\Bool$-semimodules. Endomorphisms $\End_{\Bmod}(M)$ of a $\Bool$-semimodule constitute an idempotented semiring.

 Note that the free semimodule $\Bool^n$ has few automorphisms: $\Aut_{\Bmod}(\Bool^n)\cong S_n$,  the permutation group on $n$ elements (automorphisms of a free semimodule are in bijection with permutations of irreducible elements). 
 More generally, for $M\in \Ob(\Bfmod)$, the group $\Aut(M)\subset S_{\primitive(M)}$ is usually a proper subgroup. Any automorphism permutes the set of irreducible elements, giving that inclusion of groups. Lack of automorphisms (compared to the linear case) is one feature of semimodules. 
 
 \vspace{0.05in} 
 
 A $\Bool$-semimodule is the same as a $\wedgezero$-semilattice, and the category $\Bmod$ is equivalent to the category whose objects are $\wedgezero$-semilattices and morphisms are maps taking $0$ to $0$ and intertwining the $\vee$-operation. Recall that a semilattice in this paper stands for a $\wedgezero$-semilattice. 
 
 \vspace{0.1in} 
 The dual semimodule of $M$ is given by $M^{\ast}:= \Hom_{\Bmod}(M,\Bool)$. Suppose that $M$ is finitely-generated. An element $f\in M^{\ast}$ is determined by the subset $N_f\subset M$ of elements that it takes to $0$. Let $a_f=\sum_{a\in N_f} a$, $a_f\in M$ be the sum of all elements of $N_f$. Then $N_f=\{b\in M|b\le a_f\}$. Thus, each element $a\in M$ gives rise to a functional $f_a\in M^{\ast}$ with $f_a(b)=0 \Leftrightarrow b\le a$. Vice versa, each functional $f\in M^{\ast}$ comes from a unique element $a_f\in M$. We obtain a bijection of sets $M \Leftrightarrow M^{\ast}$, for finite $\Bool$-semimodules $M$, so that $|M|=|M^{\ast}|$.  
 This bijection does not preserve the addition in $M$ and is order-reversing. It takes $0\in M$ to $\sum_{f\in M^{\ast}}f$, the largest element in $M^{\ast}$, and vice versa. 
 
 Furthermore, elements of $M^{\ast}$ separate elements of $M$. Namely, if $a< b$ (so that $a+b=b$) and $f(b)=0$ then $f(a)=0$. To such a pair $(a,b)$ we can assign homomorphism $g\in M^{\ast}$ such that $g(x)=0 \Leftrightarrow x\le a$. Then $g(a)=0, g(b)=1$, and $g$ separates $a$ and $b$. 
  Suppose now that $a\nleq b$.  Then $f_b\in M^{\ast}$ satisfies $f_b(b)=0, f_b(a)=1$. Thus, the natural pairing $M\times M^{\ast}\lra \Bool$ is nondegenerate (separating).  
 
 The above implies that the natural homomorphism $M\lra M^{\ast\ast}$ is an isomorphism in $\Bfmod$. We obtain the following result~\cite{St} and see~\cite[Proposition 2.3]{CC} for a generalization to arbitrary $\Bool$-semimodules. 
 
 \begin{prop} For any finite $\Bool$-semimodule $M$ the natural bijection of sets $M\cong M^{\ast}$, described above, is order-reversing. In the natural pairing $M\times M^{\ast}\lra \Bool$ elements of $M^{\ast}$ separate elements of $M$ (and vice versa), and the natural map $M\lra M^{\ast\ast}$ is an isomorphism of $\Bool$-semimodules. The assignment $M\mapsto M^{\ast}$ extends to a contravariant involution on $\Bfmod$.
 \end{prop} 
 
 \begin{remark}
 Finite $\Bool$-semimodules $M$ are the same as finite $\wedgezero$-semilattices. The above arguments show that a finite $\wedgezero$-semilattice $M$ is a lattice as well: define $a\wedge b =\sum_{c\in J} c$, where $J=\{ c| c\le a, c\le b\}$. The largest element of $M$ is $1:=\sum_{a\in M} a$, and the smallest is $0$. Thus, isomorphism classes of finite $\Bool$-semimodules are in a bijection with isomorphism classes of finite lattices (see monographs~\cite{Birk,Grat} for the theory of lattices). We denote this lattice enhancement of a finite semilattice $M$ by $M^{\vee}$ or just by $M$ when there's no possibility of confusion ($M^{\vee}$ is commonly denoted $\mathrm{Id}\, M$ in the literature). For infinite semilattices $\mathrm{Id}\, M$ is defined as the lattice of ideals of $M$, see~\cite[Section~3.15]{Grat}. 
 \end{remark}
 
 \begin{prop} A morphism $f:M\lra N$  in $\Bmod$ or $\Bfmod$ is injective (surjective) in the categorical sense if and only if the map of underlying sets is injective (surjective). 
 \end{prop} 
 
 This is shown, for instance, in~\cite[Proposition~2.7]{CC}. $\square$ 
 
 \vspace{0.1in} 
 
 Dualizing a minimal surjection $\Bool^n\lra M^{\ast}$ shows that any finite $M$ is a subsemimodule of a free $\Bool$-semimodule $\Bool^n$, where $n$ is the smallest number of generators of $M^{\ast}$. Let $M$ have $m$ generators (where we take the minimal set). Writing them in the basis of $\Bool^n$ as column vectors gives a $\Bool$-valued $n\times m$ matrix $\mathbb{M}$. This is a matrix of $0$'s and $1$'s. It has the following property:   
 \begin{equation}
 \label{eq:property-coln-row} 
 \begin{split}
 \mbox{ no column is a $\Bool$-semilinear combination of some other columns } \\
 \mbox{ and no row is a $\Bool$-semilinear combination of some other rows. }   
 \end{split}
 \end{equation} 
 Taking all $\Bool$-semilinear combinations of columns gives a semimodule isomorphic to $M$. Taking all $\Bool$-semilinear combinations of rows gives a $\Bool$-semimodule isomorphic to $M^{\ast}$. (It's easy to see that the row $\Bool$-semimodule is a subsemimodule of $M^{\ast}$. The fact that it is all of $M^{\ast}$ follows, for instance, from~\cite[Theorem 1.2.3]{Ki} and equality of cardinalities $|M|=|M^{\ast}|$.)  This is a rather explicit way to realize $M$ and its dual semimodule $M^{\ast}$ as subsemimodules of the free semimodules $\Bool^n$ and $\Bool^m$. Canonical pairing $M\times M^{\ast}\lra \Bool$ is given by $\mathbb{M}$ on irreducible elements of $M$ and $M^{\ast}$.  
 
 Below is an example of a matrix with property~\eqref{eq:property-coln-row}. We see that $n$ and $m$ may be different. In this example $M$ has 6 generators and $M^{\ast}$ has 4 generators; the columns are all possible elements of $\Bool^4$ with two 1's. 
 \[\begin{pmatrix}
1 & 1 & 1 & 0 & 0 & 0 \\
1 & 0 & 0 & 1 & 1 & 0 \\
0 & 1 & 0 & 1 & 0 & 1 \\
0 & 0 & 1 & 0 & 1 & 1 \\
\end{pmatrix}
\]

\begin{remark}
Semimodules $M,M^{\ast}$ are free if $n=m$ and $\mathbb{M}$ is a permutation matrix with a single 1 in each row and column.  For a matrix with property \eqref{eq:property-coln-row}, if $n\le 3$ or $m\le 3$ then $n=m$. Nonsquare matrices 
with \eqref{eq:property-coln-row} require $n,m\ge 4$. Square matrices with $n\le 3$ and property \eqref{eq:property-coln-row} are classified below (up to permutation of rows and columns and skipping the identity matrix): 
\begin{equation}\label{eq_4_matrices}
   A_1: \ \begin{pmatrix}
0 & 1 \\
1 & 1 
\end{pmatrix}, \quad \quad  A_2: \ 
\begin{pmatrix}
0 & 0 & 1 \\
0 & 1 & 0 \\
1 & 1 & 1 
\end{pmatrix}, \quad \quad  A_3: \ 
\begin{pmatrix}
0 & 0 & 1 \\
0 & 1 & 1 \\
1 & 1 & 1 
\end{pmatrix}, \quad \quad  A_4: \ 
\begin{pmatrix}
1 & 1 & 0 \\
1 & 0 & 1 \\
0 & 1 & 1 
\end{pmatrix}.
\end{equation}
Isomorphism classes of finite $\Bool$-semimodules $M$ are in a bijection with Boolean matrices $\mathbf{M}$ with property \eqref{eq:property-coln-row} up to permutations of rows and columns. Transposing the matrix gives the dual semimodule. See Figure~\ref{proj-nonfree-mod-01}. 

\input{proj-nonfree-mod-01}

\end{remark} 

 An arbitrary $n\times m$ $\Bool$-matrix $\mathbb{M}$ that does not necessarily satisfy \eqref{eq:property-coln-row} determines a pair of dual semimodules $M$ and $M^{\ast}$. Semimodule $M\subset \Bool^n$ consists of $\Bool$-semilinear combinations of columns of $\mathbb{M}$. Dual semimodule $M^{\ast}\subset \Bool^m$ consists of $\Bool$-semilinear combinations of rows of $\mathbb{M}$, and the matrix describes the pairing between $M$ and $M^{\ast}$. Any matrix $\mathbb{M}$ can be reduced to a unique, up to permutation of rows and columns, matrix with property $\eqref{eq:property-coln-row}$ by inductively removing rows and columns that are (semi)linear combinations of other remaining rows and columns. This reduction does change the associated semimodules $M$ and $M^{\ast}$. 
 
 \vspace{0.1in} 
 
 {\it Projective semimodules.}
 A $\Bool$-semimodule $P$ is called \emph{projective} if it has a lifting property: any surjective map $M\lra N$ induces a surjective map $\Hom(P,M)\lra \Hom(P,N)$. Applying this to a surjection $p:\Bool^J\lra P$ from a free $\Bool$-semimodule to $P$ shows there is a lifting map $\iota: P\lra \Bool^J$ into a free semimodule such that $p\circ\iota = \id$,
 \begin{equation}
     \Bool^J \stackrel{p}{\lra} P \stackrel{\iota}{\lra} \Bool^J, \ \  p\circ \iota = \id_P. 
 \end{equation}
 
 Vice versa, having a pair of maps $(\iota,p)$ with this property implies that $P$ is projective. We see that a $\Bool$-semimodule is projective if and only if it's a \emph{retract} of a free semimodule. 
 
 Not every such retract is a free semimodule. For instance, the semimodule $M=\{x,y|x+y=y\}$ associated to the matrix $A_1$ in \eqref{eq_4_matrices} is a projective but not a free semimodule. A surjection 
\[ p:\Bool^2\lra M, \ \  p(v_1)=x, \ p(v_2)=y
\]
admits a section 
\[ \iota: M\lra \Bool^2, \ \ \iota(x)=v_1, \ \iota(y)=v_1+v_2
,\] 
where $\{v_1,v_2\}$ is a basis of $\Bool^2$. Semimodules given by matrices $A_2$ and $A_3$ are projective as well, but not the semimodule described by $A_4$. 
 
 \vspace{0.1in} 
 
 Injective objects in $\Bmod$ and $\Bfmod$ are defined similarly. Duality exchanges projectives and injectives in $\Bfmod$. 
 Furthermore, a finite projective $\Bool$-semimodule $P$ is a retract of a finite free semimodule $\Bool^n$. Dualizing the maps $(\iota,p)$ gives a retraction $(p^{\ast},\iota^{\ast})$ for the dual semimodule $P^{\ast}$. Consequently, such $P$ is projective if and only if it's injective. We obtain (see also~\cite[Theorem~3.4]{HMS})
 
 \begin{prop} \label{prop_frobenius} Projective and injective objects in $\Bfmod$ coincide. They are exactly retracts of finite free $\Bool$-semimodules. 
 \end{prop} 
 
Note that a semimodule injective (or projective) in $\Bfmod$ is also injective (or projective) in the bigger category $\Bmod$. 

\begin{remark}
The semilinear category $\Bfmod$ of finite $\Bool$-semimodules is \emph{Frobenius} -- injective objects in it coincide with projectives. In particular, one can form the stable category $\Bfsmod$ of finitely-generated $\Bool$-semimodules. It's a quotient category of $\Bfmod$, and morphisms $f_0,f_1\in \Hom_{\Bfmod}(M,N)$ are equal in $\Bfsmod$ if there exist morphisms $g_0,g_1\in \Hom_{\Bfmod}(M,N)$ that factor through a free semimodule such that $f_0+g_0= f_1+g_1$. It's possible, however, that a better notion of the stable category should involve doubling of morphisms, similar to the constructions in~\cite{CC}. 
\end{remark}

As mentioned earlier, any finite semilattice with $0$ is necessarily a finite lattice, in a unique way. Retracts of free semimodules are distinguished among objects of $\Bfmod$ by the condition that the corresponding lattice is distributive, see~\cite[Section~3]{HMS}.   Birkhoff's representation theorem states that elements of any finite distributive lattice can be represented as finite sets,  so that the lattice operations correspond to unions and intersections of sets~\cite{Birk,Grat}. In particular, any finite distributive lattice can be realized as the lattice of open sets of a finite topological space, and all such lattices are distributive. 

The earlier proposition can be enhanced to the following statement, see \cite[Theorem~3.4]{HMS}.

\begin{prop}
\label{prop_many_char}
Let $P$ be a finite $\Bool$-semimodule. The following conditions are equivalent:
\begin{enumerate}
    \item $P$ is projective in $\Bfmod$. 
    \item $P$ is injective in $\Bfmod$. 
    \item $P$ is a retract of a free $\Bool$-semimodule. 
    \item The lattice $P^{\wedge}$ associated to $P$ is distributive. 
    \item  $P$ is the semilattice of open sets of a finite topological space (with $U+V:=U \cup V$). 
    \item Endomorphism semiring $\End(P)$ is generated, as a semimodule, by maps that factor through $\Bool$ (maps $P\lra \Bool \lra P$).  
\end{enumerate}
\end{prop}

\begin{remark} \label{remark_distributive} A $\wedgezero$-semilattice $M$ is called \emph{distributive} if for any $x\le a \wedge b$ there exist $a'\le a,b'\le b$ such that $x=a'\wedge b'$. Lemma~184 in~\cite[Section 5.1]{Grat}, restricted to finite semilattices, implies that a finite $\wedgezero$-semilattice $M$ is distributive if and only if the corresponding lattice $M^{\wedge}$ is distributive. This gives yet another characterization of projective finite $\Bool$-semimodules. 
\end{remark} 

Finite projective semimodules are further singled out in Proposition~\ref{prop_proj_coev} and Corollary~\ref{cor_proj_rigid} below.

\begin{remark} Finite topological spaces naturally appear from stratified spaces. A stratified space $S$ with finitely many strata gives rise to a finite space $X(S)$ whose points are the  strata of $S$ and open sets - unions of strata which are open in $S$. Finite space $X(S)$ is just the quotient of $S$ by the equivalence relation given by the decomposition into strata. Note that a stratified space gives rise to the categories of constructible sheaves and perverse sheaves on it. O.~Viro discussed finite topological spaces in the second part of his talk ``Compliments to Bad Spaces''~\cite{Viro}.
\end{remark} 

{\it Tensor product.}
The hom space $\Hom_{\Bmod}(M,N)$ of two semimodules is naturally a semimodule, so that categories $\Bmod$ and $\Bfmod$ have internal homs. Tensor product $M\otimes N$ of $\Bool$-semimodules is defined to consist of finite sums $\sum_i m_i\otimes n_i$ modulo the equivalence  relation generated by 
\begin{eqnarray*}
0 \otimes n & \sim & 0 \ \  \sim \ \ m\otimes 0, \\
(m_1+m_2)\otimes n & \sim &  m_1\otimes n + m_2\otimes n, \\
m\otimes (n_1+n_2) & \sim &  m\otimes n_1 + m \otimes n_2,
\end{eqnarray*}
where $0$ is the empty sum. To define the tensor product more intrinsically (see~\cite{GLQ}), call a map $f:M\times N \lra K$, for $K\in \Ob(\Bmod)$,   \emph{bisemilinear} if for each $m\in M$ and $n\in N$ the maps 
$f(m,-)\in \Hom_{\Bmod}(N,K)$ and $f(-,n)\in \Hom_{\Bmod}(M,K)$. 

The tensor product $M\otimes N$ is determined by the property that 
there exist a bisemilinear map $\phi: M\times N \lra M\otimes N$ such that any bisemilinear map $f:M\times N\lra K$ into a $\Bool$-semimodule $K$ extends to a unique semimodule homomorphism $g:M\otimes N\lra K$, $f=g\phi$.

 In addition to~\cite{GLQ}, we also refer the reader to related papers on tensor products of semilattices~\cite{BM78,AK78, GW99,GW00} and references there. 
(Fraser~\cite{Fra76} defined and studied the tensor product in the related category of semilattices without zero.) 

\vspace{0.05in} 

Tensor product $\otimes$ turns $\Bmod$ into a symmetric monoidal category with the identity object $\one = \Bool$. There is a natural isomorphism 
\begin{equation} \label{eq_adjunction} \Hom_{\Bmod}(M\otimes N,K)\cong \Hom_{\Bmod}(M, \Hom_{\Bmod}(N,K))
\end{equation}
making the tensor product functor left adjoint to the internal homs functor. Together with the dualizing object $\Bool$, this makes $\Bfmod$ into a $\ast$-autonomous category~\cite{Bar79,Bar91} (these references describe $\star$-autonomous structure on the larger category of complete semilattices). For finite $M,N$ there are also natural isomorphisms
 \begin{equation} \label{eq_other}
 M\otimes N \cong \Hom_{\Bmod}(M,N^{\ast})^{\ast}, \hspace{4mm}   
 \Hom_{\Bmod}(M,N)\cong (M\otimes N^{\ast})^{\ast},
 \end{equation} 
 obtained by setting $K=\Bool$ in \eqref{eq_adjunction} and dualizing. 
 
\vspace{0.1in}

{\it Failure of duality.}
There is a natural homomorphism 
\begin{equation}\label{map_not_iso}
    \psi_{M,N}\ : \ M^{\ast}\otimes N\lra \Hom(M,N), \ \ \ \  (m^{\ast}\otimes n)(m') \:= m(m')n, \ \ m'\in M, m^{\ast}\in M^{\ast},n\in N.
\end{equation}

An important point is that, in general, 
the map $\psi_{M,N}$ is not an isomorphism even for finite $M,N$. A counterexample is furnished by taking $M=N$ to be the semimodule given by the matrix $A_4$ in \eqref{eq_4_matrices}. Equivalently, $M=N=\{x,y,z|x+y=x+z=y+z\}$. Then, under \eqref{map_not_iso}, no element of $M^{\ast}\otimes M$ is mapped to the identity endomorphism of $M$, which can be checked directly. Consequently, the category $\Bfmod$ is not compact closed with respect to the above maps $\psi_{M,N}$.

Via isomorphisms \eqref{eq_other}, failure of $\psi_{M,N}$ to be isomorphisms is equivalent to 
natural maps below not being isomorphisms either, for general $M,N$.
\begin{equation}\label{eq_dual_not_iso} 
 \ M^{\ast}\otimes N^{\ast}  \lra (N\otimes M)^{\ast}, \ \ 
 \Hom(M, N)^{\ast}  \lra \Hom(M^{\ast},N^{\ast}). 
\end{equation}
Barr~\cite[Section 6]{Bar79} shows the absence of any compact closed structure on this monoidal category by checking that the cardinalities of sets 
$\Hom(M,N)$ and $\Hom(M^{\ast},N^{\ast})$ are different for two particular finite semilattices $M,N$ (a finite semilattice has the same cardinality as its dual). 

\begin{prop}[Barr~\cite{Bar79}] 
 The category $\Bfmod$ is a $\ast$-autonomous category but not a compact closed category (not a rigid symmetric monoidal category). 
\end{prop}

{\it Projective semimodules and coevaluation maps.}
A finite projective $\Bool$-semimodule $P$ is a retract of a free semimodule $\Bool^n$, via the maps $P\stackrel{\iota}{\lra}\Bool^n\stackrel{p}{\lra}P$ with $p\circ\iota = \id_P$.  Form the diagonal 
\begin{equation}
    d = \sum_{i=1}^n e_i \otimes e_i^{\ast}\in \Bool^n\otimes (\Bool^n)^{\ast},
\end{equation}
where $(e_1,\ldots, e_n)$ is the standard basis of $\Bool^n$. Composing it with maps $p:\Bool^n\lra P$ and $\iota^{\ast} :(\Bool^n)^{\ast}\lra P^{\ast}$, respectively, gives an element 
\begin{equation}\label{eq_coev_map}
    \coev_{P} \ := \ \sum_{i=1}^n p(e_i) \otimes \iota^{\ast}(e_i^{\ast})\in P \otimes P^{\ast}. 
\end{equation}
Relation $p\circ\iota = \id_P$ together with the definition of the dual morphism imply that $\coev_{P}$ satisfies the isotopy relation in Figure~\ref{cupcap} (with $X$ replaced by $P$ there), and likewise for $P^{\ast}$ and $\coev_{P^\ast}$ in place of $P$ or $\coev_P$.  The ``cap'' morphism  $\ev_P: P^{\ast}\otimes P\lra \Bool$ is the usual evaluation map $\ev_P(y^{\ast}\otimes y)=y^{\ast}(y).$

\begin{prop}\label{prop_proj_coev} For a finite projective $\Bool$-semimodule $P$ coevaluation and evaluation maps $\coev_P$ and $\ev_P$ satisfy the isotopy relations in Figure~\ref{cupcap}.
 Map $\psi_{P,P}:P^{\ast}\otimes P\lra \End(P)$ given by \eqref{map_not_iso} is an isomorphism.  The identity $\id_P$ is in the image of $\psi_{P,P}$.
\end{prop} 

Consequently, endofunctors of tensoring with $P$ and $P^{\ast}$ in the category $\Bmod$ are adjoint (even biadjoint due to the symmetric structure). Adjointness and biadjointness property fails if we substitute $P$ by a finite non-projective semimodule $M$. To see this, assume given a coevaluation map 
\begin{equation}\label{coev_open_sets_3} 
    \coev_M \ : \ \Bool \lra M\otimes M^{\ast}, \ \ 
    \coev_M(1) = \sum_{i=1}^n a_i \otimes b_i.
\end{equation}

Elements $b_i\in M^{\ast}$ are maps $M\lra\Bool$ which assemble into a semimodule map $b:M\lra \Bool^n$. Elements $a_i\in M$ give a map $a:\Bool^n\lra M$ with $b(e_i)=a_i$. Isotopy relation on $\ev$ and $\coev$ says that 
\begin{equation}
    m\longmapsto \sum_i a_i \otimes b_i\otimes m \lra \sum_i b_i(m) a_i  = m ,  
\end{equation}
that is, $a\circ b = \id_M$. Consequently, $b,a$ realize $M$ as a retract of $\Bool^n$, and $M$ is a projective $\Bool$-semimodule. 

One can further assume that $a_1,\ldots, a_n$ in \eqref{coev_open_sets_3} constitute the set of distinct irreducible elements of $M$. This gives a formula for $\coev_M$ with the smallest possible $n$ (the cardinality of $\irr(M)$). 

We see that, at least for the standard evaluation map $\ev_M$, coevaluation map exists if and only if $M$ is a finite projective $\Bool$-semimodule.

\vspace{0.1in} 

 For a finite $\Bool$-semimodule $M$ denote by $\Bmodo{M}$ the full subcategory of $\Bmod$ monoidally generated by $M$ and $M^{\ast}$ (so that objects of $\Bmodo{M}$ are tensor products of $M$ and its dual).   

\begin{corollary} \label{cor_proj_rigid} For a finite projective $\Bool$-semimodule $P$ the category $\Bmodo{P}$ is a rigid symmetric monoidal category.  
\end{corollary} 
The rigid structure is given by the evaluation and coevaluation maps above and their tensor products. $\square$ 

\vspace{0.1in}

{\it Finite topological spaces and projective semimodules.}
It follows from Birkhoff's theorem that a finite projective $\Bool$-semimodule $P$ can be realized as the semilattice $\mc{O}(X)$ of open sets of a finite topological space $X$.  Assume that $X$ is minimal, in the sense that for any two points $x_1,x_2\in X$ there exists an open set that contains only one of these two points. For an open set $U$ denote by $[U] $ the corresponding element of $P$. Each $x\in X$ determines the smallest open set $U_x$ that contains $x$ (the intersection of all open sets that contain $x$). Sets $U_x$, $x\in X$,  are pairwise distinct and constitute the set of smallest non-empty open subsets of $X$. 

Elements $[U_x],x\in X$ are exactly the irreducible elements of the semilattice $\mc{O}(X)$, and $X$ can be reconstructed from $P$ canonically, by taking its set of points as the set of irreducible elements $\irr(P)$ of $P$. The smallest open set that contains $x\in \irr(P)$ consists of all $y\in \irr(P)$ such that $x+y=x$. 

The dual $\Bool$-semimodule  $P^{\ast}\cong\mc{O}(X)^{\ast}$ is isomorphic to the lattice of all closed sets in $X$, with the union of closed sets as the addition operation. (Any closed set $V\subset X$ defines a semilinear map $P\lra \Bool$, taking $[U]$ to $1$ if and only if $U\cap V\not= \emptyset$. Since $P$ and $P^{\ast}$ have the same cardinality, this describes all elements of $P^{\ast}$.) Equivalently, $ \mc{O}(X)^{\ast}$ is the lattice of open sets of the dual finite space $X^{\ast}$.  Let $V_x$ be the smallest closed set that contains $x$ and denote by $[V_x]$ the corresponding element of $P^{\ast}$. Note that $U_x\cap V_x=\{x\}$ since we chose $X$ minimal. Form the element 
\begin{equation}
    c_P \ : \ \sum_{x\in X} [U_x]\otimes [V_x] \in P \otimes P^{\ast} 
\end{equation}
and consider the corresponding homomorphism
\begin{equation}\label{coev_open_sets} 
    \coev_P \ : \ \Bool \lra P\otimes P^{\ast}.
\end{equation}
 It is straightforward to check that $\coev_P$ satisfies the isotopy relation in Figure~\ref{cupcap}. The map $\coev_P$ is canonically associated to $P$, since minimal $X$ is unique, up to homeomorphism, given $P$.

\vspace{0.1in} 

For two finite projective $\Bool$-semimodules $P_1,P_2$ with associated topological spaces $X_1,X_2$ the tensor product $P_1\otimes P_2$ has associated topological space $X_1\times X_2$. This implies compatibility of coevaluation maps 
$\coev_{P_1\otimes P_2}=\coev_{P_1}\otimes \coev_{P_2}$. Evaluation maps are naturally compatible with the tensor product. Furthermore, $
\coev_{P^{\ast}}= \sigma \circ \coev_P$, where $\sigma:P\otimes P^{\ast}\lra P^{\ast}\otimes P$ is the transposition map.  

\vspace{0.1in} 

{\it Rigid monoidal structure.}
Denote by $\Bfpmod$ the full monoidal subcategory of $\Bfmod$ with finite projective semimodules as objects.

\begin{prop}\label{prop_proj_rigid} The category $\Bfpmod$ is  a rigid symmetric monoidal category.  
\end{prop}
This follows from compatibility of evaluation and coevaluation maps with the tensor product of projective semimodules.
$\square$

\vspace{0.05in} 

 Note that a rigid symmetric monoidal category is the same as a compact closed category. 

\vspace{0.05in} 

This category can be interpreted diagrammatically (see Section~\ref{subsec_pairing}).  A closed circle labelled by  any nonzero $P$ evaluates to $1$, since $\ev_{P^{\ast}}\circ \coev_P(1)=1$ (equivalently, $\dim(P)=1\in \Bool$, for any $P\not= 0$). 

\begin{remark}
  Here is a similar example in the linear world. Let $R$ be a commutative ring which is an integral domain. 
  The category $R\pfmod$ of finitely-generated projective $R$-modules is a rigid symmetric category, with the usual tensor product $P\otimes_A Q$ of modules and the dual module $P^{\ast}=\Hom_A(P,A)$. 
  The rigid structure is lost if one passes to the larger category of all finitely-generated $R$-modules. 
  Instead, one can work with the homotopy category of bounded complexes of objects in $R\pfmod$ to keep the duality and the rigid structure. 
\end{remark}

Tensor product and internal homs turn category $\Bmod$ into a cartesian closed category. As explained earlier, restricting to the full subcategory of finite projective semimodules $\Bfpmod$ gives a rigid symmetric monoidal category, while such structure is absent on the larger category of all finite $\Bool$-semimodules. 

\vspace{0.05in} 

We refer the reader to~\cite[Sections 2,3]{CC} and \cite{HMS} for more information on  $\Bool$-semimodules and to~\cite{Ki} for the theory of Boolean matrices.

{\it Reduced tensor product.} 
 Given an injective morphism $N_0\stackrel{f}{\lra} N_1$, tensoring it with $M$ results in a morphism 
 \begin{equation*} 
 M\otimes N_0\stackrel{\id\otimes f}{\lra}M\otimes N_1
 \end{equation*} 
 which is not injective, in general~\cite{GW99}. 
 A $\Bool$-semimodule $M$ is called \emph{flat} if the above morphism $\id\otimes f$ is injective for any injective $f$.

 Consider finite lattices $M_3$ and $N_5$ depicted in Figure~\ref{figure03-02} and underlying $\wedgezero$-semilattices, which we also denote by $M_3$ and $N_3$. (They are given by forgetting part of the structure, the $\mathrm{inf}$ operation, in a lattice.) 

 \input{figure03-02}
 
 Let $\iota: M_3\hooklra \Bool^3$ and $\iota': N_5 \hooklra \Bool^3$ be embeddings of these semilattices into the free rank three semilattice given by 
 
 \[
 \bordermatrix{
 & x_1 & x_2 & x_3 \cr
 & 1 & 1 & 0 \cr
 & 1 & 0 & 1 \cr
 & 0 & 1 & 1 \cr 
 }
 \hspace{1cm}
 \mbox{ and } 
 \hspace{1cm} 
 \bordermatrix{
 & a & b & c \cr
 & 1 & 0 & 1 \cr 
 & 0 & 1 & 0 \cr 
 & 1 & 1 & 0 \cr 
 },
 \] 
respectively.  
 
 
 Gr\"atzer and Wehburg~\cite{GW99} show the following.  
 \begin{itemize}
     \item 
 If the map $\iota\otimes \id_M: M_3\otimes M \lra \Bool^3\otimes M$ is injective, lattice $ M$ does not contain $M_3$ as a sublattice.
 \item 
 If the map $\iota'\otimes \id_M: N_5\otimes M \lra \Bool^3\otimes M$ is injective, lattice $ M$ does not contain $N_5$ as a sublattice.
  \end{itemize}
  A lattice is distributive if and only if it does not contain any sublattices isomorphic to $M_3$ or $N_5$. This leads to the following result.  
  
  \begin{theorem}[Gr\"atzer-Wehburg~\cite{GW99}] 
  Let $M$ be a $\wedgezero$-semilattice. The following are equivalent: 
  \begin{enumerate}
      \item $M$ is flat. 
      \item Both homomorphisms $ \iota\otimes \id_M$ and $ \iota\otimes \id_M$ are injective. 
      \item $M$ is a distributive semilattice. 
  \end{enumerate}
  \end{theorem}
  
Recall the notion of distributive semilattice from Remark~\ref{remark_distributive}.
For a finite $\wedgezero$-semilattice $M$ distributivity means that $M^{\wedge}$ is a distributive lattice, or, equivalently, that $M$ is a projective semilattice (projective $\Bool$-semimodule), see Proposition~\ref{prop_many_char}. 
 
Inclusions of finite semimodules into free semimodules 
\[ 
M\hooklra \Bool^{S_1}, \hspace{0.5in}  N\hooklra \Bool^{S_2}
\]
 induce a map 
\begin{equation}\label{eq-phi-inj}
    M\otimes N \stackrel{\phi}{\lra} \Bool^{S_1\times S_2}.
\end{equation}
We see that $\phi$ is not always injective.  Gr\"atzer-Wehburg~\cite{GW99} implies the following result. 
\begin{prop}\label{prop_inject}
 Map $\phi$ above is injective if $M$ or $N$ is a projective semimodule. 
\end{prop}

For finite $\Bool$-semimodules $M,N$, choose inclusion $M\hooklra P_1, N\hooklra P_2$ into finite projective semimodules (these are the same as finite injective semimodules). Consider the induced map 
\begin{equation}\label{eq-phi-inj-02}
    M\otimes N \stackrel{\phi}{\lra} P_1\otimes P_2
\end{equation}
and define the \emph{reduced tensor product} $M\undotimes N$ as the image of $\phi$. Coincidence of finite projective and injective semimodules implies that the reduced tensor product is well-defined. It is also isomorphic to the image of $M\otimes N$ in $M\otimes P_2$ for any inclusion $N\subset P_2$ into a finite projective semimodule, and likewise isomorphic to the image of $M\otimes N$ in $P_1\otimes N$ for inclusions $M \subset P_1$. 

Reduced tensor product comes with a canonical surjective map
\begin{equation} \label{eq_tensor_two} M\otimes N \lra M\undotimes N.
\end{equation} 
This map is an isomorphism if either $M$ or $N$ is a projective semimodule. 

\vspace{0.1in} 

\begin{remark} {\it Semilattices and urban planning.} Remarkably, semilattices made a splash in architecture and urban planning with the 1965 paper ``A city is not a tree'' by C.~Alexander, who argued that part of a structure of a city is that of a semilattice. This widely cited and influential paper, together with follow-up work, was republished as a book fifty years later~\cite{Al}.  

\end{remark} 


\subsection{Categories from an  evaluation} 
Let $\Sigma$ be a finite set, $L_I\subset \Sigma^{\ast}$ be a language, and $L_{\circ}\subset \Sigma^{\ast}_{\circ}$ be a circular language. Write this pair of languages as $L=(L_I,L_{\circ})$. Pairs $L$ are in a bijections with evaluations $\alpha$ of closed $0$-manifolds with values in $\Bool$, as explained at the end of Section~\ref{subsec_valued}.
Denote the evaluation associated to $L$ by $\alpha_L$ and the pair associated to an evaluation $\alpha$ by $L_{\alpha}$. 

\vspace{0.1in} 

To $L$ and $\alpha$ we assign several categories, similar to~\cite{KS3,KQR,KKO,Kh3}. 

(1)  Start with the category $\CCC_{\Sigma}$ of $\Sigma$-decorated one-dimensional cobordisms, with inner (floating) boundary allowed. Objects of this category are finite sequences $\underline{\varepsilon}$ of pluses and minuses (oriented $0$-manifolds). Morphisms are such decorated cobordisms, viewed up to decoration-preserving rel boundary diffeomorphisms, see an example in Figure~\ref{generalized-ex-01}. 

\vspace{0.05in} 

(2) Consider the category $\CCC'_{\alpha}$ with the same objects as in $\CCC_{\Sigma}$, but now $\Bool$-semilinear combinations of morphisms in $\CCC_{\Sigma}$ are allowed and any floating component $U$ of a cobordism is evaluated to $\alpha(U)\in \Bool$. Each diagram in $\CCC'_{\alpha}$ reduces to a diagram with each component having at least one top or bottom boundary point (and at most one or none floating boundary points), and a morphism in $\CCC'_{\alpha}$ is a $\Bool$-semilinear combination of such diagrams. 

\vspace{0.05in} 

(3) Next we pass to the quotient of $\CCC'_{\alpha}$ by the universal construction, as described earlier in Section~\ref{subsec_rigid}. Denote the resulting category by $\CCC_{\alpha}$. Following terminology in~\cite{KQR,KKO,Kh2}, we may call  it the gligible quotient category. 

\vspace{0.05in} 

 In this paper $\CCC_{\alpha}$ is our primary category. By analogy with~\cite{KS3,KKO} and other papers, one may want to pass to its Karoubi envelope or allow finite direct sums of objects, etc. It is not clear to us what the natural completion of $\CCC_{\alpha}$ would be, since Karoubi envelope or additive closure are especially useful for additive categories, while $\CCC_{\alpha}$ is only semiadditive.

Categories  $\CCC_{\Sigma},\CCC_{\alpha}'$ and $\CCC_{\alpha}$ are rigid symmetric monoidal categories. Categories $\CCC_{\alpha}'$ and $\CCC_{\alpha}$ are semiadditive, in the sense that hom spaces in these categories are $\Bool$-semimodules and composition of homs is (semi)bilinear. 

\vspace{0.1in} 

We say that  a pair $L$ of languages is $\CCC$-\emph{regular} if hom spaces in the category $\CCC_\alpha$ are finite (equivalently, hom spaces are finitely-generated $\Bool$-semimodules). 

\begin{prop} \label{prop_when_regular} 
$L=(L_I,L_{\circ})$ is $\CCC$-regular if and only if languages $L_I$ and $L_{\circ}$ are regular. 
\end{prop} 

\begin{proof} It's easy to see that $L$ is $\CCC$-regular if and only if spaces $A(+),A(-)$ and $A(+-)$ are finitely-generated (equivalently, finite) $\Bool$-semimodules, for this implies finiteness of  $A(\varepsilon)$ for any sign sequence $\varepsilon$. $\Bool$-semimodules $A(+),A(-)$ are finite if and only if the interval language $L_I$ is regular. $\Bool$-semimodule $A(+-)$ is finite if and only if both languages $L_I$ and $L_{\circ}$ are regular.
\end{proof}

\vspace{0.1in} 

%
%

\section{State spaces,  automata, and generalizations}
\label{sec_sspaces} 


\subsection{State spaces \texorpdfstring{$A(+)$}{A(+)}, \texorpdfstring{$A(-)$}{A(-)} and finite state automata} \label{subsec_plus_minus} 

Given a regular $\alpha$,  consider state spaces  $A(+)=A_{\alpha}(+)$ and $A(-)=A_{\alpha}(-)$. These are dual finite $\Bool$-semimodules, with the nondegenerate pairing 
\[  (\hspace{2mm},\hspace{2mm}) \ : \ A(-) \times A(+) \lra \Bool 
\]
given on the spanning sets by 
\[ (\langle \omega |, |\omega'\langle ) = \alpha_{I}(\omega\omega'), \hspace{4mm} \omega,\omega'\in \Sigma^{\ast},
\]
see Figure~\ref{state-space-01}. 

\input{state-space-01}

\vspace{0.1in} 
These spaces depend only on the interval language $L_I$ of $\alpha$ and not on the circle language $L_{\circ}$. 
$\Bool$-semialgebra $\Bool[\Sigma^{\ast}]$ acts on $A(-)$ on the right, via 
\begin{equation}\label{eq_Aplus_act} 
\langle \omega | a = \langle \omega a|, \ \  a\in \Sigma,
\end{equation} 
and on $A(+)$ on the left, via $
a | \omega \langle  = |a \omega \langle$. We denote the action of $a$ by 
\[m_a:A(-)\lra A(-).
\] 
The (semi)bilinear form intertwines the two actions, 
\[ (u x, v)  = (u, xv)=\alpha_I(uxv), \ \  u\in A(-), \ v\in A(+), \  x \in \Bool[\Sigma^{\ast}].
\] 

\vspace{0.1in} 

Let 
\begin{equation}
Q^- \ := \ \{\langle \omega | , \omega \in \Sigma^{\ast}\}\subset A(-)
\end{equation} 
be the subset of $A(-)$ consisting of images of all words rather then their $\Bool$-semilinear combinations. We say that an element $q\in Q^-$ is \emph{acceptable} if $\alpha_I(q)=1$. Denote by $Q_{\t}^-\subset Q^-$ the set of acceptable states, and denote by $q_{\init} = \langle \emptyset |$ the distinguished element for the empty word. Set $Q^-$ comes with a natural right action \eqref{eq_Aplus_act} of the free monoid $\Sigma^{\ast}$ that we denote $\delta: Q^-\times \Sigma\lra Q^-$. The following proposition is immediate. 

\begin{prop}\label{prop_min_dfa}  $(Q^-,\delta, q_{\init},Q^-_{\t})$ is the unique minimal deterministic automaton for the interval language $L_I$ of $\alpha$.
\end{prop} 

\begin{proof} Suppose that  $D=(Q,\delta,q_{\init},Q_{\t})$ is a deterministic FSA that accepts $L_I$ and assume, furthermore, that $D$ is \emph{trim}~\cite{E1}, that is, any state of $D$ is on a path from the initial state $q_{\init}$ to one of the accepting states. (An automaton for a given language with the minimal number of states is necessarily trim.) 

There is a surjective map $s:\Sigma^{\ast}\lra Q$ sending a word $\omega$ to $ q_{\init} \omega\in Q$. For each $q\in Q$ pick a word $\omega(q)\in \Sigma^{\ast}$ mapped to it by  this map. 
This word determines a $\Bool$-functional on $A(+)$ given by
$\alpha_I(\omega(q)\omega),$ $\omega \in \Sigma^{\ast}$. 
The functional depends only on $q$. Consequently, there is a map $u: Q\lra A(+)^{\ast}\cong A(-)$ taking a state $q$ to that functional. 
The composition $us:\Sigma^{\ast}\lra A(-)$ is necessarily the map $us(\omega)= \langle \omega |$. Likewise, the map $u$ factors through the inclusion $Q^-\subset A(-)$, $u: Q\stackrel{p}{\lra} Q^-\hookrightarrow A(-)$, for a unique $p$. Map $p$ is surjective, takes initial state to the initial state and intertwines the action of $\Sigma^{\ast}$ on $Q$ and $Q^-$. Consequently, any deterministic trip automaton for $L_I$ surjects onto the automaton $(Q^-,\delta, q_{\init},Q^-_{\t})$, and the latter is the unique deterministic automaton for $L_I$ with the minimal number of states. 
\end{proof} 

Note that it's possible for the set $Q^-$ to be all of $A(-)$, but, typically, it's a proper subset. For instance, the element $0\in A(-)$ is in $Q^-$ if and only if $L_I$ contains ``unrecoverable'' words $\omega$, that is, words such that $\omega\omega'\notin L_I$ for any word $\omega'$. Then necessarily $\langle \omega | = 0 \in A(-)$, and the inverse implication holds as well.  
\begin{remark} Subset $Q^+=\{|\omega\langle, \omega\in\Sigma^{\ast}\}$ of $A(+)$  naturally provides the minimal deterministic automaton for the opposite language $L_I^{\op}$. 
\end{remark}

Finite $\Bool$-semimodule $A(-)$ has a unique minimal set of generators $J\subset A(-)$. Necessarily $J$ consists of all irreducible elements, $J=\irr(A(-))$. Consider the induced surjection 
\begin{equation} 
\psi_J:\Bool^J\lra A(-)
\end{equation} 
from a free semimodule onto $A(-)$. 
The free monoid of words $\Sigma^{\ast}$ and the associated $\Bool$-semiring $\Bool[\Sigma^{\ast}]$ act on $A(-)$. We can lift this action to an action of $\Sigma^{\ast}$ on $\Bool^J$. For each letter $a\in \Sigma$ choose a lifting of the semimodule endomorphism $m_a$ to an endomorphism $\widetilde{m}_a$ of $\Bool^J$, so that the diagram below commutes: 
\vspace{0.1in}
\[ 
\xymatrix@-1pc{
\Bool^J \ar[rrr]^{\widetilde{m}_a} \ar[ddd]_{\psi_J} & & & \Bool^J \ar[ddd]^{\psi_J.} \\
& & & \\
& & & \\ 
A(-) \ar[rrr]^{m_a} & & & A(-)\\ 
}
\]

\vspace{0.1in}

The map $m_a$ is determined by its action on the elements of $J$. The lifting is determined by choosing  $\widetilde{m}_a(j)\in \psi^{-1}_J(m_a(j))\subset \Bool^J$ for each $a\in \Sigma$ and $j\in J$. If for some $a$ and $j$ the subset $\psi^{-1}_J(m_a(j))$ of $\Bool^J$ has cardinality greater than one, a lifting is not unique. 

Also choose a lifting $\widetilde{\langle \emptyset |}$ of the empty word element $\langle \emptyset |\in A(-)$ to $\Bool^J$. Again, such a lifting is unique if and only if the set $\psi^{-1}(\langle \emptyset |)$ has cardinality one. 

Define a  homomorphism $\alpha^J: \Bool^J\lra \Bool$  as the composition \[
\alpha^J\ :=\ \alpha_I\circ\psi_J \ : \ \Bool^J \stackrel{\psi_J}{\lra} A(-) \stackrel{\alpha_I}{\lra} \Bool 
\] of the projection onto $A(-)$ and evaluation 
\begin{equation}\label{eq_trace}
\alpha_I:A(-)\lra \Bool, \ \ \alpha_I(\langle \omega |) = \alpha_I(\omega).
\end{equation} 

\begin{prop} Any lifting as above gives rise to a minimal nondeterministic automaton $(J,\widetilde{\langle \emptyset |},\widetilde{m}, \alpha^J)$ for the language $L_I$. The set of such  liftings is in a bijection with the isomorphism classes of minimal nondeterministic automata for $L_I$. 
\end{prop} 

\begin{proof} The nondeterministic automaton has $J$ as the set of states. The set of initial states is $\widetilde{\langle \emptyset |}$ (one initial state if $\widetilde{\langle \emptyset |}\in J$, which is a possible choice for the lifting if and only if $\langle \emptyset |\in J$). In the automaton, letter $a$ takes the state $j$ to the subset $\widetilde{m}_a(j)$ of $J$ (elements of $\Bool^J$ are in a bijection with subsets of $J$). In this way the map $\delta: \Sigma \times J \lra P(J)$ is described by $\delta(a,j)= \widetilde{m}_a(j) \subset J$. A state $j\in J$ is accepting if and only if $\alpha^J(j)=1$. 

\vspace{0.1in} 
 
Going the other way, suppose given a  nondeterministic automaton $F=(Q,\delta,Q_{\init},Q_{\t})$ for $L_I$ and assume that any state in $Q$ belongs to a path from one of the initial states to an acceptable state. To this automaton we associate a free $\Bool$-semimodule $\Bool^Q$ with an action of $\Sigma^{\ast}$ given by $\delta$, initial element $\langle \emptyset |_Q := \sum_{q\in Q_{\init}} q$, and trace map $\tr_Q(q)=1 \Leftrightarrow q\in Q_{\t}$. 

\vspace{0.05in} 

To relate $\Bool^Q$ and $A(-)$, 
for a $q\in Q$ denote by $W(q)$ the set of all words
 that can take from an initial state to $q$: 
\[ W(q) \ := \ \{\omega\in \Sigma^{\ast}| q\in \omega\langle \emptyset |_Q  \} . 
\]
Here $\omega\langle \emptyset |_Q$ is the subset of $Q$
 given by applying $\omega$ to the initial element $\emptyset |_Q$. In this construction we're using the bijection between elements of $\Bool^Q$ and subsets of $Q$. The set $W(q)$ might be infinite, but  sum $\sum_{\omega\in W(q)} \langle \omega |$ is a well-defined element of $A(-)$ (any countably infinite sum in a finite semilattice eventually stabilizes).   

Define a homomorphism of $\Bool$-semimodules $\phi:\Bool^Q\lra A(-)$ by $\phi(q)=\sum_{\omega\in W(q)} \langle \omega |$ and extending $\Bool$-semilinearly. Both $\Bool^Q$ and $A(-)$ come with the pairing map:  
\[ \Bool^Q\times A(+)
\lra \Bool, \ \ A(-)\times A(+)\lra \Bool, 
\] 
and $\phi$ intertwines the two pairings with $A(+)$. Map $\phi$ is surjective, since $A(-)\cong A(+)^{\ast}$ and $\Bool^Q$ maps surjectively onto the dual of $A(+)^{\ast}$, since automaton $F$ recognizes $L_I$. 

A surjective homomorphism of finite semilattices induces a surjection on their irreducible elements $\irr(\Bool^Q)=Q\twoheadrightarrow J=\irr(A(-))$. Thus, $|Q|\ge |J|$ and equality means that $\Bool^Q\cong \Bool^J$. Consequently,  all nondeterministic automata for $L_I$ with the minimal number of states are constructed via the above minimal free cover $\Bool^J\lra A(-)$ and by lifting the action of $\Sigma^{\ast}$ and $\langle \emptyset |$ to the minimal free cover. 
\end{proof} 

This correspondence also shows when a minimal nondeterministic automaton for a language is unique -- when there's a unique lifting of the action of $\Sigma$ on $A(-)$ to an action on the minimal free cover $\Bool^J$ and a unique lifting of element $\langle \emptyset |\in A(-)$ to $\Bool^J$. A sufficient but not necessary condition for uniqueness is that $A(-)$ is a free $\Bool$-semimodule. 

\begin{remark}
If we consider nondeterministic automata with a single initial state $q_{\init}$ and $\langle \emptyset | \notin J$, we can instead take $J\sqcup \{ \widetilde{\langle \emptyset |}\}$ as the set of states of the minimal automaton of that more restrictive type, the corresponding free $\Bool$-semimodule or rank $|J|+1$ and its surjection onto $A(-)$. Adding generators to the free semimodule leads to more possible liftings of the action of $\Sigma$ on it, of course.  
\end{remark}

We see that the state space $A(-)$ together with the action of $\Sigma^{\ast}$ on it and nondegenerate inner product $A(-)\otimes A(+)\lra \Bool$ given by concatenation and then evaluation via $\alpha$ gives us a {\it unified viewpoint on both deterministic and nondeterministic minimal automata} for the language $L_I$. We recover the minimal deterministic automaton by taking the subset of $A(-)$ corresponding to cobordisms (intervals) with outer boundary $-$ rather than their $\Bool$-semilinear combinations. Minimal nondeterministic automata are obtained by lifting the action of $\Sigma^{\ast}$ on $A(-)$ to an action on the minimal free cover $\Bool^J$ (minimal free $\Bool$-semimodule surjecting to $A(-)$). Set $J$ becomes the set of states of that nondeterministic automaton. 
Map $a$ takes $j\in J$ to a subset $a(j)$ of $J$, describing nondeterministic transitions for $a$ on $J$. 
The distinguished ``empty word" element is lifted to an  element of $\Bool^J$ (to a subset of $J$) describing the set of initial states of the automaton.  

\begin{remark} The set $Q_-\subset A(-)$ of the states of the minimal DFA for $L_I$ contains $J$ as a subset. 
\end{remark}

The diagram below summarizes how to recover minimal DFA and minimal NFAs for a language $L$ given the state space $A(-)$ for $L$. Minimal DFA is obtained by restricting to the images of words in $\Sigma^{\ast}$ in $A(-)$ rather than their semilinear combinations. Equivalently, 
the minimal DFA for the language $L_{I}$ is the $\Sigma$-stable subset in $A(-)$ generated by the empty word $\emptyset$. 
        
Minimal NFAs are given by taking the minimal cover of $A(-)$ by a free $\Bool$-semimodule, together with a lifting of the action of $\Sigma^{\ast}$ and $\langle \emptyset |$.  

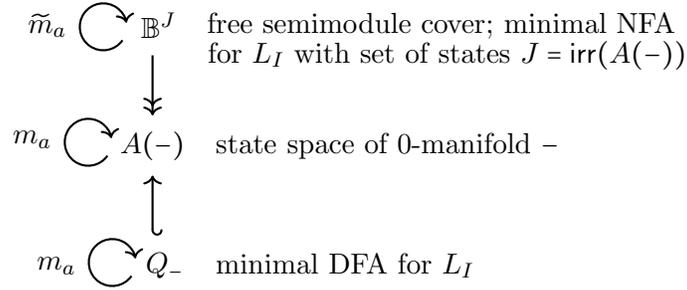
\begin{center}
\begin{tikzpicture}[scale=0.8]

\draw[thick,->] (-0.5,1.75) arc (325:10:3.9mm);
\node at (-1.75,2) {$\widetilde{m}_a$};

\node at (0.1,2) {$\Bool^J$}; 
\node at (4.8,2) {free semimodule cover; minimal NFA };
\node at (4.9,1.5) {for $L_I$ with set of states $J=\irr(A(-))$};

\draw[thick, ->>] (0,1.5) -- (0,0.5);
\node at (0,0) {$A(-)$};

\draw[thick,->] (-0.75,-0.17) arc (325:10:3.9mm);
\node at (-2.0,0.1) {$m_a$};

\node at (3.9,0) {state space of $0$-manifold $-$};
\draw [thick, rotate=270, left hook->] (1.5,0) -- (0.5,0);
\node at (0.2,-2) {$Q_-$};
\draw[thick,->] (-0.35,-2.17) arc (325:10:3.9mm);
\node at (-1.6,-2) {$m_a$};

\node at (3.2,-2) {minimal DFA for $L_I$ };

\end{tikzpicture}
\end{center}

More generally, 
consider any surjective homomorphism $\psi:\Bool^U\lra A(-)$ from a finitely-generated free $\Bool$-semimodule $\Bool^U$ onto $A(-)$, where $U$ is a finite set. 
Choose any lifting $\widetilde{m}_a$ of the map $m_a:A(-)\lra A(-)$ to $\Bool^U$ and a lifting of $\langle \emptyset |\in A(-)$ to 
$\widetilde{\langle \emptyset |}\in \Bool^U$. Map $m_a$ is determined by the action on $a$ on elements of $U$ and each $m_a(j)$ may have more than one lifting to $\Bool^U$, ditto for $\langle \emptyset |$. Trace  on $\Bool^U$ (homomorphism to $\Bool$ which determines the language) is determined via the projection onto $A(-)$ composed with the trace $\alpha_I$ on $A(-)$, see \eqref{eq_trace}. This results in more general nondeterministic automata for the language $L_I$. 

 Non-uniqueness of minimal nondeterministic automata for a language $L_I$ is explained by multiple ways of lifting the action of $\Sigma^{\ast}$ to the minimal free cover of $A(-)$. There exist various refinements of nondeterministic automata which are unique for a given regular language, including \emph{jiromata}~\cite{DLT02}, \emph{\'atomata}~\cite{BT11}, \emph{minimal xor automata}~\cite{VG10}, the universal automaton~\cite{LS08}, and \emph{residual FSA}~\cite{DLT02,DLT04}, also see~\cite{Cla15,AMMU,MU}. It's almost certain that constructions of state spaces $A(-)$ and $A(+)$ are contained, implicitly or explicitly, in some of these refinements. 

\vspace{0.05in} 

The understanding that one can pass from the minimal automaton for a regular language to semilinear combinations of its states and to $A(-)$ and then construct various nondeterministic automata for the language as free covers of $A(-)$ clearly goes back at least several decades to the early work on automata in 1960s and 1970s. (We don't know a definitive early reference, though.) 
Similar ideas feature prominently in the recent work of Myers~\cite{My1,My2} and in~\cite{AMMU,MMU,MU}, also see papers of Clark~\cite{Cla15} on the \emph{syntactic concept lattice} and references there. 


\subsection{State space \texorpdfstring{$A(+-)$}{A(+-)}} \label{subsec-Apm} 

$\quad$

{\it Associative semiring structure.}
Consider the state space $A(+-)=\Hom_{\CCC_{\alpha}}(\emptyset,+-)$ spanned by diagrams with boundary $+-$. It's a $\Bool$-semimodule and is  naturally an associative unital semiring. The multiplication is given by composition with a ``cap" diagram as shown in Figure~\ref{Aplusminus-multn}.  
The unit element $1_{+-}$ for this multiplication is shown in Figure~\ref{Aplusminus} left. 

\input{Aplusminus-multn}

\begin{remark} $A(+-)$ is a unital finite quantale: it is a $\Bool$-semimodule with a compatible associative unital multiplication. Equivalently, $A(+-)$ is a finite semiring with idempotent addition. 
\end{remark} 

$\Bool$-semimodule $A(+-)$ has a spanning set that consists of elements of two types, as shown in Figure~\ref{Aplusminus}. 
Elements of the first type are an arc with a word $\omega$ on it; elements of the second type $\uparrow\downarrow\!(\omega_1, \omega_2)$ are a pair of intervals with words $\omega_1,\omega_2 \in \Sigma^*$ on them. 
Elements of the second type span a two-sided ideal $\I_{\uparrow\downarrow}(+-)$ in the semiring $A(+-)$. 

\input{Aplusminus}


 $A(+-)$ also has a subring $A_{\circleft}(+-)$ spanned by labeled arcs $\{ \circleft(\omega) : \omega\in \Sigma^*\}$ (these are elements of the first type). The ideal and the subring may intersect nontrivially.   

\input{Aplusminus-ex}

{\it Circular-interval deterministic automata.}
$\Bool$-semimodule $A(+-)$ contains a subset $Q^{\cup}_{+-}$ of elements represented by diagrams of the form $\circleft(\omega)$, over all $\omega\in \Sigma^{\ast}$. 
Subset $Q^{\cup}_{+-}$ together with the subsets $Q_+\subset A(+),Q_-\subset A(-)$ can be turned into an automaton that determines both the circular language $L_{\circ}$ and the interval language $L_I$.  

First, part of the structure of $Q^{\cup}_{+-}$ is that of a \emph{deterministic circular automaton} for the circular language $L_{\circ}$, see Section~\ref{subsec_circular}. This circular automaton is usually not minimal for $L_{\circ}$, since construction of $A(+-)$ uses the interval language $L_I$ as well. Furthermore, there are maps 
\begin{equation*}
    \tau_+ \ : \ Q^{\cup}_{+-}\lra  Q_+,  \ \ Q^{\cup}_{+-}\lra Q_-
\end{equation*}
where a $+-$-arc with a word $\omega$ on it is turned into an arc with one endpoint floating, by composing with a corresponding morphism from $+-$ to $+$ or $-$, see Figure~\ref{sect04-001}. These maps are induced by the corresponding $\Bool$-homomorphisms 
$A(+-)\lra A(+), A(+-)\lra A(-)$. Map $\tau_+$ intertwines  the action $\delta_{r}$ of $\Sigma$ on $Q^{\cup}_{+-}$ (see Section~\ref{subsec_circular})  with the action of $\Sigma$ on $Q_+$. Likewise, $\tau_-$ intertwines the action $\delta_{\ell}$ of $\Sigma$ on $Q^{\cup}_{+-}$ with the action of $\Sigma$ on $Q_-$. 

\vspace{0.1in} 

\input{sect04-001}

Let 
\begin{equation} 
Q[+-] := Q^{\cup}_{+-}\sqcup Q_{+} \sqcup Q_{-}
\end{equation} 
be the disjoint union of these three sets. 
Sets $Q_-,Q_+$ define minimal DFA for $L$ and $L^{\op}$ as explained in Section~\ref{subsec_plus_minus}. 
Take 
$Q[+-]$ as the set of states of our \emph{mixed automaton} (\emph{circular-interval automaton}).
Circular automaton structure on $Q^{\cup}_{+-}$ and commuting actions $\delta_{\ell},\delta_r$ of $\Sigma^{\ast}$ on $Q^{\cup}_{+-}$ are compatible, via $\tau_{\pm}$, with the automata structures on $Q_-$ and $Q_+$, giving a deterministic mixed automaton for $L=(L_I,L_{\circ})$. 

\vspace{0.1in} 

 In our definition we used the subset $Q^{\cup}_{+-}\subset A(+-)$. It's possible to instead use a larger subset. Let
\[
Q^{\uparrow\downarrow}_{+-} \ := \ 
\{\uparrow\downarrow\!(\omega_1, \omega_2)\ | \ \omega_1,\omega_2\in\Sigma^{\ast}\subset A(+-)
\]
and 
\[ Q_{+-} \ := \ Q^{\cup}_{+-} \cup Q^{\uparrow\downarrow}_{+-} \ \subset A(+-) 
\] 
(note that these two subsets are not disjoint, in general). Maps $\tau_+,\tau_-$ extend to maps 
\[ \tau_+ \ : \ Q_{+-}\lra Q_+\sqcup \{0\}, \ \ 
\tau_- \ : \ Q_{+-}\lra Q_-\sqcup \{0\}. 
\]
On diagrams of the second type they are also computed by composing with diagrams that terminate one of the endpoints, see Figure~\ref{sect04-001}, so that 
\[ \tau_-(\uparrow\downarrow\!(\omega_1, \omega_2)) = 
    \alpha_{I}(\omega_1) \langle \omega_2 | = 
    \begin{cases}
      \langle \omega_2 | & \mathrm{if} \ \ \omega_1\in L_I, \\
       0   &    \mathrm{otherwise} 
    \end{cases}
\]
and likewise for $\tau_+$. The set of states of this mixed automata (of the second kind) is taken to be 
\begin{equation} \label{eq_qus} Q_{+-} \sqcup Q_+\sqcup Q_- \sqcup \{0\},
\end{equation}
with the caveat that if $0\in A(+)$ is an element of $Q_+$ (likewise for $0\in A(-)$ and $0\in A(+-)$), these zero states should be identified with $0$ in \eqref{eq_qus}. 

Maps $\tau_{\pm}$ above naturallly extends to the transition rules for this \emph{mixed automaton of the second kind}. This deterministic \emph{mixed automaton of the second kind} additionally contains another distinguished element, the state $\uparrow\downarrow\!(\emptyset, \emptyset)\in Q_{+-}$.  
We leave it to an interested reader to define \emph{nondeterministic mixed automata} and refer to a discussion in the first part of Section~\ref{subsec-circular} for related examples of nondeterministic circular automata. Having to lift two commuting actions of $\Sigma^{\ast}$ (given by multiplications at $+$ and $-$ endpoints) to a free cover of $A(+-)$ is an obstruction to an easy construction of nondeterministic circular and mixed automata, see Section~\ref{subsec-circular}.

 \begin{remark} \label{rmk_monoid} Set $Q_{+-}$ is naturally a monoid under the multiplication map in $A(+-)$. There are surjections of monoids
 \[ 
 \xymatrix@-1pc{
 \Sigma^{\ast} \ar@{->>}[rr] & &  Q_{+-} \ar@{->>}[rr] & &  E_{L_I}, 
 }
 \] 
 where $E_{L_I}$ is the syntactic monoid of the interval language $L_I$, see Section~\ref{subset_fsa}.
 \end{remark} 
 
 \begin{remark}
  When the circular language $L_{\circ}$ is empty, it would make sense to compare the semiring  $A_{\circleft}(+-)$ to Pol\'ak's \emph{syntactic semiring} of $L_I$, see~\cite{Pol01,Pol03,Pol-lecture-03}. 
 \end{remark}


\subsection{State 
spaces for general sign sequences} \label{subsec_Aplusn} 
$\quad$

\begin{prop} \label{prop_tensor_product_map} For any regular $\alpha$ and any sign sequences $\varepsilon,\varepsilon'$ there's a natural map of $\Bool$-semimodules 
\begin{equation}
    A(\varepsilon)\otimes A(\varepsilon') \stackrel{t_{\varepsilon,\varepsilon'}}{\lra} A(\varepsilon \varepsilon') .
\end{equation}
The image $\mathrm{Im}(t_{\varepsilon,\varepsilon'})$ surjects onto $A(\varepsilon)\undotimes A(\varepsilon')$.
If $\Bool$-semimodule $A(\varepsilon)$ or $A(\varepsilon')$ is projective, map $t_{\varepsilon,\varepsilon'}$ is an inclusion. 
\end{prop}

\begin{proof}
   Finite $\Bool$-semimodules  $A(\varepsilon),A(\varepsilon^{\ast})$ are dual via the natural pairing 
   $A(\varepsilon)\otimes A(\varepsilon^{\ast})\lra \Bool$. 
   There are injective homomorphisms of these semimodules into a finite free $\Bool$-semimodule $\Bool^S$ and its dual such that the pairing comes from the canonical pairing on $\Bool^S$ and its dual. 
   We can do the same with $A(\varepsilon'),A(\varepsilon^{\ast '})$ via a free $\Bool$-semimodule $\Bool^{S'}$ and its dual. Then $A(\varepsilon)\otimes A(\varepsilon')$ maps to $B^{S\times S'}$. The image of this map is the reduced tensor product  $A(\varepsilon)\undotimes A(\varepsilon')$. 
   
   In the construction of the pairing $A(\varepsilon\varepsilon')\otimes A((\varepsilon\varepsilon')^{\ast})\lra \Bool$ one implicitly uses inclusion of these spaces into free semimodules that may be bigger than $B^{S\times S'}$. That
   still gives a canonical surjection of $\mathrm{Im}(t_{\varepsilon,\varepsilon'})$ onto the reduced tensor product $A(\varepsilon)\undotimes A(\varepsilon')$. 
\end{proof}

\begin{prop}
 For any regular $\alpha$ and $n\ge 0$ the natural $\Bool$-semimodule homomorphisms
 \begin{equation}
     A(+)^{\undotimes n}\lra A(+^n), \hspace{4mm} 
     A(-)^{\undotimes n}\lra A(-^n)
 \end{equation}
 are isomorphisms. 
\end{prop}
If $A(-)$ (equivalently, $A(+)$), is a projective semimodule,  $\undotimes$ can be replaced by $\otimes$ in these isomorphisms. 

\begin{proof}
 These maps are surjective by construction. Indeed, 
 decorated cobordisms that generate the space $A(+^n)$ consist of $n$ arcs, each with a single endpoint at top boundary and with one floating endpoint. Such a cobordism is a product of $n$ cobordisms, one for each endpoint, and the induced map $A(+)^{\otimes n}\lra A(+^n)$ is surjective (likewise for the minus sign spaces). The surjective homomorphism factors through the map from the reduced tensor product, giving the above isomorphism  (injectivity follows from definition of $\undotimes$).
\end{proof}

Beyond these simple observations, we don't peek here into the structure of $A(\varepsilon)$ for general sequences $\varepsilon$ and regular evaluations $\alpha$, with the exception of Theorem~\ref{them_TQFT} in Section~\ref{subsec_decomp} which treats the simplest case from the structural viewpoint, which is when $\alpha$ gives rise to a $\Bool$-valued TQFT.


\subsection{Decomposition of the identity} \label{subsec_decomp} 
To a regular language $L\subset \Sigma^{\ast}$ we can assign state spaces $A_L(+)$ and $A_L(-)$, which are dual $\Bool$-semimodules, equipped with right (respectively left) action of the free monoid $\Sigma^{\ast}$. These spaces are naturally isomorphic to $A(+)$, respectively $A(-)$, for any evaluation function that extends $L$ via a circular language $L_{\circ}$. Independently, they can be defined by starting with free $\Bool$-semimodules 
\[  A'(+) = \Bool[\Sigma^{\ast}], \ \ A'(-) = \Bool[\Sigma^{\ast}], 
\] 
each with a basis of all words in $\Sigma$, and forming the (semi)bilinear form 
\begin{equation}
( \:\:,\:\: )_L \ : \  A'(-)\otimes A'(+)\lra \Bool, \quad  (\omega,\omega')_L \ := \ \alpha_L(\omega\omega'), 
\end{equation} 
based on the evaluation $\alpha_L$ associated with $L$ (so that $\alpha_L(\omega)=1\Leftrightarrow \omega\in L$). Define $A(-)$ and $A(+)$ as the quotient $\Bool$-semimodules of $A'(-)$ and $A'(+)$, respectively, by the kernel of the bilinear form. Thus, $\sum_i v_i=\sum_j w_j$ in $A(-)$ if and only if $\sum_i \alpha_L(v_i u) = \sum_j \alpha_L(w_j u)$ for any $u\in \Sigma^{\ast}$. 

\vspace{0.1in} 

A \emph{decomposition of the identity} for $L$ is a 
finite subset 
\[
S\subset \Sigma^{\ast}\times \Sigma^{\ast}, \hspace{1cm}  S=\{(u_1,v_1),\ldots, (u_m,v_m)\}, \hspace{1cm}  u_i,v_i \in \Sigma^{\ast}
\]
so that for any $x,y\in \Sigma^{\ast}$
\begin{equation} \label{eq_id_decomp} \alpha_L(xy) \ = \ \sum_{i=1}^m \alpha_L(xu_i) \alpha_L(v_iy). 
\end{equation} 
The relation says that, for any words $x,y$, the word $xy$ is in $L$ if and only if for some $i$ both words $xu_i$ and $v_iy$ are in $L$. 
Diagrammatically, we can represent this as a relation in 
Figure~\ref{DEC-ID}. In that figure, we are assuming that top and bottom endpoints are closed up using short arcs (labelled $x$ and $y$ respectively), not using a single arc connecting these two points from the outside. Interpreting this relation inside $\CCC_{\alpha}$ would require picking a circular language, in addition to $L$. Instead, in Figure~\ref{DEC-ID} we can restrict to closures that don't create circles.  

\input{DEC-ID}

A choice of $S$ is not unique. In particular, word $u_i$ can be replaced by any word in its equivalence class in ${}_LE$, ditto for $v_i$ and $E_L$, see Section~\ref{subset_fsa} for notations. 
More generally, we may ask for elements $u_i \in A(+)$, $v_i\in A(-)$ such that relation \eqref{eq_id_decomp} holds for any $x,y\in \Sigma^{\ast}$ and call a \emph{decomposition of the identity} the corresponding subset $S\subset A(+)\times A(-)$. In full generality, a decomposition of the identity is an element $\id_L\in A(+)\otimes A(-)$, $\id_L=\sum_i u_i\otimes v_i$, such that  \eqref{eq_id_decomp} holds for any $x,y\in \Sigma^{\ast}$. 

\vspace{0.05in} 

Note that any language with a decomposition of the identity is necessarily regular. 
Not any regular language has a decomposition of the identity. 

\begin{prop}  
A regular language $L$ has a decomposition of the identity if and only if $A(-)$ is a projective $\Bool$-semimodule.  
\end{prop}

\begin{proof}
 Decomposition of the identity element $\id_L$ has the same properties as the coevaluation map in \eqref{eq_coev_map} and \eqref{coev_open_sets_3}, satisfying the isotopy relations \eqref{cupcap} when combined with an evaluation map.  Proposition~\ref{prop_proj_coev} now implies that if $A(-)$ is a projective semimodule, we can take $\coev_{A(-)}(1)$ as the decomposition of the identity for $L$. Vice versa, having a decomposition of the identity (the coevaluation map satisfying the isotopy relation) implies that semimodule $A(-)$ is projective, see the earlier discussion that follows  Proposition~\ref{prop_proj_coev}. 
 \end{proof}
 
Proposition~\ref{prop_many_char} lists many equivalent characterizations of (finite) projective $\Bool$-semimodules, see also Proposition~\ref{prop_proj_coev}. 

\vspace{0.1in} 

Suppose $L$ has a decomposition of the identity \eqref{eq_id_decomp}. Consider the evaluation $\alpha(L)=(L,L_{\circ})$ with $L$ as the interval language and the circular language $L_{\circ}$ given by 
\begin{equation}
    \omega \in L_{\circ}\ \Leftrightarrow\  \sum_{i=1}^m \alpha_L(\omega u_i)\alpha_L(v_i) = 1.
\end{equation}

\input{figure04-001}

In other words, define $L_{\circ}$ on circular words $\omega\in \Sigma^{\ast}_{\circ}$ by inserting the right-hand side of Figure~\ref{DEC-ID} left relation anywhere on the circle. If we denote the right-hand side by a box labelled $\id$, see Figure~\ref{figure04-001} top, define $L_{\circ}$ via Figure~\ref{figure04-001} bottom. Then relations in Figure~\ref{figure04-002} hold.

\vspace{0.1in} 

\input{figure04-002}

\vspace{0.1in}

Thus, given a regular language $L_I$ such that $A(-)$ is a projective semimodule, there is a canonical circular language $L_{\circ}$ assigned to it as above. The pair $(L_I,L_{\circ})$ gives rise to a regular evaluation $\alpha$. Arc connecting $+$ and $-$ and viewed as element of $A(+-)$ for this evaluation $\alpha$ simplifies into a sum of product terms. 
Consequently, any arc with two outer endpoints in a diagram representing an element of $A(\varepsilon)$ for any sequence $\varepsilon$ can be reduced to a sum of terms which consist of arcs with one out of two endpoints floating (and the other outer). This tells us that the natural ``disjoint union of diagrams'' map $A(\varepsilon)\otimes A(\varepsilon')\lra A(\varepsilon\varepsilon')$ is surjective for any sign sequences $\varepsilon,\varepsilon'$. 

Since $A(-)$ is projective semimodule and the lattice $A(-)^{\wedge}$ is distributive, we can further conclude, see Section~\ref{subsec_semiduality} and Proposition~\ref{prop_inject} there, that there are natural isomorphisms. 
\begin{equation}
    A(\varepsilon)\otimes A(\varepsilon') \ \cong \  A(\varepsilon)\undotimes A(\varepsilon')\ \cong \  A(\varepsilon\varepsilon'), 
\end{equation}
which we can further refine down to the terms for individual signs. For a sequence $\varepsilon=(\varepsilon_1,\ldots, \varepsilon_n)$, $\varepsilon_i\in \{ +,-\}$, there is a decomposition 
\begin{equation}\label{eq_A_factor} 
A(\varepsilon)\ \cong \ A(\varepsilon_1)\otimes \cdots \otimes A(\varepsilon_n). 
\end{equation}  
For this evaluation $\alpha$, the hom spaces in the category $\CCC_{\alpha}$ likewise factor, 
\begin{equation}\label{eq_hom_factor} 
\Hom_{\CCC_{\alpha}}(\varepsilon,\varepsilon')\ \cong \ A(\varepsilon^{\ast}\varepsilon') \ \cong \  A(\varepsilon^{\ast})\otimes A(\varepsilon'), 
\end{equation}
and further factor via \eqref{eq_A_factor}. We can summarize these observations into the following statement. 

\begin{theorem} \label{them_TQFT} 
Suppose that the state space $A(-)$ for a regular language $L_I$ is a projective semimodule (equivalently, lattice $A(-)^{\wedge}$ is distributive). Let $L_{\circ}$ be the associated circular language, via Figure~\ref{figure04-001}, and $\alpha$ the evaluation for $(L_I,L_{\circ})$. Then category $\CCC_{\alpha}$ is a $\Bool$-valued one-dimensional TQFT with defects, with state spaces and hom spaces admitting tensor product decompositions \eqref{eq_A_factor} and \eqref{eq_hom_factor}. 
\end{theorem} 

For such $\alpha$,
we can think of $\CCC_{\alpha}$ and the state spaces $A(\varepsilon)$ as describing a symmetric monoidal functor from the category $\CCC_{\Sigma}$ of $\Sigma$-decorated one-dimensional cobordisms to the category of $\Bool$-semimodules that satisfies the Atiyah's axioms for a TQFT~\cite{A}, with the caveat of using Boolean semiring $\Bool$ in place of a ground field and dropping the involutory axiom. 

\vspace{0.1in} 

We see that the case when $A(-)$ is a projective semimodule ($A(-)^{\wedge}$ a distributive lattice) is important: 
\begin{itemize}
    \item Language $L_I$ has a decomposition of the identity exactly in this case.
    \item There's a unique circular language compatible with $L_I$ in the sense of giving rise to a genuine $\Bool$-valued one-dimensional TQFT with defects. 
    \item Surjective maps $A(-)\otimes N\lra A(-)\, \undotimes \, N$ are isomorphisms for all semimodules $N$ and the state space $A(\varepsilon)$ for the above theory is both the ordinary $\otimes$ and reduced $\undotimes$ tensor products of $A(+)$, $A(-)$, over all signs in the sequence $\varepsilon$.
\end{itemize}

We call languages with a decomposition of the identity \emph{cuttable languages}. A language is cuttable if and only if it is regular and $A(-)$ is a projective $\Bool$-semimodule (equivalently, $A(-)^{\wedge}$ is a distributive lattice). 
It's natural to pose the following question (we don't know if it's been studied in the theory of regular languages). 

\begin{problem}
  Find other characterizations of cuttable languages $L_I$. What is their significance among regular languages? 
\end{problem}

\begin{remark}
The involutory axiom, restricted to cobordisms from the empty manifold to itself, says that $\alpha(\overline{N})=\overline{\alpha(N)}$, where $\overline{N}$ is given by reversing the orientation of $N$. 
To recover the latter condition, one can extend our setup from $\Bool$ to the ground semiring $\Bool[\sigma]/(\sigma^2=1)$, with $\sigma$ playing a role of involution, and consider evaluation $\alpha$ on words and circular words with values in that bigger semiring subject to the involutory condition. Note that reversing an orientation of an interval with word $\omega$ written on it results in the interval with the opposite word $\omega^{\op}$, so one should have, in particular, $\sigma(\alpha(\omega))=\alpha(\omega^{\op})$, for interval words $\omega$. It might be interesting to compare this setup with various doubling contructions in Connes-Consani~\cite{CC}. 
\end{remark} 

\begin{remark} Universal construction for one-dimensional decorated cobordisms with evaluation $\alpha$ taking values in a ground field $\kk$, see~\cite{Kh3}, gives finite-dimensional state spaces when in the pair $\alpha=(\alpha_I,\alpha_{\circ})$ of noncommutative power series both $\alpha_I$ and $\alpha_{\circ}$ are rational. 
As in our Boolean case, the state space $A(-)$ is finite-dimensional if and only if the noncommutative power series $\alpha_I$ is rational (see~\cite{Kh3} for details). Over a field, with $A(-)$ finite-dimensional, decomposition of the identity always exists (vector spaces $A(-)$ and $A(+)$ are dual), unlike the Boolean case, where we need $A(-)$ to be a projective $\Bool$-semimodule. Consequently, there exists a unique circular power series $\alpha_{\circ}$, defined analogously to the formula in Figure~\ref{figure04-001}, such that the pair $\alpha=(\alpha_I,\alpha_{\circ})$ gives rise to a one-dimensional decorated TQFT subject to Atiyah's axioms (again, dropping the involutivity axiom). 
\end{remark} 

Here are some examples of cuttable languages:  
\begin{itemize}
    \item 
Choose a surjection of monoids $f:\Sigma^{\ast}\lra G$, where $G$ is a finite group. Let $L_I=f^{-1}(1)$ and for each $g\in G$ choose $u_g,v_g\in \Sigma^{\ast}$ with $f(u_g)=g, f(v_g)=g^{-1}$. Let  $S=\{(u_g,v_g)|g\in G\}.$ Then $S$ is a decomposition of the identity for $L_I$. The circular language $L_{\circ}=L_I$ in this case. Example~\ref{ex:LI-LO-asquared} in Section~\ref{section:examples} works out in details the case when $G$ is a cyclic group of order two and $\Sigma$ has one letter. 
\item Let $\Sigma=\{a\}$ and  $L_I=a^na^{\ast}=\{a^n,a^{n+1},\ldots\}$ consists of powers of $a$ starting with $a^n$. We can then take  $\sum_{i=0}^n a^i \otimes a^{n-1}$ to be a decomposition of the identity, with the associated circular language $L_{\circ}$ equals $L_I$ in this case as well. 
\end{itemize}
For an example of a cuttable language $L_I$ with the associated circular language different from $L_I$ consider the state space $A(-)\cong \Bool^2$ with basis $\{x,y\}$ and $\Sigma=\{a,b\}$ with the action, initial state and trace given in this basis by (also see Figure~\ref{LI-not-equal-Lo}) 
\begin{equation}
a=  \begin{pmatrix}
0 & 1 \\
1 & 0 
\end{pmatrix}, 
\hspace{4mm}
b=  \begin{pmatrix}
1 & 1 \\
0 & 0 
\end{pmatrix}, 
\hspace{4mm} 
q_{\init} = \begin{pmatrix} 1 \\ 0 \end{pmatrix}, 
\hspace{4mm}
\tr = \begin{pmatrix} 0 & 1 \end{pmatrix}.
\end{equation}
Operator $a$ is a transposition, $b$ is the projection onto $\Bool x$, the initial state $q_{\init}=x$ and the trace map is $0$ on $x$ and $1$ on $y$. It's easy to check that the state space $A(-)$ is indeed $\Bool^2$ and not a quotient of the latter. The language 
$L_I=(a+b)^{\ast} b (a^2)^{\ast}a+(a^2)^{\ast}a.$  The decomposition of the identity is $\id = a\otimes 1 + 1\otimes a$, see Figure~\ref{LI-not-equal-Lo-02}, and the circular language is given by the evaluation 
\[ \alpha_{\circ}(\omega) \ := \ \alpha_I(a\omega)+\alpha_I(\omega a), \ \ \omega\in \Sigma^{\ast}. 
\] 
The empty word $\emptyset$ is in $\alpha_{\circ}$ but not in $\alpha_I$, since $\alpha_I(\emptyset)=0$ and $\alpha_I(a)=1$. The word $a$ is in $\alpha_I$ but not in $\alpha_{\circ}$, since $\alpha_I(a^2)=0$. Thus, in this example, neither of the two languages $L_I,L_{\circ}$ is a subset of the other. 

\input{LI-not-equal-Lo}

\input{LI-not-equal-Lo-02}

\begin{remark} 
We can consider a weaker notion than a decomposition of the identity.  Call a word $\omega\in \Sigma^{\ast}$  \emph{decomposable} if there exists $S$ as above such that for any $x,y\in \Sigma^{\ast}$ 
\[  \alpha_L(x\omega y) \ = \ \sum_{i=1}^m \alpha_L(xu_i) \alpha_L(v_iy). 
\]
$L$-decomposable words constitute a 2-sided ideal in the monoid $\Sigma^{\ast}$ and generate a 2-sided ideal in the semiring $\Bool[\Sigma^{\ast}]$. Under the homomorphism $\gamma:\Bool[\Sigma^{\ast}]\lra A(+-)$ that sends a word $\omega$ to the arc with two two outer endpoints carrying $\omega$, this ideal is the inverse image of $\beta(A(+)\otimes A(-))$ in $A(+-)$: 
\[  \Bool[\Sigma^{\ast}] \stackrel{\gamma}{\lra} A(+-) \stackrel{\beta}{\longleftarrow} A(+)\otimes A(-), 
\] 
where we can identify
\[ A_{\circleft}(+-)= \gamma(\Bool[\Sigma^{\ast}]), \ \  \ \I_{\uparrow\downarrow}(+-) = \beta(A(+)\otimes A(-)),
\] 
see Section~\ref{subsec-Apm} for these notations. 
\end{remark}


\subsection{Circular, interval, and unoriented theories}\label{subsec-circular} 
$\quad$
\vspace{0.05in}

\emph{Circle-only theories and circular automata.}
If the interval language is empty, $L_I=\emptyset$, any interval evaluates to $0$. Hence, any cobordism with a floating endpoint gives the $0$ morphism in $\CCC_{\alpha}$ between its bottom and top boundary. Consequently, in this case from the start we can restrict to the subcategory $\CCC^{\circ}_\Sigma$ of $\Sigma$-decorated oriented cobordisms without floating endpoints. 
This category and the corresponding universal construction (in the linear, not Boolean, case) is considered in~\cite[Section~2]{Kh3}. All endpoints (boundary points) of a cobordism are on its top and bottom boundary (encoded by signed sequences $\varepsilon$ and $\varepsilon'$). 
 
In this category $\Hom(\varepsilon,\varepsilon')\not=0$ if and only if $|\varepsilon|=|\varepsilon'|$, where $|\varepsilon|$ of a sign sequence $\varepsilon$ is the difference of the number of pluses and minuses in it. 

\vspace{0.05in}

Given a circular language $L_{\circ}$, one can do the universal construction with the corresponding evaluation $\alpha_{\circ}$ to get a rigid symmetric monoidal $\Bool$-semilinear category $\CCC_{\alpha_{\circ}}$ and, in particular, define state spaces $A(\varepsilon)$. 
The latter space  is nontrivial only if $|\varepsilon|=0$, and the 
state space for a sequence of $n$ pluses and $n$ minuses  is isomorphic to  $A(+^n-^n)$ via a permutation cobordism. 

State space $A(+-)$ is  a unital associative idempotent semiring spanned over $\Bool$  by arcs with endpoints on $+$ and $-$ and with words $\omega\in \Sigma^{\ast}$. 
Free monoid $\Sigma^{\ast}$ acts on $\Bool$-semimodule $A(+-)$ in two commuting ways, by attaching an interval with $a\in \Sigma$ on it either to the $+$ or to the $-$ endpoint of an arc. 

Take the subset $Q_{+-}\subset A(+-)$ consisting of arcs rather than their $\Bool$-semilinear combinations. This subset is naturally a \emph{deterministic circular automaton} as defined in Section~\ref{subsec_circular} for the language $L_{\circ}$, and it is the minimal such automaton for $L_{\circ}$. 

One can attempt to give examples of \emph{non-deterministic circular automata} for $L_{\circ}$ by analogy with the Section~\ref{subsec_plus_minus} construction of non-deterministic automata via lifting of action of $\Sigma^{\ast}$ on $A(+)$ to an action on its free $\Bool$-semimodule cover. The problem here is that while an action of $\Sigma^{\ast}$ on a $\Bool$-semimodule can be lifted to its free cover, a lifting of two commuting actions might not exist. A similar problem appears when trying to define a non-deterministic ``automaton'' for the entire category $\CCC_{\alpha}$, where $\alpha$ is either a pair of languages $(\alpha_I,\alpha_{\circ})$ or just a circular language. Such an automaton would presumably consist of lifting all hom spaces $\Hom_{\CCC_{\alpha}}(\varepsilon,\varepsilon')$ to free $\Bool$-semimodules in a compatible way preserving the evaluation of closed objects. Equivalently, one would want to have a rigid symmetric monoidal $\Bool$-semilinear category 
$\widetilde{\CCC}_{\alpha}$ with the same objects as $\CCC_{\alpha}$ (signed sequences), homs being \emph{free} $\Bool$-semimodules, and a rigid monoidal $\Bool$-semilinear functor $\widetilde{F}: \widetilde{\CCC}_{\alpha}\lra \CCC_{\alpha}$ which is the identity on objects and surjective on morphisms. 
We do not attempt to study such liftings in the present paper. Non-deterministic circular automata would be related to such liftings of the category $\CCC_{\alpha_{\circ}}$. 

\vspace{0.05in} 

State space $A(+-)$ is naturally a unital associative semialgebra over $\Bool$ under the concatenation multiplication described in Section~\ref{subsec-Apm}. Suppose given a surjective homomorphism $\widetilde{A}\lra A(+-)$ of such semialgebras such that $\widetilde{A}$ is, in addition, a free $\Bool$-semimodule (of finite rank). Semialgebra $\widetilde{A}$ comes with a natural trace, via the composition of the surjection onto $A(+-)$ and the trace on $A(+-)$. The set $\irr(\widetilde{A})$ of irreducible elements (generators or basis elements) of this free $\Bool$-semimodule can then be taken as the set of states of a nondeterministic circular automaton for the 
circular language $L_{\circ}$. Such semialgebras and homomorphisms always exist. For example, one can take for the set of basis elements of  $\widetilde{A}$ the set $A(+-)\setminus \{0\}$ of nonzero elements of $A(+-)$. The multiplication in this basis of  $\widetilde{A}$ is induced by the multiplication in $A(+-)$ and gives a basis element of $\widetilde{A}$ or $0$. This is not an efficient lifting, due to large size of $\widetilde{A}$ compared to that of $A(+-)$.

We leave it to the reader to give a general definition of a nondeterministic circular automaton.

\vspace{0.1in} 

\emph{Interval-only theories.}
Assume that $L_{\circ}=\emptyset$ is the empty language. By the empty language $\emptyset$ we mean the language that contains no words. It's different from the language $\{\emptyset\}$ that contains the empty word and nothing else. In this case we can identity the syntactic monoid $E_{L_I}$ of $L_I$ with the image of the free monoid $\Sigma^{\ast}$ in $A(+-)$ under the map sending $\omega\in \Sigma^{\ast}$ to the arc with endpoints $+,-$ and carrying $\omega$. For a general language $L_I$, the image of $\Sigma^{\ast}$ in $A(+-)$ only surjects onto the syntactic monoid, see Remark~\ref{rmk_monoid}. It's a  natural question to characterize circular languages, given $L_I$, such that the image of $\Sigma^{\ast}$ in $A(+-)$ is $E_{L_I}$. Such circular languages are, in a sense, the most compatible with the interval language $L_I$. 

\vspace{0.1in} 

{\it Unoriented cobordisms.} 
There is an unoriented  version of the category  $\CCC_{\Sigma}$: one-dimensional cobordisms are now unoriented and the objects are numbers  $n\in \Z_+$, counting  the number of  top and  bottom endpoints of the  cobordism. Evaluation  $\alpha$ must  be reflection-invariant: $\alpha_I(\omega^{\op})=\alpha_I(\omega)$, for $\omega\in\Sigma^{\ast}$ and likewise for the circular evaluation: $\alpha_{\circ}(\omega^{\op})=\alpha_{\circ}(\omega)$, for $\omega\in\Sigma^{\ast}_{\circ}.$

The state space $A(1)$ of a single point in this theory comes with a symmetric pairing $(\:\:,\:\:): A(1)\times A(1)\lra \Bool$, nondegenerate in the sense that different elements of $A(1)$ define different functionals on $A(1)$. We can pick the subset $Q_1\subset A(1)$ which consists of images $\langle \omega |$ of all words $\omega\in \Sigma^{\ast}$ rather than their semilinear combinations. There is an action of $\Sigma^{\ast}$ on $Q_1$ via the map 
$\delta: \Sigma\times Q_1\lra Q_1$. This action is self-adjoint with respect to the above pairing $(\:\:,\:\:)$. The pairing restricts to the map, also denoted $(\:\:,\:\:):Q_1\times Q_1\lra \Bool$, with the self-adjointness property, $(\delta_a(q_1),q_2)=(q_1,\delta_a(q_2))$, for states $q_1,q_2\in Q_1$. There is a distinguished initial state $q_{\init}=\langle \emptyset | \in Q_1$. Pairing with $q_{\init}$ determines whether a state $q\in Q_1$ is acceptable.

$Q_1$ with this additional structure gives the minimal \emph{symmetric} or \emph{unoriented} deterministic automaton for the symmetric language $L=L^{\op}$. 
It is straightforward to turn the above structure on $Q_1$ into a proper definition of a deterministic symmetric automaton for a symmetric regular language $L$. One can further study the case when $A(1)$ is a projective semimodule and there's a decomposition of the identity for a symmetric language $L$. 

%
%

\section{Topological theories for semimodules and semimodule automata}\label{sec_semimod_automata}


\subsection{Topological theory from a semimodule pairing}
\label{subsec_pairing} 

The category $\Bfmod$ has a natural tensor product, and each object $M$ has a dual object $M^{\ast}$ with a nondegenerate pairing $\coev_M: M^{\ast}\otimes M\lra \Bool$. Despite of that, $\Bfmod$ is not naturally a rigid monoidal category, due to the absence of coevaluation maps $\Bool \lra M \otimes M^{\ast}$ that satisfy suitable isotopy relations, see discussion in Section~\ref{subsec_semiduality}.

Figure~\ref{failure-duality-01} shows four examples of pairings between $M$ and $M^{\ast}$, for various semimodules $M$. In examples (a)-(c) coevaluation map exists, but not in (d). Spanning set for $M$ is $x,y,z$, and $x',y',z'$ for $M^{\ast}$ (except in (a), with just $x,y$ and $x',y'$).

\input{failure-duality-01}

A possible solution is to fix $M$ and enhance the tensor product by artificially adding a ``cup",   
\begin{equation}\label{eq_add_cup}
    M\widetilde{\otimes} M^{\ast} \ :=  \ (M\otimes M^{\ast})\oplus \Bool  \circleft  / \sim .
\end{equation}
That is, we take the direct sum of the tensor product $M\otimes M^{\ast}$ with a free rank one $\Bool$-semimodule generated by a cup diagram and then mod out by the equivalence relation coming from the (semi)bilinear form on this direct sum. The bilinear form is determined by the pairing between $M$ and $M^{\ast}$ plus a choice of evaluation for a circle, equal to the pairing 
\begin{equation} 
\alpha(\bigcirc)=(\circleft,\circleft)\in \Bool
\end{equation} 
of the cup with itself. Given $\alpha(\bigcirc)$, we can then build a rigid category, as follows. 

\vspace{0.05in} 

The pairing between $M$ and $M^{\ast}$ is given by a Boolean matrix once generators of these two semimodules are chosen. Let's start with a Boolean matrix $\mathbf{M}$, possibly infinite. Denote by $S_-$ and $S_+$ the sets of rows and columns, correspondingly, and by $\mathrm{M}(a,b)$ the entry at the row $a\in S_-$ and column $b\in S_+$. To $\mathbf{M}$ assign $\Bool$-semimodules $M_-$ and $M_+$. Semimodule $M_+$ is the subsemimodule of $\Bool S_+$ generated by the columns of the matrix, while $M_-$ is the subsemimodule of $\Bool S_-$ generated by the rows. Pairing $M_+\otimes M_- \lra \Bool$ given by the matrix is nondegenerate. $M_-$ is a subsemimodule of $M_+^{\ast}$, and, if the matrix is infinite, a proper subsemimodule most of the time (likewise for $M_-$). For finite matrices, we can identify $M_-\cong (M_+)^{\ast}$, $M_+\cong (M_-)^{\ast}$. 

\vspace{0.05in} 

Define the category $\CCC_S$  to be the category of oriented one-dimensional cobordisms where inner (floating) endpoints of cobordisms are labelled by elements of $S_+$ and $S_-$. Objects of $\CCC_S$ are finite sequences of signs. A morphism from $\varepsilon$ to $\varepsilon'$ is an oriented one-dimensional cobordism between these zero-manifolds which may have inner endpoints. Inner ``in" endpoints (orientation directed into the manifold) are labelled by elements of $S_+$. Inner ``out" endpoints (orientation directed out of the manifold) are labelled by elements of $S_-$. Figure~\ref{figure05-002} depicts an example.

\input{figure05-002}

Composition of morphisms is given by concatenation. An endomorphism of the empty sequence $\emptyset$ is a disjoint union of oriented intervals $I(a,b)$ with endpoints labelled by elements $a\in S_-$ and $b\in S_+$ and circles. $\CCC_S$ is a rigid symmetric monoidal category.  We call morphisms in this category \emph{$S$-cobordisms}. 

Let $\alpha$ be the evaluation function on $\End_{\CCC_S}(\emptyset)$ that evaluates connected floating components by  
\begin{equation}
    \alpha(I(a,b))=\mathM(a,b), \ \ \alpha(\bigcirc) = \lambda \in \Bool
\end{equation}
and is multiplicative on disjoint unions. 

\vspace{0.05in} 

Define $\CCC_{\alpha}'$ to be the $\Bool$-semilinear category with the same objects as $\CCC_{S}$. Morphisms are $\Bool$-semilinear combinations of morphisms in $\CCC_{S}$, with floating components evaluated using $\alpha$. When composing morphisms, the diagram is simplified if it has any floating components, see examples in Figure~\ref{figure05-004} and \ref{figure05-003}. 

\input{figure05-004}

\input{figure05-003}

Figure~\ref{failure-duality-03} shows an example of converting from a pairing matrix $\mathbf{M}$ between $M_+$ and $M_-$ to the evaluation rules for boundary-labelled intervals. 
 
\input{failure-duality-03} 

\vspace{0.05in} 

Category $\CCC_{\alpha}'$ is an intermediate category introduced for convenience. Define $\CCC_{\alpha}$ as a quotient of $\CCC_{\alpha}$ via the bilinear pairing. For two sequences $\varepsilon$, $\varepsilon'$ and two finite sets $U,V$ of $S$-cobordisms (each with boundary $(\varepsilon',\varepsilon)$) the semilinear combinations are equal in $\CCC_{\alpha}$: 
\begin{equation}
     \sum_{u\in U} u = \sum_{v\in V} v
\end{equation}
if for any way to close $u$ and $v$ using a morphism $\omega: \varepsilon^{\ast}\lra \varepsilon'^{\ast}$ the evaluations are equal: 
\begin{equation}
     \sum_{u\in U} \alpha(\ev \circ (u\otimes \omega) \circ \coev ) = \sum_{v\in V} \alpha(\ev \circ (v\otimes \omega)\circ \coev ), 
\end{equation}
where the closure is depicted in Figure~\ref{figure05-001}.   

\input{figure05-001}

Categories $\CCC_{\alpha}'$ and $\CCC_{\alpha}$ are $\Bool$-semilinear rigid symmetric monoidal categories.
Denote by $A(\varepsilon)= \Hom_{\CCC_{\alpha}}(\emptyset,\varepsilon)$ the $\Bool$-semimodule of homs from the empty  0-manifold 
$\emptyset$ (empty sequence) to $\varepsilon$. 
There are natural identifications of $\Bool$-semimodules
\begin{equation}
    A(+)\ \cong \ M_+, \quad\quad A(-)\ \cong \  M_- ,
\end{equation}
and 
\begin{equation}
    A(+-) \ \cong \ (M_+\otimes M_-)\oplus \Bool \circleft /\sim ,  
\end{equation}
as discussed earlier, see \eqref{eq_add_cup},
where the equivalence relation comes from the bilinear form on $(M_+\otimes M_-)\oplus \Bool \circleft$ defined via the pairing between $M_+$ and $M_-$ and circle evaluation. 

\vspace{0.1in}

Category $\CCC_{\alpha}$ converts nondegenerate pairing between $M_+$ and $M_-$ into a rigid symmetric category. Cup morphism adds coevaluation  between $M_+$ and $M_-$ that's missing from the category $\Bfmod$ unless $M_+, M_-$ are (finite) projective semimodules.  Note that $\CCC_{\alpha}$ does not contain all endomorphisms of $M_+$, only the subsemiring of endomorphisms generated by those that factor through $\Bool$ and by the identity endomorphism. For a similar completion of the entire semiring $\End_{\Bool}(M_+)$ one needs a choice of a trace map from this semiring to $\Bool$. A rigid monoidal completion of the category of all homs in $\Bfmod$ between tensor products of $M_+$ and $M_-$ further requires a choice of a compatible collection of trace maps on endomorphism semirings of all such products. 

\input{failure-duality-02}

\vspace{0.1in} 

A example of extended pairing is shown in Figure~\ref{failure-duality-02}. Here one starts with the pairing matrix $\mathbf{M}$ given in Figure~\ref{failure-duality-03} on the left. This extends to a $4\times 4$ pairing matrix on the direct product $S_+\times S_-$, which is the top $4\times 4$ matrix in each of the four matrices in Figure~\ref{failure-duality-02}.
In $M_+$, respectively $M_-$, there is a relation $x+y=y$, respectively $x'+y'=y'$. 

We add the \emph{cup} morphism $\circleft$ to the four product morphisms in $\Hom_{\CCC_{\alpha}'}(\emptyset,(+-))$ given by elements of $S_+\times S_-$. The pairing between $\circleft$ and elements of $S_+\times S_-$, written in the rightmost column  and row of these matrices comes from $\mathbf{M}$ as well. The bottom right entry (entry $(5,5)$) is the evaluation of the circle. Thus, $\mathbf{M}$ determines all but that entry of the matrix. 

In this example $M_+,M_-$ are projective semimodules and the coevaluation element exists in $M_+\otimes M_-$. It is given by 
\begin{equation}\label{eq_circle_equal}
\circleft = x\otimes y'+x'\otimes y.
\end{equation} 
Circle evaluates to $\lambda=y'(x)+x'(y)=1+1=1.$ This corresponds to the matrix (c) in the figure. Note that column $5$ is the sum of columns 2 and 3, which is just the formula \eqref{eq_circle_equal}. 

In case (b) we're adding a cup element but with a different evaluation $0$ from the one that exists in $M_+\otimes M_-$. Passing then to the state space $A(+-)$, cup $\circleft$ is not an element of $M_+\otimes M_-$ but satisfies nontrivial relations with the elements of the latter: 
\begin{equation}
    \circleft +y\otimes y' = y \otimes y', \ \ \circleft +x\otimes x'= \circleft, 
    \ \ \circleft + x\otimes y'= \circleft + y\otimes x'. 
\end{equation}
We don't know whether, together with the relations in $M_+$ and $M_-$, written above, and the standard isotopy relations, this gives a set of defining relations in the monoidal category $\CCC_{\alpha}$ for this matrix  $\mathbf{M}$ and $\lambda=0$. 
In a similar Example 4 in Section~\ref{subsec_no_labels}, corresponding to a free rank one semimodule $M_+$ and circle evaluation $\lambda=0$, new defining relations appear beyond those extracted from the state space $A(+-)$, see relations (b), (c) in  Figure~\ref{figure05-009}.  

\vspace{0.05in} 

{\it Linear examples of circular regularization.}
A similar ``completion'' to a rigid symmetric monoidal category can be constructed in many other contexts. For a toy example, consider an infinite set $J$ and a pair of vector spaces 
\begin{equation}
V_+=\oplusop{j\in J} \kk v_j, \ \ V_- = \oplusop{j\in J} \kk v^{\ast}_j,
\end{equation} 
where we think of $V_-$ as a restricted dual of $V_+$, via the pairing $v^{\ast}_k(v_j)=\delta_{j,k}$. This pairing is an ``evaluation'' map $\ev : V_+\otimes V_-\lra \kk$. Coevaluation map should have the form 
\begin{equation}
\coev: \     \kk \lra V_+\otimes V_-, \ \ \ 1\longmapsto \sum_{j\in J} v_j \otimes v_j^{\ast},
\end{equation}
which does not make sense, due to  the absence of infinite sums and an attempted composition with the evaluation $\ev$ giving an infinite sum $\sum_{j\in J}1$. 

Instead, we pretend that $\ev\circ \coev(1)=\lambda\in \kk$  and build a rigid symmetric monoidal category $\CCC_{J,\lambda}$ similarly to our setup with semimodules. The category is $\kk$-linear this time, with elements of $V_+$ and $V_-$ written at the \emph{in} and \emph{out} floating endpoints of intervals, evaluation of floating intervals given by the pairing on these two vectors spaces and circles evaluating to $\lambda$, see Figure~\ref{figure05-005}. This can be thought of as a very simple example of regularization of setting the ``dimension'' of an infinite-dimensional space to $\lambda$. 

\vspace{0.05in} 

A similar example can be done with a finite-dimensional $\kk$-vector space $V$ with a basis $\{v_i\}_{i=1}^n$. The dual space $V^{\ast}$ has a basis $\{v_i^{\ast}\}$ and the pairing $V^{\ast}\otimes V\lra \kk$ is $v_j^{\ast}(v_i)=\delta_{i,j}$. We consider a category of oriented one-dimensional cobordisms with floating endpoints labelled $i\in\{1,2,\ldots,n\}$ and evaluate an oriented $(i,j)$-segment to $\delta_{i,j}$, see Figure~\ref{figure05-005} as well.

\input{figure05-005}

Circle, if understood as trace of the identity operator on $V$, should evaluate to $\dim(V)=n\in \kk$. Instead, we evaluate a circle to some $\lambda\in \kk$. 
We can then pass to the corresponding $\kk$-linear quotient category to get a rigid symmetric monoidal category $\CCC_{n,\lambda}$. If $\lambda=n$, it's just the category whose objects are tensor products of $V$ and $V^{\ast}$ and morphisms are all $\kk$-linear maps between these products. When $\lambda\not=n$, the category is different. It's similar to the gligible quotient of the oriented rook Brauer category (see \cite{HdelM} for the rook Brauer algebra), but with endpoints carrying labels.

 
\subsection{Four categories for dotless cobordisms}\label{subsec_no_labels} 

Consider a special case when there are no labels neither inside nor on floating boundaries of cobordisms. Then there are only two homeomorphism classes of connected floating cobordisms -- an interval and a circle, and four ways to evaluate the pair to elements of $\Bool$ (we denote oriented interval by $\leftarrow$ and oriented circle by $\bigcirc$ below): 
\begin{enumerate}
    \item\label{sec5-case01} $\alpha(\leftarrow )=0$, $\alpha(\bigcirc)=0$, 
    \item\label{sec5-case02} $\alpha(\leftarrow)=0$, $\alpha(\bigcirc)=1$,
    \item\label{sec5-case03} $\alpha(\leftarrow)=1$, $\alpha(\bigcirc)=1$,
    \item\label{sec5-case04} $\alpha(\leftarrow)=1$, $\alpha(\bigcirc)=0$.
\end{enumerate}
This results in four rigid monoidal categories $\CCC_{\alpha}$, which we consider separately. Categories in (1)-(3) have very simple structure; in case (4) we don't know an explicit identification of hom spaces in the category nor a complete set of defining relations. 

\vspace{0.05in} 
\eqref{sec5-case01}
$\alpha(\leftarrow)=\alpha(\bigcirc)=0$. Every closed cobordism evaluates to $0$ and the bilinear pairing on one-manifolds with outer boundary a given non-empty sign sequence $\varepsilon$ is trivial. Consequently, each state space $A(\varepsilon)=0$
except that for the empty sequence $A(\emptyset_0)=\Bool[\emptyset_1]$, so that the corresponding state space is generated by the empty one-dimensional cobordism $\emptyset_1$ (which evaluates to $1$). 
The hom spaces are given by  
\[\Hom(\varepsilon,\varepsilon')\cong
\begin{cases}
 \Bool  & \varepsilon=\varepsilon'=\emptyset_0, \\
 0 & \mathrm{otherwise}. 
\end{cases} 
\]
The identity object $\one = \emptyset_0$ has $\Bool$ as the endomorphism semiring and all other objects are trivial. 

\vspace{0.05in} 

\eqref{sec5-case02}
$\alpha(\leftarrow)=0, \alpha(\bigcirc)=1$. Any cobordism with non-empty floating boundary evaluates to $0$ upon closure, due to the presence of an interval in the resulting floating cobordism. Consequently, there are relations in Figure~\ref{figure05-007} left and nontrivial diagrams in hom spaces $\Hom(\varepsilon,\varepsilon')$ are those that have no floating endpoints and consist of a pairing between the outer endpoints via connecting oriented intervals. A pair of any two such diagrams evaluates to $1$ upon closure. Consequently, any two diagrams without floating endpoints are equal in $\CCC_{\alpha}$, and the hom space $\Hom(\varepsilon,\varepsilon')$ is one-dimensional when an orientation-matching pairing of endpoints exists and trivial otherwise: 
\[\Hom(\varepsilon,\varepsilon')=
\begin{cases}
 \Bool [\emptyset_1] & \mathrm{if} \ h(\varepsilon)=h(\varepsilon'), \\
 0 & \mathrm{otherwise}. 
\end{cases} 
\]
Here $h(\varepsilon)$ is the number of $+$ signs in $\varepsilon$ minus the number of $-$ signs. These additional relations are shown in Figure~\ref{figure05-007} on the right.  
 
\vspace{0.1in} 

\input{figure05-007}

\eqref{sec5-case03}  
$\alpha(\leftarrow)=1, \alpha(\bigcirc)=1$. In this case any floating 1-manifold evaluates to $1$ and any two diagrams with the same signed endpoints are equal. Consequently, $A(\varepsilon)=\Bool$, $\Hom(\varepsilon,\varepsilon')=\Bool$ for any signed sequences. Objects $\varepsilon$ are isomorphic, over all sequences $\varepsilon$, and isomorphic to the identity object $\one=\emptyset_0$. A set of defining relations for this monoidal category is shown in Figure~\ref{figure05-008}. 

\vspace{0.1in} 

\input{figure05-008}

\eqref{sec5-case04}  
$\alpha(\leftarrow)=1, \alpha(\bigcirc)=0$. Computing the bilinear form on $A(+-)$, on $A(++--)$, and the pairing $A(++-)\otimes A(+--)\lra \Bool$ gives us relations shown in Figure~\ref{figure05-009}. We don't know a full set of defining relations for this category.

\input{figure05-009}

 
\subsection{Measuring language complexity and  compatibility of languages}\label{subsec_measuring} 

For a regular language $L$ we can measure its complexity as 
\[
c(L) \ := \ \log_2|A_L(-)|.
\]
Recall that $A_L(-)$ is the state space of $-$ for the language $L$. The equality $|A_L(-)|=|A_L(+)|$ shows that  $c(L^{\op})=c(L)$. 

The state space $A_L(-)$ is a $\Bool$-semimodule, and it may be hard to compute its cardinality given the language. In the special case when $A_L(-)$ is a free semimodule on $n$ generators, the complexity $c(L)=n$. An upper bound on $c(L)$ is the minimal number of generators of $A_L(-)$ (the cardinality of $\irr(A_L(-))$).  

\vspace{0.1in} 

Consider two regular languages $L_1,L_2\subset \Sigma^{\ast}$. To introduce a measure of similarity between them, consider a modification of the category $\CCC$ similar to the one in the previous section. Namely, allow inner (floating) endpoints with ``in'' orientation of two types, corresponding to our two languages, while the inner endpoints with the ``out'' orientation have only one type, see 
Figure~\ref{figure05-006} for pictures of endpoints and evaluation of floating intervals in that category 
and 
Figure~\ref{generalized-ex-02} for an example of a more complicated morphism. 

\input{figure05-006}

\input{generalized-ex-02} 

More precisely, we first consider a ``free'' category of decorated cobordisms with those generators and subject to isotopy relations. 
Then define an  evaluation $\alpha$ using $L_1, L_2$ and an arbitrary circular language $L_{\circ}$, see Figure~\ref{figure05-006} (Language $L_{\circ}$ is used to evaluated circle diagrams). $\Bool$-semilinear combinations of cobordisms are allowed. Evaluation $\alpha$ allows to remove floating components from a diagram. After that, pass to the quotient category, denoted $\CCC_{\alpha}$, via the universal construction. The state space $A_{\alpha}(+)$ in this category is spanned by two types of diagrams: words $\omega$ on an interval with floating endpoint labelled $1$ or $2$. State space $A_{\alpha}(-)$ is spanned by only one type of diagrams. Again, $A_{\alpha}(-)\cong A_{\alpha}(+)^{\ast}$ and  $|A_{\alpha}(-)|=|A_{\alpha}(+)|$. 
These state spaces do not depend on the choice of $L_{\circ}$, but state spaces for sequences that contain both $+$ and $-$ depend on $L_{\circ}$.   

\begin{remark} One way to define a deterministic automaton for a pair of languages $(L_1,L_2)$ is as an automaton with two subsets of accepting states $Q^1_{\t}$ and $Q^2_{\t}$, one for each language. Word $\omega\in L_i$ if and only if, after reading it, the automaton is in a state in $Q^i_{\t},$ $i=1,2$. Forgetting one of these subsets produces an automaton for the other language. Similar definitions work for non-deterministic automata, etc. Variations on this definition include 
\begin{itemize}
    \item Having two initial states (or sets of initial states) for the two languages but the same set $Q_{\t}$ of accepting states. 
    \item Allowing different initial states $q_{\init}^1,q_{\init}^2$ and different subsets of accepting states $Q_{\t}^1,Q_{\t}^2$ for the two languages. 
\end{itemize}
\end{remark} 

{\it Joint complexity.}
Define the joint complexity $c(L_1,L_2)$ of languages $L_1,L_2$ as the logarithm of the size of state space $A_{1,2}(-)$ in that category.
\begin{equation} 
c(L_1,L_2) \ := \  \log_2 |A_{1,2}(-)|.
\end{equation} 
Define the relative complexity of $L_2$ given $L_1$ as \begin{equation}
    c(L_2|L_1) \ := \ c(L_1,L_2) - c(L_1)\ = \ \log_2\biggl(\frac{|A_{1,2}(-)|}{|A_1(-)|}\biggr).
\end{equation} 
Clearly, $c(L_2|L_1)\ge 0$ for any regular $L_1,L_2$. It's easy to produce languages $L_2\not= L_1$ with $c(L_2|L_1)=0$, for instance by taking any $L_1$ and the empty language $L_2$. With any language $L$ we can associate a set of languages $\{L'\}$ such that $c(L'|L)=0$. 

\vspace{0.05in} 

In our definition of the category, the labels are different at the ``out'' floating endpoints. With our conventions for the automata (which is build from the space $A(-)$ with the action of $\Sigma^{\ast}$), this means that the automaton first reads the words fully and then finds out whether the language is $L_1$ or $L_2$ at the end of the word (at the ``in'' endpoint). Alternatively, we can redefine the category by decorating only ``out'' floating endpoints with $1$ or $2$. Then whether the word being tested is in language $L_1$ or $L_2$ is known at the start. One can flip between these two types of categories and corresponding automata by reversing orientations of cobordisms and passing to opposite languages $L_1^{\op},L_2^{\op}$. A third possibility is to have labels $1$ and $2$ at both types of endpoints, but then one needs an additional evaluation function for words on floating intervals that carry opposite labels $1,2$ at their two endpoints (one possibility is to evaluate all such cobordisms to $0$). 

\vspace{0.05in} 
    
Going back to the case of a single pair $L=(L_I,L_{\circ})$ of a regular language $L_I$ and a regular circular language $L_{\circ}$, we can introduce a measure of complexity that $L_{\circ}$ adds to $L_I$. For that, consider the state space $A(+-)$ and the $\Bool$-subsemimodule in it spanned by product diagrams.  The natural map $A(+)\otimes A(-)\lra A(+-)$ has the property that its image (the semimodule above) surjects onto the reduced tensor product $A(+)\undotimes A(-)$. One can define the \emph{relative complexity} of a circular language $L_{\circ}$ given a language $L_I$ by 
\begin{equation}
    c(L_{\circ}|L_I)\ := \ \log_2|A(+-)| -  \log_2 |A(+)\undotimes A(-)| \ = \ \log_2\biggl(\frac{|A(+-)|}{|A(+)\undotimes A(-)|}\biggr).
\end{equation}

\vspace{0.05in} 

{\it Labelling both endpoints and dots.}
So far we have given several examples of closing a regular language or a $\Bool$-semimodule data into a rigid symmetric semilinear category: 
\begin{itemize}
    \item Category $\CCC_L$ assigned to a pair $L=(L_I,L_{\circ})$ of a regular language and a regular circular language. Oriented 1-cobordisms are decorated by dots inside the cobordism carrying labels from the set $\Sigma$ of letters. 
    \item Category $\CCC_{\alpha}$, where $\alpha=(\mathbf{M},\lambda)$, associated to a duality pairing between semimodules and a circle evaluation $\lambda$. Floating ends oriented 1-cobordisms are decorated by elements of sets $S_+$, $S_-$ of generators of dual semimodules $M_+,M_-$ and 1-manifolds carry no dots. 
    \item Categories for a pair of regular languages (and a circular language). One-manifolds are decorated by $\Sigma$-labelled dots and either ``in'' or ``out'' floating endpoints are labelled by $1,2$. 
\end{itemize}

It is straightforward to unify these examples. One would consider oriented one-manifold cobordisms with dots labelled by elements of $\Sigma$ and floating endpoints (depending on their orientation) labelled by elements of finite sets $S_+,S_-$. Evaluation $\alpha$ is a map from diffeomorphism classes of possible floating connected one-manifolds with these decorations to $\Bool.$ When a component is a circle, its evaluation depends only on the circular word written on it. When a component is an interval, its evaluation depends on the word $\omega$ on it and labels of the two endpoints (elements of $S_+$ and $S_-$, correspondingly). 

\input{sect05-001}

An evaluation of this form is determined by a matrix of languages and a circular language, see Figure~\ref{sect05-001}. Evaluation $\alpha$ on $(i,j)$-intervals are given by a language $L_{ij}$, and on circles by the circular language $L_{\circ}$. 

Such a data is called \emph{regular} when state spaces $A(\varepsilon)$ in the corresponding theory are finite $\Bool$-semimodules, for all sign sequences $\varepsilon$.  Regularity is equivalent to finiteness of the state space $A(+-)$, and to all languages $L_{ij},L_{\circ}$ being regular. One can ask when does this data correspond to a $\Bool$-valued TQFT so that $A(\varepsilon)^{\wedge}$ are finite distributive lattices and there are isomorphisms $A(\varepsilon)\otimes A(\varepsilon')\cong A(\varepsilon\varepsilon')$. It's straightforward to extend the notion of a decomposition of the identity to this setup.

 %
 %
 
\section{Examples} 
 \label{section:examples}
 
This section provides several examples for various notions and constructions introduced introduced in this paper.

\begin{example}
\label{ex:LI-LO-asquared}

Let $\Sigma=\{ a\}$ and $L=(L_I,L_{\circ})$ with $L_I = L_{\circ}=(a^2)^*$. The languages are invariant under adding or removing $a^2$ from a word in any position, so there's a skein relation shown in Figure~\ref{LILo-asquared-03} left, where a dot with label $n$ denotes $a^n$. 

\input{LILo-asquared-03}

Modulo this relation, given a sequence $\varepsilon$, for a set of spanning vectors in $A(\varepsilon)$ can take all possible one-manifolds with outer boundary $\varepsilon$, without floating components, and with at most one dot on each component (strand). 
In particular, 
the spanning vectors for $A(+)$, $A(-)$ (two for each space), and for $A(+-)$ (six vectors) are given in Figure~\ref{LILo-asquared-01}. Pairing between $A(+)$ and $A(-)$ is given by the identity matrix, see Figure~\ref{circ-automata-03}, 
so that $A(+),A(-)$ are free $\Bool$-semimodules of rank two, with the bases shown in that figure. 

\input{LILo-asquared-01}

The matrix of the pairing between $A(+)$ and $A(-)$ is given in Figure~\ref{circ-automata-03}, 
showing that $A(+)$ and $A(-)$ are each a free $\Bool$-semimodule on those two generators. 

\input{circ-automata-03}

Pairing between $A(+^n)$ and $B(+^n)$ in the product sets of generators will be a $2^n\times 2^n$ matrix with a single $1$ in each row and column and $0$ everywhere else. 
Case $n=2$ (pairing $A(++)\times A(--)\lra \Bool$) is shown in Figures~\ref{circ-automata-11} and~\ref{circ-automata-12}, with the pairing given by the permutation matrix. 
Consequently, $A(+^n)\cong A(+)^{\otimes n}$ is a free $\Bool$-semimodule on $2^n$ product generators, and likewise for $A(-^n)\cong A(-)^{\otimes n}$. 

\input{circ-automata-11}

\input{circ-automata-12}

\vspace{0.05in} 

Matrix of the pairing on $A(+-)$ is given in Figure~\ref{LILo-asquared-02}, on the six diagrams that span the space. We see that column 5 is the sum of columns 1 and 4, shown as the middle relation in Figure~\ref{LILo-asquared-03}. 
Since the alphabet has a single letter, we don't need to label the dots.

Adding a dot at one of the endpoints of that relation gives us the relation that column 6 is the sum of columns 2 and 3. Columns 1 through 4 generate $A(+-)$ and have no relations on them. Consequently, $A(+-)$ is a free rank four $\Bool$-semimodule with the basis given by the first four columns.

\input{LILo-asquared-02}

The center relation in Figure~\ref{LILo-asquared-03} allows to ``cut" any component in any position and reduce any cobordism to a linear combination of cobordisms where each component has exactly one outer boundary point. This is exactly the decomposition of the identity for $L_I$, and the circular language $L_{\circ}$ is compatible with this decomposition and determined by it, as in Figure~\ref{figure04-001}. For a sequence $\varepsilon$ of length $n$ this gives $2^n$ vectors in the state space $A(\varepsilon)$: an oriented segment for each sign in the sequence $\varepsilon$, with at most one dot on it. The pairing between $A(\varepsilon)$ and $A(\varepsilon^{\ast})$ is perfect on these sets of vectors (given by the permutation matrix, with a single 1 in each row and column).

Thus, $A(\varepsilon)$ is a free $\Bool$-semimodule of rank $2^{|\varepsilon|}$ with that basis, for each sequence $\varepsilon$. The category $\CCC_{\alpha}$ for this pair $(L_I,L_{\circ})$ has two relations shown in Figure~\ref{LILo-asquared-03} left and center, and, together with evaluation of floating intervals with at most one dot, that's a complete set of defining relations in  monoidal category $\CCC_{\alpha}$. Evaluation of circles follows from the evaluation of intervals and Figure~\ref{LILo-asquared-03} relations. The category $\CCC_{\alpha}$ determines a $\Bool$-valued TQFT with $A(+),A(-)$ free rank two $\Bool$-semimodules. 
\end{example}

\begin{example}
\label{ex:LI-LO-a-asquared}
Let $\Sigma=\{ a \}$ and $L=(L_I,L_{\circ})$, with $L_I=(a^2)^*$ and $L_{\circ}=a(a^2)^*$. Again, adding or removing two dots on a strand does not change the evaluation, implying relation (3) in Figure~\ref{LI-LO-a-asquared-02}.  The spanning vectors for $A(+)$, $A(-)$, and $A(+-)$ are given in Figure~\ref{LILo-asquared-01}. They are the same as the spanning sets for the previous  Example~\ref{ex:LI-LO-asquared}. State spaces $A(+)$ and $A(-)$ are free $\Bool$-semimodules of rank two each, and $A(+^n),A(-^n)$ are free semimodules, each of rank $2^n$.
Indeed, these spaces depend only on the interval language $L_I$, which is the same as in the previous example. 

\input{LI-LO-a-asquared-01}

\input{LI-LO-a-asquared-02}

Matrix of the pairing between $A(+-)$ and itself is given in Figure~\ref{LI-LO-a-asquared-01}. 
Elements $x_1,\ldots, x_6$ are irreducible  in $A(+-)$, since no column is a $\Bool$-semilinear combination of other columns. 
However, these elements satisfy nontrivial linear relations, with the minimal ones written below. 
\begin{eqnarray}
   x_1+ x_5 & = & x_4 + x_5, \label{eq_ex2_2} \\
   x_2 + x_6 & = & x_3 + x_6,
   \label{eq_ex2_3} \\
   x_1+x_2+x_3+x_4 & = & x_5 + x_6, \label{eq_ex2_1} \\
   x_1+x_2+x_3+x_5 & = & x_5+x_6. \label{eq_ex2_4}
\end{eqnarray}
Relations (\ref{eq_ex2_1}),(\ref{eq_ex2_2}) and (\ref{eq_ex2_4}) are shown in Figure~\ref{LI-LO-a-asquared-02} as relations (I), (II), and (III), respectively.
These relations can be found directly. It's also possible to find and visualize them by assigning a hypergraph to the Boolean matrix in Figure~\ref{LI-LO-a-asquared-01}. Vertices of the hypergraph are labelled by rows 1-6 of the matrix, see Figure~\ref{LI-LO-a-asquared-05}, and to each column one assigns a (generalized) edge.  If a column contains two ones, draw an edge between corresponding rows (edges labelled $x_1,x_2,x_3,x_4$). If a column contains three or more ones, draw a generalized edge (hyperedge), that is, a subset of the set of vertices. For the present matrix, the hyperedges are labelled by $x_5$ and $x_6$. In a square Boolean matrix, there are nontrivial relations on columns as long as the matrix is not a permutation matrix. Whenever two subsets of hyperedges cover the same set of vertices, corresponding sums of $x$'s are equal. For the present example,  $x_1+x_5=x_4+x_5$, since the two sums contain the same union of  vertices $\{2,4,5,6\}$. The remaining three relations among  \eqref{eq_ex2_1}-\eqref{eq_ex2_4} follow likewise. 

\input{LI-LO-a-asquared-05}
 
With six irreducible elements and non-trivial relations on them,  $A(+-)$ is not a free $\Bool$-semimodule. Bringing monoidal structure into play, relations \eqref{eq_ex2_2} and \eqref{eq_ex2_3} are equivalent modulo relation (IV) in Figure~\ref{LI-LO-a-asquared-02}, by adding a dot at an endpoint to all terms in the relation, see Figure~\ref{LI-LO-a-asquared-04}. 

We don't know a full set of defining relations in  $\CCC_{\alpha}$ for this evaluation $\alpha$. The set of eight relations:  four relations (I)-(IV), in Figure~\ref{LI-LO-a-asquared-02} and the four evaluation via $\alpha$ of an interval and a circle with at most one dot allow to evaluate all floating components and imply all relations in $A(+-)$ and in  $A(+^n),A(-^n)$ for $n\ge 1$.  

\input{LI-LO-a-asquared-04}
\end{example}

\begin{example}
\label{ex:LI-aa-Lo-all}
Let $\Sigma=\{a\}$ and $L=(L_I,L_{\circ})$ with $L_I=(a^2)^{\ast}=\emptyset\sqcup a^2\sqcup a^4\sqcup \ldots$ and $L_{\circ}=a^{\ast}=\Sigma^{\ast}$. In this example, $\Sigma$ has a single letter and $L_{\circ}$ contains all possible words. Minimal deterministic automata for the two languages are shown in Figure~\ref{circ-automata-01}. 

\vspace{0.05in} 

The interval language $L_I$ in Examples~\ref{ex:LI-LO-asquared}-\ref{ex:LI-aa-Lo-all} is the same, and the state spaces $A(-)$ and $A(+)$ in these three examples are isomorphic (as well as state spaces $A(+^n)$ and $A(-^n)$, for all $n$, which are free $\Bool$-semimodules of rank $2^n$.)  
  
\input{circ-automata-01}

The category $\CCC_{\alpha}$ has a monoidal relation that two dots on a strand can be removed, see Figure~\ref{circ-automata-02} left. 

\input{circ-automata-02}

The matrix of the pairing between $A(+)$ and $A(-)$ is given in Figure~\ref{circ-automata-03}, 
showing that $A(+)$ and $A(-)$ are each a free $\Bool$-semimodule on those two generators. Likewise, $A(+^n)$ and $A(-^n)$ are free $\Bool$-semimodules on $2^n$ generators each.


\input{circ-automata-04}

Spanning set of elements for $A(+-)$ is given in Figure~\ref{circ-automata-05}, and the spanning set of elements for $A(-+)$ is given in Figure~\ref{circ-automata-06}.  

\input{circ-automata-05}

\input{circ-automata-06}

\vspace{0.1in} 

Consider now the bilinear form on $A(+-)$. Relation in Figure~\ref{circ-automata-02} on the left allows to reduce $A(+-)$ to six generators. Among these are four generators that are a pair of intervals with $0$ or $1$ dots on each and two generators given by an arc with at most one dot (see Figures~\ref{circ-automata-05} and \ref{circ-automata-06}). The form is written in Figure~\ref{circ-automata-07}, and the matrix is symmetric. 

\vspace{0.05in} 

\input{circ-automata-07}

Denote these generators, corresponding to column vectors, by $x_1,\ldots, x_6$. We see that the following additive relations on columns hold 
\begin{eqnarray}
  x_5 & = & x_5 + x_4 , \label{eq_A_1}  \\
  x_5 & = & x_5 + x_1, \label{eq_A_2}\\
  x_6 & = & x_6 + x_2, \label{eq_A_3}\\
  x_6 & = & x_6 + x_3, \label{eq_A_4}\\
  x_5+x_6 & = & x_1 + x_2 + x_3 + x_4 , \label{eq_A_5}\\
  x_5+x_3+x_2 & = & x_6 + x_4 + x_1. \label{eq_A_6}
\end{eqnarray}
Each of these relations in $A(+-)$ corresponds to a skein (or monoidal) relation in $\CCC_{\alpha}$. Within the monoidal structure, relations \eqref{eq_A_1}-\eqref{eq_A_4} are equivalent and obtained from each other by adding dots near the endpoints of strands. Relation \eqref{eq_A_1} is shown in Figure~\ref{circ-automata-08}. Denoting the sum $x_5+x_3+x_2$ by a box labelled $y_1$ on a strand, relation \eqref{eq_A_6} can be written as in Figure~\ref{circ-automata-09}, saying that a dot next to box $y_1$ equals $y_1$. Relation \eqref{eq_A_5} is shown in Figure~\ref{circ-automata-10}. 

\input{circ-automata-08}

\input{circ-automata-09}

\input{circ-automata-10}

So far we have four independent, monoidal relations in $\CCC_{\alpha}$: the relation in Figure~\ref{circ-automata-02} on the left and relations \eqref{eq_A_1}, \eqref{eq_A_6}, \eqref{eq_A_5}, in addition to the four evaluation formulas for an interval and circle with at most one dot. We do not know a complete set of relations for $\CCC_{\alpha}$. 

\end{example}

\begin{example}
\label{ex:2nd-last-b}

Consider the interval language $L_I=(a+b)^{\ast}b(a+b)$ accepting those words whose second from last letter is $b$. For a spannning set in  $A(-)$  we can take 
$\{\langle \emptyset |$, $\langle a |$, $\langle b|$, $\langle aa|$, $\langle ba|$, $\langle ab|$, $\langle bb|\}$, see Figure~\ref{ex-2nd-last-basis}. Recall that we denote a word $w$ on a  strand with $-$ boundary (and the second boundary point an ``out" arrow) by $\langle w |$ and view it as an element of $A(-)$. 

\input{ex-2nd-last-basis}

 Similarly, we take $\{ |\emptyset \langle$,  $|a \langle$, $|b \langle$, $|aa \langle$, $|ab\langle$ , $|ba \langle$, $|bb \langle\}$ as 
 a spanning set of $A(+)$, viewed as a $\Bool$-semimodule, see Figure~\ref{ex-basis-up}.
 Here a word $w$ on a strand with $+$ boundary and the other boundary point an ``in" arrow is denoted $|w\langle$ a viewed as an element of $A(+)$. 

\input{ex-basis-up}

That these are spanning sets for $A(-)$ and $A(+)$, respectively, we leave as a straightforward exercise for the reader. That the same sets of words are chosen for generators of $A(+)$ and $A(-)$  is due to the relative simplicity of $L_I$ and us not trying to pick a minimal spanning set for each of these two  $\Bool$-semimodules, see Figure~\ref{ex-2nd-last}. 


\input{ex-2nd-last}

\input{ex-2nd-last-03}

The state space $A(-)$ and the unique minimal DFA $Q_-$ for $L_I$ are shown in Figure~\ref{ex-2nd-last-03}. We have
\begin{equation}
A(-) = \Bool x \oplus \Bool y \oplus \Bool z/(y=y+x, \hspace{2mm} z=z+x),
\end{equation}
which consists of $5$ elements $\{ 0,x,y,z,y+z\}$, with $x,y,z$ irreducible. Element $0$ is the only element of $A(-)$ not in $Q_-$. 

\input{ex-2nd-last-02}

The minimal free cover is $\Bool^3\lra A(-)$, mapping generators to $x,y,z$. There is more than one way to lift the action of $\Sigma^{\ast}$ from $A(-)$ to $\Bool^3$, and each lifting gives rise to a different minimal nondeterministic FA for $L_I$. Two of these minimal NFA are shown in Figure~\ref{ex-2nd-last-02}. Since $b(y)=y+z=x+y+z$ in $A(-)$, upon the lifting we have two choices for the corresponding action of $b$ on that generator of the free module. The second automaton has additional $b$ arrow from $y$ to $x$ to show that difference. Since $b(x)=y=x+y$, that also gives two choices for lifting the action of $b$ on $x$ to $\Bool^3$, with the  difference also shown in the two automata. It's straightforward to classify all possible liftings of $\Sigma^{\ast}$ action to $\Bool^3$ and the corresponding minimal DFA. Since $q_{\init}=x$ and $x$ lifts to only one element of $\Bool^3$, there are no additional choices for the initial subset $Q_{\init}$ of the minimal NFA. 

\vspace{0.05in} 

The language $L_I^{\op}=(a+b)b(a+b)^\ast$ consists of words
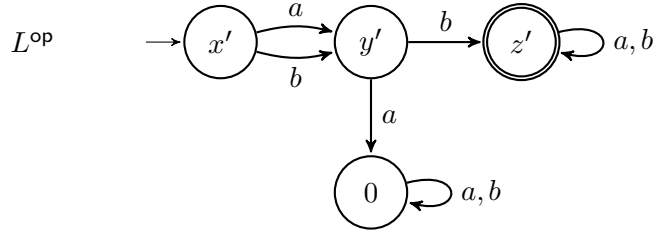
such that the second letter is $b$. The minimal  deterministic finite-state automaton (naturally denoted $Q_+$) for $L_I^{\op}$ is given in Figure~\ref{ex-2nd-last-04}. The corresponding 
$\Bool$-semimodule $A(+)$ is described by   
\begin{equation}
A(+) = \Bool x' \oplus
\Bool y' \oplus
\Bool z'/(z'= z'+x', \hspace{2mm} z'= z'+y')
\end{equation}
and consists of $5$ elements $\{ 0,x',y',z',x'+y'\}$.
Note that $x'+y'$ is not in $Q_+$. Element $0$ is in $Q_+$, since language $L_I^{\op}$ has \emph{unrecoverable} words, that is, words $\omega$ such that $\omega\omega'$ is not in the language for any word $\omega'$.  

\input{ex-2nd-last-04}

From Figure~\ref{ex-2nd-last}, we see that the sum of $z'$ with anything is $z'$, giving relations 
$z'= x'+z'$ and $z'= y'+z'$. 
We also see that $x'+y'\not=x',y'$, which implies $x'+y'$ is an element of $A(+)$ not in $Q_+$ (set of states of the minimal automaton for $L_{\op}$). Note that,  
although $|A(-)|=|A(+)|$ and $\alpha:A(-)\times A(+)\rightarrow \Bool$ is a nondegenerate pairing,  
$A(-)$ and $A(+)$ are not isomorphic as $\Bool$-semimodules. 

In this example, we restricted to $L_I$, state spaces $A(-),A(+)$ and associated automata and did not consider any circular language and the corresponding category $\CCC_{\alpha}$. 

\end{example}

\begin{example}
\label{ex:circular-aa} {\it A circular automaton.}
For the alphabet  $\Sigma=\{ a,b\}$ consider 
the circular language 
of all circular words that have two letters $a$ next to each other,  
\[
L_{\circ}=\{ \cdots aa\cdots\} =(a+b)^{\ast}aa(a+b)^{\ast}+a(a+b)^{\ast}a.
\]

\input{circ-aa-01}

Circular automaton for $L_{\circ}$ is given in Figure~\ref{circ-aa-01}. Circular automaton for the complementary language $L_{\circ}^{\comp}$ which consists of all circular words without $aa$ as a subword is shown in Figure~\ref{circ-aa-02}.


\input{circ-aa-02}

Figure~\ref{circ-aa-00} shows how a typical circular automaton on a 2-letter alphabet $\{a,b\}$ looks near the initial state $q_{\init}$. 

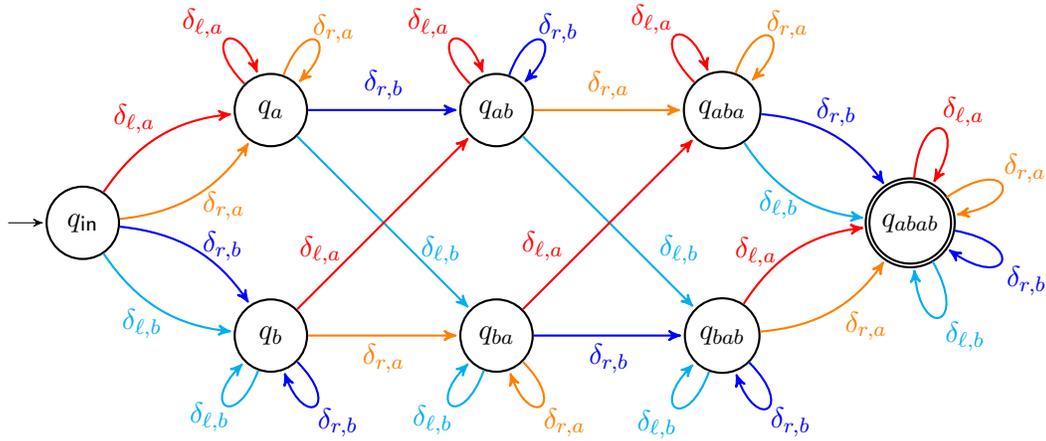
\input{circ-aa-00}

\end{example}

\begin{example}
\label{ex:Lcirc-abab}
Let $\Sigma=\{ a,b\}$ and circular language $L_{\circ}$ consist of words of the form $\cdots a \cdots b \cdots a \cdots b\cdots $. 
The corresponding minimum circular automaton is given in Figure~\ref{circ-abab-01}. 

\input{circ-abab-01}
 
\end{example}

Table~\ref{table:ex-summary} lists the languages for each example in this section.

\begin{table}
    \centering
    \begin{tabular}{||c|c|c||}
    \hline 
    \hline 
    Example & $L_I$ & $L_{\circ}$ \\
    \hline 
    \hline 
    Example~\ref{ex:LI-LO-asquared} & $(a^2)^*$ & $(a^2)^*$ \\  
    Example~\ref{ex:LI-LO-a-asquared} & $(a^2)^*$ & $a(a^2)^*$ \\ 
    Example~\ref{ex:LI-aa-Lo-all} & $(a^2)^*$ & $a^{\ast}$ \\
    Example~\ref{ex:2nd-last-b} & second from last letter is $b$ & -- \\    
    Example~\ref{ex:circular-aa}     & --  & has two consecutive $a$ \\         
    Example~\ref{ex:Lcirc-abab} & -- & has $\cdots a \cdots b \cdots a \cdots b\cdots $ \\ 
    \hline 
    \hline 
    \end{tabular}
    \caption{Languages for examples in Section~\ref{section:examples}. }
    \label{table:ex-summary}
\end{table}





\bibliographystyle{abbrv} 

\bibliography{top-automata}

\end{document}